\documentclass[twoside,11pt]{article}

\usepackage[english]{babel}
\usepackage{amsmath}
\usepackage{amsthm}
\usepackage{amsfonts}
\usepackage{amssymb}
\usepackage{graphicx}
\usepackage{colortbl,dcolumn}
\usepackage{paralist}  

\allowdisplaybreaks[4]
       \newtheorem{lemma}{\bf Lemma}[section]
       \newtheorem{theorem}{\bf Theorem}[section]
       \newtheorem{proposition}{\bf Proposition}[section]

       \newtheorem{remark}{\bf Remark}[section]

       \numberwithin{equation}{section}

\newcommand{\R}{\mathbb{R}}

\newcommand{\be}{\begin{equation}}
\newcommand{\ee}{\end{equation}}

\begin{document}
\title{{\Large Nonlinear stability of composite waves of traveling wave and rarefaction wave for a parabolic-hyperbolic system arising from chemotaxis}
\footnotetext{\small *Corresponding author.}
\footnotetext{\small E-mail addresses: liust533@nenu.edu.cn (S. Liu), lijy645@nenu.edu.cn (J. Li)}
}
\author{{Sitong Liu$^{1}$, Jingyu Li$^{2, *}$}\\[2mm]
\small\it $^1$School of Mathematics and Statistics, Jiangxi Normal University,\\
		\small\it    Nanchang, 330022, P.R.China \\
		\small\it $^2$School of Mathematics and Statistics, Northeast Normal University,\\
\small\it   Changchun 130024, PR China}
\date{ }
\maketitle

\begin{quote}
\small
\emergencystretch=2em
\textbf{Abstract:}
We study the nonlinear stability of a composite wave consisting of a traveling wave and a
rarefaction wave for a
parabolic-hyperbolic system arising from chemotaxis. We prove that if the initial value is a small $H^1$-type perturbation of  composite wave, then the system admits a global solution that converges
toward the composite wave with an absolutely continuous shift. The proof combines the weighted
relative-entropy mechanism for viscous shocks with the energy structure of
rarefaction waves. A key ingredient is the inclusion of a rarefaction modulation factor
in the weighted relative entropy; its derivatives combine with the
terms generated by the rarefaction profile to produce a rarefaction dissipation.
Moreover, the spatial separation of the two waves yields time-integrable
interaction errors caused by the non-exact superposition. Since diffusion acts
only on the density component, the full $H^1$ estimate is closed by exploiting
the coupling structure of the system to recover the missing dissipation of the
hyperbolic component.


\indent \textbf{AMS(2020) Subject Classification}: 35B35, 35L65, 35Q92, 92C17

\end{quote}

\section{Introduction}\label{introduction}
In this paper, we study the large-time behavior of solutions to the Cauchy problem for the
one-dimensional parabolic-hyperbolic system
\begin{equation}\label{1.1}
\begin{cases}
n_t-(nq)_x=n_{xx},& x\in\mathbb R,\ t>0,\\
q_t-n_x=0,& x\in\mathbb R,\ t>0,
\end{cases}
\end{equation}
subject to
\begin{equation}\label{i d}
(n,q)(x,0)=(n_0,q_0)(x)\longrightarrow(n_\pm,q_\pm)
\qquad\mbox{as }x\to\pm\infty.
\end{equation}
The system \eqref{1.1} is derived from the PDE--ODE hybrid chemotaxis model with
logarithmic sensitivity
\begin{equation}\label{chemotaxis model}
\begin{cases}
n_t=n_{xx}-\big(n(\ln c)_x\big)_x,\\
c_t=-nc+\beta c,
\end{cases}
\end{equation}
which was introduced by Levine et al. \cite{H.A. Levine 00} in the modeling of the
interaction between vascular endothelial cells and the vascular endothelial growth factor
during the initiation of tumor angiogenesis.  Here $n>0$ denotes the cell density and
$c>0$ the concentration of chemical signal.  The logarithmic sensitivity reflects
the Weber--Fechner law and, at the same time, produces a singularity at $c=0$; see, for
example, \cite{W. Alt 87,F.W. Dahlquist 72,E.F. Keller 71}.  The Hopf--Cole transformation
introduced by Levine and Sleeman \cite{H.A. Levine 97},
\[
q:=-(\ln c)_x=-\frac{c_x}{c},
\]
removes this singularity and transforms \eqref{chemotaxis model} into \eqref{1.1}.

The Cauchy problem for \eqref{1.1} has been studied extensively around constant
background states.  Guo et al. \cite{GXZZ} established global strong solutions for large
initial data.  Zhang et al. \cite{Y. Zhang 13} constructed global classical solutions near
positive constants and derived algebraic decay, while Li et al. \cite{D.Li 15} developed
entropy estimates yielding global well-posedness for large classical solutions.  Related
multidimensional well-posedness results can be found in
\cite{D.Li 15,DL,Hao}.  When the end states in \eqref{i d} are different, however, the
large-time dynamics are governed by nonconstant wave patterns associated with the
Riemann problem for the inviscid system
\begin{equation}\label{conservation law}
\begin{cases}
n_t-(nq)_x=0,\\
q_t-n_x=0,
\end{cases}
\end{equation}
with
\begin{equation}\label{n0,q0}
(n,q)(x,0)=
\begin{cases}
(n_-,q_-),&x<0,\\
(n_+,q_+),&x>0.
\end{cases}
\end{equation}
Writing \eqref{conservation law} as
\begin{equation}\label{equivalent system}
\binom{n}{q}_t+
\begin{pmatrix}
-q&-n\\
-1&0
\end{pmatrix}
\binom{n}{q}_x=0,
\end{equation}
we obtain the characteristic speeds
\begin{equation}\label{1.8}
\lambda_1(n,q)
=
\frac{-q-\sqrt{q^2+4n}}2
<
\frac{-q+\sqrt{q^2+4n}}2
=
\lambda_2(n,q),
\end{equation}
and the corresponding right eigenvectors
\[
r_k(n,q)=(-\lambda_k(n,q),1)^T,\qquad k=1,2.
\]
Moreover,
\begin{equation}\label{genuinely nonlinear}
\begin{aligned}
\nabla\lambda_1(n,q)\cdot r_1(n,q)
&=-1-\frac{q}{\sqrt{q^2+4n}}\neq0,\\
\nabla\lambda_2(n,q)\cdot r_2(n,q)
&=-1+\frac{q}{\sqrt{q^2+4n}}\neq0, \text{ if } n>0.
\end{aligned}
\end{equation}
Thus, in the region $n>0$, the inviscid system is strictly hyperbolic
and both characteristic fields are genuinely nonlinear.

The Riemann structure of \eqref{conservation law} contains shocks, rarefaction waves and
their superpositions.  Wang and Hillen \cite{Wang08} constructed shock waves for the
inviscid system and the corresponding traveling waves of \eqref{1.1}.  The
one-dimensional stability theory for such traveling waves under zero-mass perturbation was subsequently developed in
\cite{T.Li 09,Jin 13,20JDE}; Li, Wang and Zhang \cite{J.Li 13} also treated a composite wave of
two traveling waves.  An important advance was made by Choi, Kang, Kwon and Vasseur
\cite{K. Choi 20 M3}, who introduced a weighted relative entropy and a time-dependent
shift and proved an $a$-contraction property for a small traveling wave, allowing
arbitrarily large perturbations at the level of the relative entropy.  This contraction
estimate was then used by Choi, Kang and Vasseur \cite{Choi 20} to obtain global
well-posedness for large $H^1$ perturbations of a small traveling wave.

Two more recent developments are especially relevant to the present work.  First, Liu,
Li and Xu \cite{Liu} established nonlinear stability of planar traveling waves for the
three-dimensional version of \eqref{1.1} under general perturbations.  Their argument
combines the relative-entropy method with a time-dependent shift in the basic estimate
and exploits intrinsic cancellations of the chemotaxis system in the higher-order
estimates, thereby removing the zero-mass restriction that is inherent in the classical
anti-derivative framework.  Second, Liu and Li \cite{LiuRare2026} established the
asymptotic stability of strong rarefaction waves for \eqref{1.1}, including both a single
rarefaction wave and a superposition of two rarefaction waves.  In that setting the wave
strength is not required to be small.  The key ingredients are the monotonicity of the
smooth rarefaction approximation in both space and time and the use of the coupling
structure of \eqref{1.1} to recover the missing dissipation of the hyperbolic component.

These two mechanisms do not simply superpose when a rarefaction wave and a viscous shock
are present simultaneously.  This is the main issue addressed in the present paper.  We
consider a 1-rarefaction wave connecting $(n_-,q_-)$ to an intermediate state
$(n_m,q_m)$ and a 2-viscous shock wave (i.e. traveling wave) connecting $(n_m,q_m)$ to $(n_+,q_+)$.  We prove that,
when the two wave strengths and the initial perturbation are small, the
resulting composite wave of traveling wave and rarefaction wave is asymptotically stable under general
$H^1$-type perturbations.

To formulate the result, define the $k$-rarefaction curve through a state
$(n_-,q_-)$ by
\begin{equation}\label{Rk}
R_k(n_-,q_-)
=
\left\{
(n,q)\in\mathbb R^2:
h_k(n,q)=h_k(n_-,q_-),\
\lambda_k(n,q)\ge\lambda_k(n_-,q_-)
\right\},
\end{equation}
where $h_k$ is a $k$-Riemann invariant.  Following
\cite{F.He 24,T.Li 22}, we use
\begin{equation}\label{RI}
\begin{aligned}
h_1(n,q)
&=
(\sqrt{q^2+4n}+q)(\sqrt{q^2+4n}-2q)^2,\\
h_2(n,q)
&=
(\sqrt{q^2+4n}-q)(\sqrt{q^2+4n}+2q)^2,
\end{aligned}
\end{equation}
which satisfy
\begin{equation}\label{RId}
\nabla h_k\cdot r_k=0,\qquad k=1,2.
\end{equation}
For the shock family, we write $S_2(n_R,q_R)$ for the local backward
2-shock curve consisting of left states $(n_L,q_L)$ for which there exists a speed
$\sigma>0$ satisfying the Rankine--Hugoniot relations and the Lax condition.  The
associated viscous profile solves
\begin{equation}\label{2-viscous shock}
\begin{cases}
-\sigma(\tilde n^S)'-(\tilde n^S\tilde q^S)'=(\tilde n^S)'',\\
-\sigma(\tilde q^S)'-(\tilde n^S)'=0,\\
(\tilde n^S,\tilde q^S)(-\infty)=(n_L,q_L),\\
(\tilde n^S,\tilde q^S)(+\infty)=(n_R,q_R).
\end{cases}
\end{equation}

Our main theorem is the following.

\begin{theorem}\label{composite wave theorem}
For a given constant state $(n_+,q_+)\in\mathbb R_+\times\mathbb R$, there exist
constants $\delta_0,\delta_1>0$ such that the following holds.  Let
\[
(n_m,q_m)\in S_2(n_+,q_+),
\qquad
(n_m,q_m)\in R_1(n_-,q_-),
\]
and assume
\[
|n_+-n_m|+|n_m-n_-|\le\delta_0.
\]
Let $(n^r,q^r)(x/t)$ be the 1-rarefaction wave of \eqref{conservation law}
connecting $(n_-,q_-)$ to $(n_m,q_m)$, and let
$(\tilde n^S,\tilde q^S)(x-\sigma t)$ be the 2-viscous shock profile
connecting $(n_m,q_m)$ to $(n_+,q_+)$.  Assume that the initial data satisfy
\begin{equation}\label{initial data}
\sum_{\pm}
\|(n_0-n_\pm,q_0-q_\pm)\|_{L^2(\mathbb R_\pm)}
+
\|(n_{0x},q_{0x})\|_{L^2(\mathbb R)}
\leq\delta_1,
\end{equation}
where $\mathbb R_-:=(-\infty,0)$ and $\mathbb R_+:=(0,\infty)$.
Then \eqref{1.1} admits a unique global solution $(n,q)$.  Moreover, there exists an
absolutely continuous shift $\mathbf X(t)$ such that
\begin{equation}\label{1D solution space}
\begin{aligned}
&n(t,x)-
\left(
n^r(x/t)
+\tilde n^S(x-\sigma t-\mathbf X(t))
-n_m
\right)
\in C((0,\infty);H^1(\mathbb R)),\\
&q(t,x)-
\left(
q^r(x/t)
+\tilde q^S(x-\sigma t-\mathbf X(t))
-q_m
\right)
\in C((0,\infty);H^1(\mathbb R)),\\
&n_{xx}(t,x)
-\tilde n^S_{xx}(x-\sigma t-\mathbf X(t))
\in L^2(0,\infty;L^2(\mathbb R)).
\end{aligned}
\end{equation}
In addition,
\begin{equation}\label{Limsup}
\begin{aligned}
\sup_{x\in\mathbb R}\Big|
(n,q)(t,x)
-&\big(
n^r(x/t)+\tilde n^S(x-\sigma t-\mathbf X(t))-n_m,\\
&\quad
q^r(x/t)+\tilde q^S(x-\sigma t-\mathbf X(t))-q_m
\big)
\Big|
\longrightarrow0 .
\end{aligned}
\end{equation}
as $t\to\infty$, and
\begin{equation}\label{LimX'}
\lim_{t\to\infty}\frac{\mathbf X(t)}{t}=0.
\end{equation}
\end{theorem}

We next explain the main novelties and the proof strategy.
The first difficulty is that the two components of the composite wave have very different
dissipative mechanisms.  The traveling wave is spatially localized and has a neutral
translation mode.  Its stability is therefore naturally measured by a weighted relative
entropy with a time-dependent shift, as in \cite{K. Choi 20 M3,Liu}.  In contrast, the
smooth rarefaction wave spreads in space and its favorable terms come from monotonicity of
the profile, as in \cite{LiuRare2026}.  If one applies the shock contraction estimate to
the composite profile without modification, the derivatives of the rarefaction profile
generate additional terms that are not controlled by the shock dissipation alone.

To overcome this obstruction, we introduce, in addition to the shock weight
$a^{-\mathbf X}$, a rarefaction modulation factor
\[
b(\xi,t):=\frac{N(\xi,t)}{n_m},
\qquad
N(\xi,t)=\tilde n^R(\xi+\sigma t,t),
\]
and use a modulated relative entropy of the form
\[
\int_{\mathbb R}
a^{-\mathbf X}b
\left[
\Pi(n|\tilde n)+\frac12|q-\tilde q|^2
\right]\,d\xi.
\]
The derivatives of $b$ do not constitute merely an error.  After being combined with the
rarefaction term already present in the relative entropy identity, they produce the
coercive rarefaction dissipation
\[
\mathcal G_R(t)
=
\frac1{n_m}
\int_{\mathbb R}a^{-\mathbf X}
\left[
N|Q_\xi|\Pi(n|\tilde n)
+
\frac{|N_t-\sigma N_\xi|}{2}|q-\tilde q|^2
\right]d\xi,
\]
where $Q(\xi,t)=\tilde q^R(\xi+\sigma t,t)$.  This is the mechanism that
allows the shock contraction and rarefaction monotonicity to coexist in a single
$L^2$ estimate.  At the same time, the shock part controls a characteristic combination
of the perturbations rather than the two components separately; preserving this structure
is essential for avoiding a loss of the factor $\lambda/\delta_S$.

The second difficulty is that the superposition of the rarefaction and the traveling wave is only an
approximate solution.  The perturbation equations contain source terms
\[
F_1=-N_{\xi\xi},
\qquad
F_2=
-\partial_\xi
\left[
(N-n_m)(T-q_m)+(S-n_m)(Q-q_m)
\right].
\]
The two waves propagate in opposite directions: the rarefaction wave has negative
characteristic speed, whereas the traveling wave to the right.  Their spatial separation
therefore yields exponentially decaying interaction terms.  We exploit this separation in
both $L^1$- and $L^2$-based estimates and show that the mixed source terms are integrable
in time.  These interaction estimates are absent in the traveling wave contraction theory of
\cite{K. Choi 20 M3} and in the rarefaction wave analysis of
\cite{LiuRare2026}.

The third difficulty comes from the degenerate dissipation of \eqref{1.1}: diffusion acts
only on $n$, while the $q$-equation is hyperbolic.  In the basic estimate, the derivative
$q_\xi-\tilde q_\xi$ is expressed through the parabolic component and its derivatives;
the time-derivative term produced by this substitution is treated by integration by parts
in time.  In the first-order estimate, the same coupling is used to recover
$\int_0^t\|\partial_\xi(q-\tilde q)\|^2\,d\tau$ from the strong dissipation of
the $n$-equation.  This part is inspired by the structural replacement used for strong
rarefaction waves in \cite{LiuRare2026}, but here it must be compatible with the
traveling wave weight, the time-dependent shift and the composite-wave source terms.

These ingredients yield a closed a priori estimate consisting of the $H^1$ norm of the
perturbation, the parabolic dissipation, the shift dissipation, the traveling wave weighted
dissipation and the new rarefaction dissipation $\mathcal G_R$.  A standard continuation
argument then gives global existence.  Finally, the time integrability of the first and
second derivatives implies decay of the perturbation in $L^\infty$, while the shift
equation gives $\dot{\mathbf X}(t)\to0$ and hence
$\mathbf X(t)/t\to0$.

It is particularly useful to compare our argument with the composite-wave result of
Kang, Vasseur and Wang \cite{KangVasseurWang23} for the one-dimensional barotropic
Navier--Stokes equations.  They proved the time-asymptotic stability of a composite wave
consisting of a 1-rarefaction wave and a 2-viscous shock.  Their proof replaces the
classical anti-derivative approach for shocks by the $a$-contraction with shifts in the
BD effective variables
\[
h=u-(\ln v)_x.
\]
The weight and the shift are attached to the shock component, while the errors caused by
the non-exact superposition are controlled by spatial separation of the two waves.  In
their weighted relative-entropy identity, the rarefaction contribution itself has the
favorable sign: the good part contains
\[
\int_{\mathbb R} a\,u^R_\xi\,p(v|\tilde v)\,d\xi ,
\]
which is nonnegative because $a>0$, $u^R_\xi>0$, and the relative pressure
$p(v|\tilde v)$ is nonnegative.  Hence this term enters the energy identity with a
dissipative sign.

This point is substantially different for the chemotaxis system considered here.  In the
unmodulated weighted relative-entropy identity, the rarefaction derivative does not
produce a sign-definite coercive term.  In particular, it gives the mixed contribution
\[
\mathbf B_9(t)
=
\int_{\mathbb R}
a^{-\mathbf X}
\frac{N_\xi}{\tilde n}
(n-\tilde n)(q-\tilde q)\,d\xi ,
\]
whose sign is indefinite.  Thus, in contrast with the barotropic Navier--Stokes case, the
rarefaction contribution is initially a bad term rather than part of the good
dissipation.  A main new ingredient of the present paper is to introduce the additional
rarefaction modulation
\[
b(\xi,t)=\frac{N(\xi,t)}{n_m}
\]
in the weighted relative entropy.  The derivatives $b_t$ and $b_\xi$ then combine with
$\mathbf B_9^{(b)}$ so that
\[
\mathbf B_9^{(b)}+\mathcal Q_b
=
-\mathcal G_R+\sum_{j=1}^4\mathcal R_{R,j},
\]
where $\mathcal G_R$ is positive and the four remainders are either absorbable or
integrable in time.  This rarefaction conversion mechanism is isolated in
Lemma~\ref{rarefaction-modulation-L34}.  It is the key point that allows the
shock-contraction structure and the rarefaction dynamics to be combined for
\eqref{1.1}.

There is also a structural difference in the recovery of the missing derivative
dissipation.  The barotropic Navier--Stokes argument uses the BD effective variable to
obtain the basic estimate and then returns to the classical variables.  Here the
hyperbolic component $q$ has no direct diffusion, and we instead exploit the exact
coupling in \eqref{1.1} to express the spatial derivative of the $q$-perturbation through
the parabolic component; the resulting time derivative is handled by integration by
parts in time.  Thus the present proof follows the same broad
$a$-contraction-with-shift philosophy, but both the rarefaction coercivity and the
recovery of the hyperbolic dissipation require mechanisms specific to the chemotaxis
system.

The rest of the paper is organized as follows.  In Section~\ref{preliminary}, we introduce
the relative quantities, construct the smooth approximate 1-rarefaction wave, recall the
2-viscous shock profile, define the weight and the shift, and derive the perturbation
system for the shifted composite profile.  Section~\ref{Sect.3} is devoted to the
a priori estimates and the proof of Theorem~\ref{composite wave theorem}.  The basic
estimate combines the weighted relative entropy for the shock with the rarefaction
modulation described above; the higher-order estimates then recover the full $H^1$
dissipation and complete the nonlinear stability argument.

\section{Preliminaries}\label{preliminary}
In this section, we collect the relative-entropy quantities used below,
construct a smooth approximation of the 1-rarefaction wave, recall the
2-viscous shock profile, and define the time-dependent shift
$\mathbf X(t)$ and the corresponding composite profile.

\subsection{Relative quantities}
For a smooth scalar or vector-valued function $f$, its relative function
with respect to two states $u$ and $v$ is defined by
\begin{align}\label{fu,v}
f(u|v):=f(u)-f(v)-\nabla f(v)\cdot(u-v).
\end{align}
For
\[
\Pi(n):=n\ln n-n,\qquad n>0,
\]
we write
\begin{equation}\label{4.Pin1n2}
\Pi(n_1|n_2)
=
\Pi(n_1)-\Pi(n_2)-\Pi'(n_2)(n_1-n_2).
\end{equation}

\begin{lemma}\label{4.local inequality}
Fix $n_->0$.  There exist constants $\delta_*>0$ and $C>0$ such that
if
\[
0<\delta<\delta_*,
\qquad
\left|\frac{n_1}{n_2}-1\right|<\delta,
\qquad
\frac{n_-}{2}<n_2<2n_-,
\]
then
\begin{equation}\label{PI>}
\Pi(n_1|n_2)
\ge
\frac{n_2}{2}
\left[
\left(\frac{n_1}{n_2}-1\right)^2
-\frac13
\left(\frac{n_1}{n_2}-1\right)^3
\right],
\end{equation}
and
\begin{equation}\label{PI<}
\Pi(n_1|n_2)
\le
\frac{n_2}{2}
\left[
\left(\frac{n_1}{n_2}-1\right)^2
-\frac13
\left(\frac{n_1}{n_2}-1\right)^3
\right]
+
C\delta
\left|\frac{n_1}{n_2}-1\right|^3.
\end{equation}
\end{lemma}

\begin{proof}
Set
\[
w:=\frac{n_1}{n_2}-1.
\]
Then
\[
\Pi(n_1|n_2)
=
n_2\big[(1+w)\ln(1+w)-w\big].
\]
For $|w|<\delta_*<1/2$, Taylor's formula gives
\[
(1+w)\ln(1+w)-w
=
\frac{w^2}{2}-\frac{w^3}{6}
+\frac{w^4}{12(1+\theta w)^3}
\]
for some $\theta\in(0,1)$.  The remainder is nonnegative and is bounded
above by $C|w|^4\le C\delta|w|^3$.  This proves
\eqref{PI>}--\eqref{PI<}.
\end{proof}

\subsection{Rarefaction wave}
We consider the 1-rarefaction wave connecting the left state
$(n_-,q_-)$ to the intermediate state $(n_m,q_m)$.  In accordance with
the definition \eqref{Rk}, throughout this paper we assume
\begin{equation}\label{R1-orientation-prelim}
(n_m,q_m)\in R_1(n_-,q_-),
\end{equation}
so that
\[
h_1(n_m,q_m)=h_1(n_-,q_-),
\qquad
\lambda_1(n_-,q_-)
<
\lambda_1(n_m,q_m).
\]
Let
\[
\lambda_-:=\lambda_1(n_-,q_-),
\qquad
\lambda_m:=\lambda_1(n_m,q_m).
\]
Since $n>0$, both $\lambda_-$ and $\lambda_m$ are negative.

Consider the Riemann problem for the inviscid Burgers equation
\begin{equation}\label{inviscid Burgers}
\begin{cases}
w^r_t+w^r w^r_x=0,\\
w^r(x,0)=w^r_0(x):=
\begin{cases}
\lambda_-,&x<0,\\
\lambda_m,&x>0.
\end{cases}
\end{cases}
\end{equation}
For $t>0$, its rarefaction fan is
\begin{equation}\label{wr'}
w^r(x,t)=
\begin{cases}
\lambda_-,&x\le\lambda_-t,\\
x/t,&\lambda_-t\le x\le\lambda_mt,\\
\lambda_m,&x\ge\lambda_mt.
\end{cases}
\end{equation}
The 1-rarefaction wave $(n^r,q^r)$ is defined by
\begin{equation}\label{nr,qr}
\lambda_1((n^r,q^r)(x,t))=w^r(x,t),
\qquad
h_1((n^r,q^r)(x,t))=h_1(n_-,q_-).
\end{equation}
Thus
\[
\lambda_-
\le
\lambda_1((n^r,q^r)(x,t))
\le
\lambda_m,
\qquad
(n^r,q^r)(x,t)\in R_1(n_-,q_-).
\]

Let
\[
\delta_R:=|n_m-n_-|.
\]
On the fixed compact set of states considered here,
\begin{equation}\label{deltaR-equivalence-prelim}
|q_m-q_-|
+
|\lambda_m-\lambda_-|
\sim\delta_R.
\end{equation}
To construct a smooth approximation, consider
\begin{equation}\label{smooth}
\begin{cases}
w_t+ww_x=0,\\
w(x,0)=w_0(x),
\end{cases}
\end{equation}
with
\begin{equation}\label{w0}
w_0(x)
:=
\frac{\lambda_m+\lambda_-}{2}
+
\frac{\lambda_m-\lambda_-}{2}\tanh x.
\end{equation}

By the characteristic method, the solution $w(x,t)$ of \eqref{smooth}
has the following standard properties.

\begin{lemma}[cf. Matsumura--Nishihara \cite{Matsumura 86}, Lemma 2.1]
\label{lem2.1ref}
The problem \eqref{smooth} has a unique smooth global solution satisfying:
\begin{enumerate}
\item[(1).]
\[
\lambda_-<w(x,t)<\lambda_m,\qquad
w_x(x,t)>0,
\qquad (x,t)\in\mathbb R\times[0,\infty).
\]

\item[(2).]
For every $1\le p\le\infty$,
\begin{equation}\label{wx,wxx}
\begin{aligned}
\|w_x(t)\|_{L^p}
&\le
C_p\min\{\delta_R,\delta_R^{1/p}t^{-1+1/p}\},\\
\|w_{xx}(t)\|_{L^p}
&\le
C_p\min\{\delta_R,t^{-1}\},
\qquad t>0.
\end{aligned}
\end{equation}

\item[(3).]
For $x\le\lambda_-t$,
\begin{equation}\label{w-w-}
|w(x,t)-\lambda_-|
+
|w_x(x,t)|
\le
C\delta_R
e^{-c|x-\lambda_-t|}.
\end{equation}

\item[(4).]
For $x\ge0$,
\begin{equation}\label{w-w+}
|w(x,t)-\lambda_m|
+
|w_x(x,t)|
\le
C\delta_R
e^{-c(x+|\lambda_m|t)}.
\end{equation}

\item[(5).]
\[
\lim_{t\to\infty}
\sup_{x\in\mathbb R}|w(x,t)-w^r(x,t)|=0.
\]
\end{enumerate}
\end{lemma}

\begin{proof}
The characteristic representation is
\[
w(x,t)=w_0(x_0),\qquad
x=x_0+w_0(x_0)t.
\]
If $x\le\lambda_-t$, then
$x_0\le x-\lambda_-t\le0$.  Since
$w_0(x_0)-\lambda_-\le C\delta_R e^{2x_0}$ and
$w'_0(x_0)\le C\delta_R e^{2x_0}$, \eqref{w-w-} follows.
If $x\ge0$, then
$x_0=x-w_0(x_0)t\ge x-\lambda_mt=x+|\lambda_m|t$.
Using the right tail of $\tanh x$ gives \eqref{w-w+}.
The remaining assertions are the standard characteristic estimates.
\end{proof}

By the implicit function theorem, the map
\[
(n,q)\longmapsto(\lambda_1(n,q),h_1(n,q))
\]
is locally invertible along the compact rarefaction segment because
$\nabla\lambda_1$ and $\nabla h_1$ are linearly independent.  We define
the smooth approximate rarefaction by
\begin{equation}\label{nR}
(\tilde n^R,\tilde q^R)(x,t)\in R_1(n_-,q_-),
\qquad
\lambda_1((\tilde n^R,\tilde q^R)(x,t))=w(x,t).
\end{equation}

\begin{lemma}\label{lem2.1}
The smooth profile $(\tilde n^R,\tilde q^R)$ defined by \eqref{nR}
has the following properties:
\begin{enumerate}
\item[(1).]
It satisfies the inviscid system
\begin{equation}\label{nRqR}
\begin{cases}
\tilde n^R_t-(\tilde n^R\tilde q^R)_x=0,\\
\tilde q^R_t-\tilde n^R_x=0.
\end{cases}
\end{equation}

\item[(2).]
\[
\lim_{t\to\infty}
\sup_{x\in\mathbb R}
|(\tilde n^R,\tilde q^R)(x,t)-(n^r,q^r)(x,t)|
=0.
\]

\item[(3).]
For every $0<t_0<T<\infty$,
\[
(\tilde n^R,\tilde q^R)(\cdot,t)
-(n^r,q^r)(\cdot,t)
\in C([t_0,T];H^1(\mathbb R)).
\]
\end{enumerate}
\end{lemma}

\begin{proof}
Let $\Phi_1$ denote the local inverse of
$(\lambda_1,h_1)$ on the rarefaction segment, with the second component
fixed at $h_1(n_-,q_-)$.  Then
\[
(\tilde n^R,\tilde q^R)=\Phi_1(w),
\qquad
\Phi_1'(w)
=
\frac{r_1}{\nabla\lambda_1\cdot r_1}.
\]
Therefore
\begin{equation}\label{2.13}
\binom{\tilde n^R}{\tilde q^R}_t
+
\begin{pmatrix}
-\tilde q^R&-\tilde n^R\\
-1&0
\end{pmatrix}
\binom{\tilde n^R}{\tilde q^R}_x
=
\frac{w_t+\lambda_1w_x}
{\nabla\lambda_1\cdot r_1}\,r_1
=0,
\end{equation}
which proves (1).  Assertion (2) follows from Lemma
\ref{lem2.1ref}-(5) and smooth invertibility of $\Phi_1$.
For each $t\ge t_0>0$, both $w-w^r$ and
$\partial_x(w-w^r)$ belong to $L^2(\mathbb R)$; this follows from the
finite rarefaction fan and the exponential tails
\eqref{w-w-}--\eqref{w-w+}.  Their dependence on $t$ is continuous on
$[t_0,T]$.  Smoothness of $\Phi_1$ then gives (3).
\end{proof}

\begin{lemma}\label{properties of nRqR}
Let
\[
\Lambda(x,t)
:=
\lambda_1(\tilde n^R(x,t),\tilde q^R(x,t)).
\]
The smooth approximate rarefaction satisfies:
\begin{enumerate}
\item[(1).]
\[
\tilde n^R_x<0,\qquad
\tilde q^R_x<0,\qquad
\tilde n^R_t<0,\qquad
\tilde q^R_t<0
\qquad (x\in\mathbb R,\ t>0).
\]

\item[(2).]
For every $1\le p\le\infty$,
\begin{equation}\label{nRx}
\begin{aligned}
\|(\tilde n^R_x,\tilde q^R_x)(t)\|_{L^p}
&\le
C_p\min\{\delta_R,
\delta_R^{1/p}t^{-1+1/p}\},\\
\|(\tilde n^R_{xx},\tilde q^R_{xx})(t)\|_{L^p}
&\le
C_p\min\{\delta_R,t^{-1}\},
\qquad t>0.
\end{aligned}
\end{equation}

\item[(3).]
For $x\le\lambda_-t$,
\begin{equation}\label{nR-n}
\begin{aligned}
&|(\tilde n^R,\tilde q^R)(x,t)-(n_-,q_-)|\\
&\qquad
+|(\tilde n^R_x,\tilde q^R_x)(x,t)|
\le
C\delta_R e^{-c|x-\lambda_-t|}.
\end{aligned}
\end{equation}

\item[(4).]
For $x\ge0$,
\begin{equation}\label{nR-nx>0}
\begin{aligned}
&|(\tilde n^R,\tilde q^R)(x,t)-(n_m,q_m)|\\
&\qquad
+|(\tilde n^R_x,\tilde q^R_x)(x,t)|
\le
C\delta_R e^{-c(x+|\lambda_m|t)}.
\end{aligned}
\end{equation}

\item[(5).]
The following identities are exact:
\begin{equation}\label{rare-exact-prelim}
\tilde n^R_x=-\Lambda\tilde q^R_x,
\qquad
\tilde q^R_t=\tilde n^R_x=-\Lambda\tilde q^R_x,
\qquad
\tilde n^R_t=\Lambda^2\tilde q^R_x.
\end{equation}
In particular,
\[
|\tilde n^R_t|
+|\tilde q^R_t|
\le
C|(\tilde n^R_x,\tilde q^R_x)|.
\]
\end{enumerate}
\end{lemma}

\begin{proof}
From \eqref{nR},
\[
\binom{\tilde n^R_x}{\tilde q^R_x}
=
\frac{w_x}{\nabla\lambda_1\cdot r_1}r_1.
\]
Since $w_x>0$, $\lambda_1<0$, and
$\nabla\lambda_1\cdot r_1<0$ on the compact rarefaction segment, both
components are negative and
\[
\tilde n^R_x=-\Lambda\tilde q^R_x.
\]
Also
\[
\binom{\tilde n^R_t}{\tilde q^R_t}
=
-\Lambda
\binom{\tilde n^R_x}{\tilde q^R_x},
\]
because $w_t=-ww_x$ and $w=\Lambda$.  This proves
\eqref{rare-exact-prelim} and the signs in (1).  The estimates in
(2)--(4) follow from Lemma \ref{lem2.1ref} and the smoothness of the
inverse map $\Phi_1$; for second derivatives one uses
$\Phi_1''(w)w_x^2+\Phi_1'(w)w_{xx}$.
\end{proof}

\subsection{Traveling wave}
We now consider the 2-viscous shock connecting the left state
$(n_m,q_m)$ to the right state $(n_+,q_+)$, with
\[
(n_m,q_m)\in S_2(n_+,q_+).
\]
The Rankine--Hugoniot conditions are
\begin{equation}\label{R-H condition'}
\begin{cases}
-\sigma(n_+-n_m)-(n_+q_+-n_mq_m)=0,\\
-\sigma(q_+-q_m)-(n_+-n_m)=0.
\end{cases}
\end{equation}
On this 2-shock branch,
\begin{equation}\label{lax}
n_m>n_+>0,\qquad
q_m<q_+,
\end{equation}
and the Lax inequalities are
\[
\lambda_2(n_+,q_+)<\sigma<\lambda_2(n_m,q_m).
\]
The corresponding inviscid shock is
\begin{equation}\label{nsqs}
(n^s,q^s)(x,t)=
\begin{cases}
(n_m,q_m),&x<\sigma t,\\
(n_+,q_+),&x>\sigma t.
\end{cases}
\end{equation}
Let
\[
\delta_S:=n_m-n_+>0.
\]
From \eqref{R-H condition'},
\[
q_+-q_m=\frac{\delta_S}{\sigma},
\qquad
\sigma^2+q_m\sigma-n_+=0,
\]
hence
\begin{equation}\label{Wave speed}
\sigma
=
\frac{-q_m+\sqrt{q_m^2+4n_+}}{2}>0.
\end{equation}

With $\xi=x-\sigma t$, the viscous shock profile satisfies
\begin{equation}\label{traveling wave}
\begin{cases}
-\sigma\tilde n^S_\xi
-(\tilde n^S\tilde q^S)_\xi
=\tilde n^S_{\xi\xi},\\
-\sigma\tilde q^S_\xi-\tilde n^S_\xi=0,\\
(\tilde n^S,\tilde q^S)(-\infty)=(n_m,q_m),\\
(\tilde n^S,\tilde q^S)(+\infty)=(n_+,q_+).
\end{cases}
\end{equation}
After integration and use of \eqref{R-H condition'}, one obtains
\begin{equation}\label{nS'}
\tilde n^S_\xi
=
\frac{(\tilde n^S-n_m)(\tilde n^S-n_+)}{\sigma},
\end{equation}
together with
\begin{equation}\label{shock-q-relation-prelim}
\tilde q^S-q_m
=
-\frac{\tilde n^S-n_m}{\sigma},
\qquad
\tilde q^S_\xi
=
-\frac1\sigma\tilde n^S_\xi.
\end{equation}
Fixing the translation by
\[
\tilde n^S(0)=\frac{n_m+n_+}{2},
\]
the profile is explicitly
\begin{equation}\label{nSqS}
\tilde n^S(\xi)
=
n_+
+
\frac{\delta_S}
{1+e^{\frac{\delta_S}{\sigma}\xi}},
\qquad
\tilde q^S(\xi)
=
q_m+\frac{n_m-\tilde n^S(\xi)}{\sigma}.
\end{equation}
Thus
\begin{equation}\label{nS'<0}
\tilde n^S_\xi<0,
\qquad
\tilde q^S_\xi=-\frac1\sigma\tilde n^S_\xi>0.
\end{equation}

\begin{lemma}\label{properties of the 2-viscous shock wave}
Let
\begin{equation}\label{-sigma}
\sigma_m
:=
\frac{-q_m+\sqrt{q_m^2+4n_m}}{2}
=
\lambda_2(n_m,q_m).
\end{equation}
There exist $\delta_1>0$ and $C>0$, independent of $\delta_S$, such
that for $0<\delta_S<\delta_1$,
\begin{equation}\label{-sigma/2}
0<\frac{\sigma_m}{2}
\le
\sigma_m-C\delta_S
\le
\sigma
<
\sigma_m.
\end{equation}
Moreover, for every $\xi\in\mathbb R$,
\begin{equation}\label{N'}
-\frac{\delta_S^2}{\sigma}
e^{-\frac{\delta_S}{\sigma}|\xi|}
\le
\tilde n^S_\xi(\xi)
\le
-\frac{\delta_S^2}{4\sigma}
e^{-\frac{\delta_S}{\sigma}|\xi|},
\end{equation}
and
\begin{equation}\label{N''}
\begin{aligned}
&|(\tilde n^S,\tilde q^S)(\xi)-(n_m,q_m)|
\le
C\delta_S e^{-c\delta_S|\xi|},
\qquad \xi<0,\\
&|(\tilde n^S,\tilde q^S)(\xi)-(n_+,q_+)|
\le
C\delta_S e^{-c\delta_S|\xi|},
\qquad \xi>0,\\
&|(\tilde n^S_{\xi\xi},\tilde q^S_{\xi\xi})|
\le
C\delta_S
|(\tilde n^S_\xi,\tilde q^S_\xi)|,\\
&|(\tilde n^S_{\xi\xi\xi},\tilde q^S_{\xi\xi\xi})|
\le
C\delta_S^2
|(\tilde n^S_\xi,\tilde q^S_\xi)|,\\
&\|(\tilde n^S_\xi,\tilde q^S_\xi)\|_{L^1}
\le C\delta_S,\qquad
\|(\tilde n^S_\xi,\tilde q^S_\xi)\|_{L^2}
\le C\delta_S^{3/2},\\
&\|(\tilde n^S_\xi,\tilde q^S_\xi)\|_{L^\infty}
\le C\delta_S^2.
\end{aligned}
\end{equation}
\end{lemma}

\begin{proof}
The comparison \eqref{-sigma/2} follows from
\[
\sigma_m-\sigma
=
\frac{2\delta_S}
{\sqrt{q_m^2+4n_m}+\sqrt{q_m^2+4n_+}}
\]
and the positivity of the fixed end states.  From \eqref{nSqS},
\[
\tilde n^S_\xi
=
-\frac{\delta_S^2}{\sigma}
\frac{e^{\frac{\delta_S}{\sigma}\xi}}
{\left(1+e^{\frac{\delta_S}{\sigma}\xi}\right)^2}.
\]
For $y>0$,
\[
\frac14\min\{y,y^{-1}\}
\le
\frac{y}{(1+y)^2}
\le
\min\{y,y^{-1}\},
\]
which gives \eqref{N'}.  The first two estimates in \eqref{N''} follow
from the explicit profile and \eqref{shock-q-relation-prelim}.
Differentiating \eqref{nS'} once and twice yields the second- and
third-derivative bounds.  The $L^p$ estimates follow by integrating
\eqref{N'} and using \eqref{shock-q-relation-prelim}.
\end{proof}

\subsection{Construction of the shift function $\mathbf X(t)$}
Define
\begin{equation}\label{4.a}
a(\xi)
:=
1+\frac{\lambda}{\delta_S}
\big(n_m-\tilde n^S(\xi)\big),
\qquad
\lambda:=\sqrt{\delta_S}.
\end{equation}
Then
\[
1<a(\xi)<1+\lambda,
\]
and
\begin{equation}\label{4.a'}
a_\xi(\xi)
=
-\frac{\lambda}{\delta_S}\tilde n^S_\xi(\xi)>0,
\qquad
\int_{\mathbb R}a_\xi\,d\xi=\lambda.
\end{equation}

We work in the shock frame
\[
\xi=x-\sigma t.
\]
The system \eqref{1.1} becomes
\begin{equation}\label{1D orginal model' rewrite}
\begin{cases}
n_t-\sigma n_\xi-(nq)_\xi=n_{\xi\xi},\\
q_t-\sigma q_\xi-n_\xi=0.
\end{cases}
\end{equation}

For a time-dependent shift $\mathbf X(t)$, set
\[
(\tilde n^S)^{-\mathbf X}(\xi,t)
:=
\tilde n^S(\xi-\mathbf X(t)),
\qquad
(\tilde q^S)^{-\mathbf X}(\xi,t)
:=
\tilde q^S(\xi-\mathbf X(t)),
\]
and
\[
a^{-\mathbf X}(\xi,t)
:=
a(\xi-\mathbf X(t)),
\qquad
a_\xi^{-\mathbf X}(\xi,t)
:=
a_\xi(\xi-\mathbf X(t)).
\]
The smooth shifted composite profile is
\begin{equation}\label{n-x}
\begin{aligned}
\tilde n^{-\mathbf X}(\xi,t)
&=
\tilde n^R(\xi+\sigma t,t)
+(\tilde n^S)^{-\mathbf X}(\xi,t)-n_m,\\
\tilde q^{-\mathbf X}(\xi,t)
&=
\tilde q^R(\xi+\sigma t,t)
+(\tilde q^S)^{-\mathbf X}(\xi,t)-q_m.
\end{aligned}
\end{equation}

We define $\mathbf X(t)$ by
\begin{equation}\label{X}
\begin{cases}
\begin{aligned}
\dot{\mathbf X}(t)
={}&
\frac{M}{\lambda}
\left[
\int_{\mathbb R}
a^{-\mathbf X}a_\xi^{-\mathbf X}
\frac{n-\tilde n^{-\mathbf X}}
{\tilde n^{-\mathbf X}}\,d\xi\right.\\
&\left.\hspace{2.5cm}
+
\int_{\mathbb R}
a^{-\mathbf X}a_\xi^{-\mathbf X}
\frac{\varphi(n)}{\sigma}\,d\xi
\right]
=:F(t,\mathbf X(t)),
\end{aligned}\\
\mathbf X(0)=0,
\end{cases}
\end{equation}
where
\[
M
:=
\frac{6\sigma_m^3n_m}
{(\sigma_m^2+n_m)^2},
\]
and
\begin{equation}\label{Varphi(n)}
\varphi(n)
:=
\frac1\sigma
\left[
\Pi(n|\tilde n^{-\mathbf X})
+
\left(
1+
\frac{\delta_S}{\lambda}
\frac{a^{-\mathbf X}}{\tilde n^{-\mathbf X}}
\right)
(n-\tilde n^{-\mathbf X})
\right].
\end{equation}

On any time interval on which $n$ and $\tilde n^{-\mathbf X}$ remain
in a fixed compact subset of $(0,\infty)$, the map
$X\mapsto F(t,X)$ is locally Lipschitz.  This follows from the smooth
exponential localization of the shock profile and its derivatives.
Hence standard ODE theory yields a unique absolutely continuous
solution $\mathbf X(t)$ on the interval of existence of the PDE
solution.

Under the small perturbation regime used in Proposition
\ref{4 A priori estimates}, Lemma \ref{4.local inequality} gives
\[
|\varphi(n)|
\le
C|n-\tilde n^{-\mathbf X}|.
\]
Using \eqref{4.a'} and
$\int_{\mathbb R}a_\xi^{-\mathbf X}d\xi=\lambda$, we obtain
\begin{equation}\label{4.supF}
|\dot{\mathbf X}(t)|
=
|F(t,\mathbf X(t))|
\le
C
\|n-\tilde n^{-\mathbf X}\|_{L^\infty}.
\end{equation}
In particular, on every finite interval on which the perturbation is
bounded in $H^1$,
\begin{equation}\label{Xt}
|\mathbf X(t)|
\le Ct.
\end{equation}

For later use, introduce
\[
N(\xi,t):=\tilde n^R(\xi+\sigma t,t),
\qquad
Q(\xi,t):=\tilde q^R(\xi+\sigma t,t).
\]
By \eqref{nRqR},
\begin{equation}\label{nRqR'}
\begin{cases}
N_t-\sigma N_\xi-(NQ)_\xi=0,\\
Q_t-\sigma Q_\xi-N_\xi=0.
\end{cases}
\end{equation}
Combining \eqref{nRqR'} with the shifted shock equations, the composite
profile \eqref{n-x} satisfies
\begin{equation}\label{approximate combination}
\left\{
\begin{aligned}
&(\tilde n^{-\mathbf X})_t
-\sigma(\tilde n^{-\mathbf X})_\xi
+\dot{\mathbf X}(\tilde n^S)_{\xi}^{-\mathbf X}
-(\tilde n^{-\mathbf X}\tilde q^{-\mathbf X})_\xi\\
&\qquad
=
(\tilde n^{-\mathbf X})_{\xi\xi}+F_1+F_2,\\
&(\tilde q^{-\mathbf X})_t
-\sigma(\tilde q^{-\mathbf X})_\xi
+\dot{\mathbf X}(\tilde q^S)_{\xi}^{-\mathbf X}
-(\tilde n^{-\mathbf X})_\xi
=0.
\end{aligned}
\right.
\end{equation}
Here
\[
(\tilde n^S)_{\xi}^{-\mathbf X}(\xi,t)
:=
\tilde n^S_\xi(\xi-\mathbf X(t)),
\qquad
(\tilde q^S)_{\xi}^{-\mathbf X}(\xi,t)
:=
\tilde q^S_\xi(\xi-\mathbf X(t)),
\]
and
\begin{equation}\label{F1F2}
\begin{aligned}
F_1
&:=
(\tilde n^S)_{\xi\xi}^{-\mathbf X}
-(\tilde n^{-\mathbf X})_{\xi\xi}
=
-N_{\xi\xi},\\
F_2
&:=
\big[
NQ
+(\tilde n^S)^{-\mathbf X}(\tilde q^S)^{-\mathbf X}
-\tilde n^{-\mathbf X}\tilde q^{-\mathbf X}
\big]_\xi\\
&=
-\partial_\xi
\left[
(N-n_m)\big((\tilde q^S)^{-\mathbf X}-q_m\big)
+\big((\tilde n^S)^{-\mathbf X}-n_m\big)(Q-q_m)
\right].
\end{aligned}
\end{equation}
Subtracting \eqref{approximate combination} from
\eqref{1D orginal model' rewrite} yields
\begin{equation}\label{4.perturbed system}
\left\{
\begin{aligned}
&(n-\tilde n^{-\mathbf X})_t
-\sigma(n-\tilde n^{-\mathbf X})_\xi
-\dot{\mathbf X}(\tilde n^S)_{\xi}^{-\mathbf X}
-(nq-\tilde n^{-\mathbf X}\tilde q^{-\mathbf X})_\xi\\
&\qquad
=
(n-\tilde n^{-\mathbf X})_{\xi\xi}-F_1-F_2,\\
&(q-\tilde q^{-\mathbf X})_t
-\sigma(q-\tilde q^{-\mathbf X})_\xi
-\dot{\mathbf X}(\tilde q^S)_{\xi}^{-\mathbf X}
-(n-\tilde n^{-\mathbf X})_\xi
=0.
\end{aligned}
\right.
\end{equation}

\subsection{Composite wave of a rarefaction wave and a viscous shock}
Given the end states $(n_\pm,q_\pm)$, we assume that there exists a
unique intermediate state $(n_m,q_m)$ such that
\begin{equation}\label{R1S2}
(n_m,q_m)\in R_1(n_-,q_-),
\qquad
(n_m,q_m)\in S_2(n_+,q_+).
\end{equation}
The corresponding inviscid-viscous composite wave is
\begin{equation}\label{composite wave}
\left(
n^r(x/t)+\tilde n^S(x-\sigma t)-n_m,\,
q^r(x/t)+\tilde q^S(x-\sigma t)-q_m
\right),
\qquad t>0,
\end{equation}
where $(n^r,q^r)$ is defined by \eqref{nr,qr} and
$(\tilde n^S,\tilde q^S)$ is the viscous shock profile above.

\section{Proof of main results}
\label{Sect.3}
In this section, we investigate the stability of the composite wave for the system \eqref{1.1}-\eqref{i d} and prove Theorem \ref{composite wave theorem}. Applying the standard iteration method, one can readily derive the local well-posedness of the system \eqref{1D orginal model' rewrite}. The routine and tedious proof details are omitted for brevity.

\begin{proposition}[Local existence]\label{Local existence}
Let $\underline{n}$ and $\underline{q}$ be smooth monotone functions such that
\begin{equation}\label{smooth monotone}
  \underline{n}(x)=n_\pm, \ \underline{q}(x)=q_\pm,\  \text{for}\ \pm\geq 1.
\end{equation}
For any constants \( M_0, M_1, \kappa_0, \bar{\kappa}_0, \kappa_1, \bar{\kappa}_1 \), with \( M_1 > M_0 > 0 \) and  \( \bar{\kappa}_1 > \bar{\kappa}_0 > \kappa_1 > 0 \), there exists a constant \( T_0 > 0 \) such that if
\begin{align*}
& \|n_0 - \underline{n}\|_{H^1(\mathbb{R})} + \|q_0 - \underline{q}\|_{H^1(\mathbb{R})} \leq M_0, \\
 & 0 < \kappa_0 \leq n_0(x) \leq \bar{\kappa}_0, \quad \forall x \in \mathbb{R},
\end{align*}
then system \eqref{1D orginal model' rewrite} has a unique solution \((n, q)\) on \([0, T_0]\) such that
\begin{align*}
  & n - \underline{n} \in C([0, T_0]; H^1(\mathbb{R}))\cap L^2(0, T_0; H^2(\mathbb{R})), \\
   & q - \underline{q} \in C([0, T_0]; H^1(\mathbb{R})),
\end{align*}
and
\begin{align*}
  \sup_{0\leq t\leq T_0}\left\|(n-\underline{n},
q-\underline{q})(\cdot, t)
\right\|_{H^{1}}\leq M_{1}.
\end{align*}
Moreover,
\begin{align}\label{n}
\kappa_1\leq n(x,t)\leq \bar{\kappa}_1, \quad \forall (x,t) \in \mathbb{R}\times [0, T_0].
\end{align}
\end{proposition}

To show the global existence result claimed in Theorem \ref{composite wave theorem}, in view of the local existence result and the standard continuation argument,
it suffices to establish the following \emph{a priori} estimate.

\begin{proposition}[\emph{A priori} estimate]\label{4 A priori estimates}
For a given end state 
  $(n_+,q_+)\in \mathbb{R}^+\times\mathbb{R}$, suppose that$(n,q)$ is the solution to \eqref{1D orginal model' rewrite} on $[0,T]$ for $T>0$, and  $(\tilde{n}^{-\mathbf{X}},\tilde{q}^{-\mathbf{X}})$ is defined in \eqref{n-x} with the shift function $\mathbf{X}(t)$. 
  Then there exist positive constants $\delta_0\leqq 1$, $\chi_{1}\leqq 1$ and $C_{0}$ independent of $T$, such that if both the rarefaction and shock waves strength satisfy $\delta_R,\ \delta_S<\delta_0$,
\begin{equation*}
\begin{aligned}
&n(t,x)-\tilde{n}^{-\mathbf{X}}\in C([0,T];H^{1}(\mathbb{R}))\cap L^{2}(0,T;H^{2}(\mathbb{R})),\\
&q(t,x)-\tilde{q}^{-\mathbf{X}}\in C([0,T];H^{1}(\mathbb{R})),
\end{aligned}
\end{equation*}
and
\begin{equation}\label{4.assumption}
\sup_{0\leq t\leq T}
\left\|(n-\tilde{n}^{-\mathbf{X}},q-\tilde{q}^{-\mathbf{X}})(\cdot, t)
\right\|_{H^{1}}\leq \chi_{1},
\end{equation}
then the following estimates hold\textup{:}
\begin{equation}\label{4.priori estimate}
\begin{aligned}
&\sup_{0\leq t\leq T}
\left\|(n-\tilde{n}^{-\mathbf{X}},q-\tilde{q}^{-\mathbf{X}})(\cdot,t)
\right\|_{H^{1}}^{2}
+\delta_S \int_{0}^{T}\left|\dot{\mathbf{X}}(t)\right|^{2}dt\\
&+\int_{0}^{T}
(\left\|\partial _{\xi}(n-\tilde{n}^{-\mathbf{X}})(\cdot,t)\right\|_{H^{1}}^{2}
+\left\|\partial _{\xi}(q-\tilde{q}^{-\mathbf{X}})(\cdot,t)\right\|^{2})dt\\
&\leq C_{0}\left\|(n_{0}-\tilde{n},q_{0}-\tilde{q})(\cdot)\right\|_{H^{1}}^{2}
+C_{0}\delta_R^{1/3}.
\end{aligned}
\end{equation}
and
\begin{equation}\label{4.dotX}
|\dot{\mathbf{X}}(t)|
\leq C_{0}
\|(n-\tilde{n}^{-\mathbf{X}})(\cdot,t)\|_{L^{\infty}}
,\quad \forall t\leq T.
\end{equation}
\end{proposition}
By Lemma \ref{4.local inequality}, whenever  $|n-\tilde{n}^{-\mathbf{X}}|\ll 1$, there exist two positive constants $C_1$ and $C_2$ such that
\begin{equation}\label{4.phi n-ntilde}
C_1\left|n-\tilde{n}^{-\mathbf{X}}\right|^{2}
\leq \Pi(n|\tilde{n}^{-\mathbf{X}})
\leq C_2
\left|n-\tilde{n}^{-\mathbf{X}}\right|^{2}.
\end{equation}
Hence, to prove Proposition \ref{4 A priori estimates}, we first estimate the relative entropy $\Pi(n|\tilde{n}^{-\mathbf{X}})+
\frac{\left|q-\tilde{q}^{-\mathbf{X}}\right|^{2}}{2}$.

\emph{\bf{Notations.}} In what follows, we use the following notations for simplicity.
\begin{enumerate}
  \item $C$ denotes a positive constant which is independent of the small constants $\delta_0$, $\chi_1$, $\delta_S$, $\delta_R$, $\lambda$ and the time $T$.
  \item For any function  $f:\mathbb{R}^+\times \mathbb{R}\rightarrow \mathbb{R}$ and time-dependent shift  $\mathbf{X}(t)$, we set
      \begin{equation*}
        f^{\pm\mathbf{X}}(\xi,t):=f(\xi\pm \mathbf{X}(t),t).
      \end{equation*}
  \item We omit the dependence on  $\mathbf{X}(t)$ in \eqref{n-x} as follows:
  \begin{equation*}
  \begin{aligned}
        (\tilde{n}, \tilde{q})(\xi,t):= &\left(\tilde{n}^R(\xi+\sigma t,t)
 +\tilde{n}^S(\xi-\mathbf{X}(t))-n_m,\right.\\
&~~\left.\tilde{q}^R(\xi+\sigma t,t)
+\tilde{q}^S(\xi-\mathbf{X}(t))-q_m\right).
      \end{aligned}
      \end{equation*}
\end{enumerate}

\subsection{Relative entropy estimate}
In this subsection, we shall estimate the relative entropy function $\Pi(n|\tilde{n}^{-\mathbf{X}})
\frac{\left|q-\tilde{q}^{-\mathbf{X}}\right|^{2}}{2}$, which plays a crucial role in  deriving the $L^2$ energy estimate. To this end, we consider employing the classical relative entropy method, which originally introduced by Dafermos\cite{Dafermos 79} and  Diperna\cite{DiPerna79} in the context of establishing the $L^2$ stability and uniqueness of Lipschitz solutions to the  hyperbolic conservation laws  endowed with convex entropy.

To apply the relative entropy method, we rewrite \eqref{1.1} into the following general system of viscous conservation laws:
\begin{equation}\label{vis conservation}
  \partial_t U + \partial_\xi [A(U)] = \partial_\xi [M(U) \partial_\xi \nabla \eta(U)],
\end{equation}
where
\begin{equation}\label{U A(U)}
  \begin{aligned}
  &U := \begin{pmatrix} n \\ q \end{pmatrix}, \quad
A(U) := \begin{pmatrix} -nq - \sigma n \\ -n - \sigma q \end{pmatrix}, \quad
M(U) := \begin{pmatrix} n & 0 \\ 0 & 0 \end{pmatrix}, \\
  &\eta(U) := \frac{|q|^2}{2} + \Pi(n), \quad \text{with}\  \Pi(n) := n \ln n - n.
\end{aligned}
\end{equation}
Moreover,
\begin{equation}\label{eta'}
  \nabla \eta(U) = (\partial_n \eta(U) \; \partial_q \eta(U))
  = (\ln n \; q).
\end{equation}

Let
\begin{equation}\label{Utilde}
\widetilde{U}(\xi,t):=\begin{pmatrix}
\widetilde{n}(\xi,t) \\ \widetilde{q}(\xi,t)
 \end{pmatrix}
=\begin{pmatrix}\widetilde{n}^R(\xi,t)
 +(\widetilde{n}^S)^{-\mathbf{X}}(\xi)-n_m\\
\widetilde{q}^R(\xi,t)
+(\widetilde{q}^S)^{-\mathbf{X}}(\xi)-q_m
\end{pmatrix},
\end{equation}
then the above system \eqref{approximate combination} can be rewritten as
\begin{equation}\label{Utilde A(U)}
 \partial_t \widetilde{U} + \partial_\xi [A(\widetilde{U})] = \partial_\xi [M(\widetilde{U}) \partial_\xi \nabla \eta(\widetilde{U})]
 -\dot{\mathbf{X}}\partial_\xi
 ((\widetilde{U}^S)^{-\mathbf{X}})
 +\begin{pmatrix}
 F_1+F_2 \\ 0
 \end{pmatrix},
\end{equation}
where $F_1, F_2$ are defined in \eqref{F1F2}.

In view of the definition of the relative function in \eqref{fu,v}, for \( U_i = \binom{n_i}{q_i}, i = 1, 2 \), we have
\begin{equation}\label{A(U)}
  \begin{aligned}
  A(U_1 | U_2) =& A(U_1) - A(U_2) - \nabla A(U_2)(U_1 - U_2) \\
  =& \begin{pmatrix} -(n_1 - n_2)(q_1 - q_2) \\ 0 \end{pmatrix},
  \end{aligned}
\end{equation}
%
and
\begin{align}\label{eta}
  \eta(U_1 | U_2) = \eta(U_1) - \eta(U_2) - \nabla \eta(U_2)(U_1 - U_2) = \frac{|q_1 - q_2|^2}{2} + \Pi(n_1 | n_2),
\end{align}
where, since
 \(\Pi(n) = n \ln n - n\), one readily finds that
\begin{align}\label{Pin1n2'}
\Pi(n_1 | n_2) = n_1 \ln \left( \frac{n_1}{n_2} \right) - (n_1 - n_2).
\end{align}
We define the corresponding flux \( G(\cdot; \cdot) \) for our relative entropy \( \eta(\cdot | \cdot) \) by
\begin{equation}\label{GU1U2}
  \begin{aligned}
 G(U_1; U_2) :=& G(U_1) - G(U_2) - \nabla \eta(U_2)(A(U_1) - A(U_2))\\
=&-(q_1 - q_2)\Pi(n_1 | n_2) - q_2\Pi(n_1 | n_2) - (n_1 - n_2)(q_1 - q_2) - \sigma \eta(U_1 | U_2),
\end{aligned}
\end{equation}
where $G$ is the entropy flux of $\eta$, satisfying   \(\partial_i G(U) = \sum_{k=1}^2 \partial_k \eta(U) \partial_i A_k(U)\), \(1 \leq i \leq 2\), i.e.
\begin{align}\label{G}
G(U) := -qn \log n - \sigma \eta(U).
\end{align}

Below, we will estimate the relative entropy of the solution $U$ of \eqref{vis conservation} w.r.t. the shifted wave $\widetilde{U}$( see \eqref{Utilde}) as follows:
\begin{equation*}
  \eta(U(\xi, t)|\widetilde{U}(\xi, t))=\Pi(n|\tilde{n})+ \frac{\left|q-\tilde{q}\right|^{2}}{2}.
\end{equation*}

\begin{lemma}\label{relative entropy}
Let $a$ be the weight function defined by \eqref{4.a}. Let $U$ be a solution to \eqref{vis conservation}, and  $\widetilde{U}$ the shifted wave satisfying  \eqref{Utilde A(U)}. Then,
\begin{equation}\label{dt}
\begin{aligned}
\frac{d}{dt}\int_{\mathbb{R}}
a^{-\mathbf{X}}(\xi)
\left(\Pi(n|\tilde{n})
+\frac{\left|q-\tilde{q}\right|^{2}}{2}
\right)d\xi
=\mathbf{\dot{X}}(t)\mathbf{Y}(t)+\sum_{i=1}^{9}\mathbf{B_{i}}(t)-\mathbf{G}(t)-\mathbf{D}(t),
\end{aligned}
\end{equation}
where \allowbreak
\begin{align*}
\mathbf{Y}(t):=&-\int_{\mathbb{R}} a_{\xi}^{-\mathbf{X}} \eta(U | \widetilde{U}) d\xi+ \int_{\mathbb{R}} a^{-\mathbf{X}}\nabla^2 \eta(\widetilde{U})(U - \widetilde{U}) (\widetilde{U}^S)^{-\mathbf{X}}_\xi d\xi,\\
\mathbf{B_1}(t):=&\frac{\sigma}{2}\int_{\mathbb{R}}
a^{-\mathbf{X}}_{\xi}
\left|\varphi(n)\right|^{2}
d\xi,\\
\mathbf{B_{2}}(t):=&
-\int_{\mathbb{R}}
a^{-\mathbf{X}}_{\xi}\tilde{q}
\Pi(n|\tilde{n})
d\xi,\\
\mathbf{B_{3}}(t):=&
\int_{\mathbb{R}}a^{-\mathbf{X}}
(\tilde{n}^S)^{-\mathbf{X}}_{\xi\xi}
\frac{\Pi(n|\tilde{n})}{\tilde{n}}
d\xi,\\
\mathbf{B_{4}}(t):=&\int_{\mathbb{R}}
\left[a^{-\mathbf{X}}
\frac{(\tilde{n}^S)^{-\mathbf{X}}_{\xi}}
{\tilde{n}}-a^{-\mathbf{X}}_{\xi}\right]
n\ln \frac{n}{\tilde{n}}
\partial_{\xi}\left(\ln \frac{n}{\tilde{n}}\right)
d\xi,\\
\mathbf{B_{5}}(t):=&\int_{\mathbb{R}}
a^{-\mathbf{X}}
\frac{(\tilde{n}^R)_{\xi}}
{\tilde{n}}
n\ln \frac{n}{\tilde{n}}
\partial_{\xi}\left(\ln \frac{n}{\tilde{n}}\right)d\xi,\\
\mathbf{B_{6}}(t):=&\int_{\mathbb{R}}a^{-\mathbf{X}}
(\tilde{n}^R)_{\xi\xi}
\frac{\Pi(n|\tilde{n})}{\tilde{n}}d\xi,\\
\mathbf{B_{7}}(t):=&-\int_{\mathbb{R}}
a^{-\mathbf{X}}
\frac{n-\tilde{n}}{\tilde{n}}
F_1d\xi,\\
\mathbf{B_{8}}(t):=&-\int_{\mathbb{R}}
a^{-\mathbf{X}}
\frac{n-\tilde{n}}{\tilde{n}}
F_2d\xi,\\
\mathbf{B_{9}}(t):=&\int_{\mathbb{R}}
a^{-\mathbf{X}}
\frac{(\tilde{n}^R)_{\xi}}
{\tilde{n}}
(n-\tilde{n})(q-\tilde{q})d\xi,
\end{align*}
and \allowbreak
\begin{align*}
\mathbf{G}(t):=&\sigma\int_{\mathbb{R}}
a^{-\mathbf{X}}_{\xi}
\Pi\left(n|\tilde{n}\right)d\xi
+\frac{\sigma}{2}\int_{\mathbb{R}}
a^{-\mathbf{X}}_{\xi}
\left|q-\tilde{q}+\varphi(n)\right|^{2}
d\xi\\
\triangleq & \mathbf{G_{1}}(t)+\mathbf{G_{2}}(t),\\
\mathbf{D}(t):=&\int_{\mathbb{R}}
a^{-\mathbf{X}}n\left|\partial_{\xi}(\ln \frac{n}{\tilde{n}})\right|^{2}
d\xi,
\end{align*}
with $\varphi(n)$ given by \eqref{Varphi(n)}.
\end{lemma}
\begin{remark}\label{R.4.1}
In \eqref{dt}, $\mathbf{G}(t)$ and $\mathbf{D}(t)$ are good terms, while $\sum\limits_{i=1}^{9}\mathbf{B_i}(t)$ consists of bad terms.
\end{remark}

\begin{proof}
By the definition of the relative entropy with \eqref{vis conservation} and \eqref{eta}, one first gets
\begin{align*}
&\frac{d}{dt} \int_{\mathbb{R}} a^{-\mathbf{X}}(\xi) \eta(U(\xi,t) | \widetilde{U}(\xi,t)) d\xi\\
=& -\dot{\mathbf{X}}(t) \int_{\mathbb{R}} a_{\xi}^{-\mathbf{X}} \eta(U | \tilde{U}) d\xi \\
&+ \int_{\mathbb{R}} a^{-\mathbf{X}} \left[ \left( \nabla \eta(U) - \nabla \eta(\widetilde{U}) \right) \partial_t U - \nabla^2 \eta(\widetilde{U})(U - \widetilde{U}) \partial_t \widetilde{U} \right] d\xi \\
=& -\dot{\mathbf{X}}(t) \int_{\mathbb{R}} a_{\xi}^{-\mathbf{X}} \eta(U | \widetilde{U}) d\xi \\
&+ \int_{\mathbb{R}} a^{-\mathbf{X}} \Bigg[ \left( \nabla \eta(U) - \nabla \eta(\widetilde{U}) \right) \left( -\partial_\xi A(U) + \partial_\xi \left( M(U) \partial_\xi \nabla \eta(U) \right) \right) \\
&\qquad- \nabla^2 \eta(\widetilde{U})(U - \widetilde{U}) \left( -\partial_\xi A(\widetilde{U}) + \partial_\xi \left( M(\widetilde{U}) \partial_\xi \nabla \eta(\widetilde{U}) \right) \right. \\
&\qquad\left. - \dot{\mathbf{X}} \partial_\xi \left( (\widetilde{U}^S)^{-\mathbf{X}} \right) + \begin{pmatrix} F_1+F_2\\0 \end{pmatrix}\right) \Bigg] d\xi.
\end{align*}
It follows from the definitions  \eqref{A(U)} and  \eqref{GU1U2} along with the same computation as in  \cite[Lemma 4]{Vasseur08} (see also \cite[Lemma 2.3]{M.-J. Kang 21}) that
\begin{equation*}
  \frac{d}{dt} \int_{\mathbb{R}} a^{-\mathbf{X}}(\xi)\eta( U(\xi, t)|\widetilde{U}(\xi, t))d\xi = \dot{\mathbf{X}}(t)\mathbf{Y}(t) + \sum_{i=1}^{6} I_i,
\end{equation*}
where,
\begin{align}\label{4 Dt}
  I_1 &:= -\int_{\mathbb{R}} a^{-\mathbf{X}}\partial_\xi G(U; \widetilde{U})d\xi, \nonumber\\
I_2 &:= -\int_{\mathbb{R}} a^{-\mathbf{X}}
\partial_\xi\nabla\eta(\widetilde{U})A(U|\widetilde{U})d\xi, \nonumber\\
I_3 &:= \int_{\mathbb{R}} a^{-\mathbf{X}}\left(\nabla\eta(U) - \nabla\eta(\widetilde{U})\right)
\partial_\xi\left(M(U)\partial_\xi(\nabla\eta(U) - \nabla\eta(\widetilde{U}))\right)d\xi, \\
I_4 &:= \int_{\mathbb{R}} a^{-\mathbf{X}}\left(\nabla\eta(U) - \nabla\eta(\widetilde{U})\right)
\partial_\xi\left((M(U) - M(\widetilde{U}))\partial_\xi\nabla\eta(\widetilde{U})\right)d\xi, \nonumber\\
I_5 &:= \int_{\mathbb{R}} a^{-\mathbf{X}}(\nabla\eta)(U|\widetilde{U})
\partial_\xi\left(M(\widetilde{U})\partial_\xi\nabla\eta(\tilde{U})\right)d\xi, \nonumber\\
I_6 &:= -\int_{\mathbb{R}} a^{-\mathbf{X}}\nabla^2\eta(\widetilde{U})(U - \widetilde{U})\begin{pmatrix} F_1+F_2 \\ 0 \end{pmatrix}d\xi.\nonumber
\end{align}
Performing an integration by parts, and using \eqref{A(U)} and \eqref{GU1U2},we are led to
\begin{align*}
  I_1=& \int_{\mathbb{R}} a^{-\mathbf{X}}_\xi G(U; \widetilde{U}) d\xi \\
  =& -\int_{\mathbb{R}} a^{-\mathbf{X}}_\xi (q - \tilde{q}) \Pi(n | \tilde{n}) d\xi - \int_{\mathbb{R}} a^{-\mathbf{X}}_\xi \tilde{q} \Pi(n | \tilde{n}) d\xi\\
  &- \int_{\mathbb{R}} a^{-\mathbf{X}}_\xi  (n - \tilde{n})(q - \tilde{q}) d\xi - \sigma \int_{\mathbb{R}} a^{-\mathbf{X}}_\xi  \eta(U | \widetilde{U}) d\xi.
\end{align*}
Applying \eqref{A(U)} again, together with \eqref{eta'}, yields
\begin{align*}
I_2 = \int_{\mathbb{R}} a^{-\mathbf{X}} \frac{\tilde{n}_\xi}{\tilde{n}} (n - \tilde{n})(q - \tilde{q}) d\xi.
\end{align*}
By integration by parts, we have
\begin{align*}
I_3 = &\int_{\mathbb{R}} a^{-\mathbf{X}} \ln \frac{n}{\tilde{n}} \partial_\xi\left(n\partial_\xi\left(\ln \frac{n}{\tilde{n}} \right) \right) d\xi\\
=&- \int_{\mathbb{R}} a^{-\mathbf{X}}_\xi n \left( \ln \frac{n}{\tilde{n}} \right) \partial_\xi \left( \ln \frac{n}{\tilde{n}} \right) d\xi- \int_{\mathbb{R}} a^{-\mathbf{X}} n \left| \partial_\xi \left( \ln \frac{n}{\tilde{n}} \right) \right|^2 d\xi,\\
I_4 = &\int_{\mathbb{R}} a^{-\mathbf{X}} \ln \frac{n}{\tilde{n}} \partial_\xi\left((n-\tilde{n})
\partial_\xi\ln \tilde{n}\right) d\xi\\
=&- \int_{\mathbb{R}} a^{-\mathbf{X}}_\xi\frac{n-\tilde{n}}{\tilde{n}}
\tilde{n}_\xi \ln \frac{n}{\tilde{n}} d\xi- \int_{\mathbb{R}} a^{-\mathbf{X}}\frac{n-\tilde{n}}{\tilde{n}}
\tilde{n}_\xi \partial_\xi\left(\ln \frac{n}{\tilde{n}}\right) d\xi.
\end{align*}
Using \eqref{eta'} and
\begin{equation}\label{eta''}
  \nabla^2 \eta(\widetilde{U}) = \begin{pmatrix}
  \frac{1}{\tilde{n}}& 0  \\
     0 & 1
                     \end{pmatrix},
\end{equation}
we obtain
\begin{align*}
I_5 = \int_{\mathbb{R}} a^{-\mathbf{X}} \frac{\tilde{n}_{\xi\xi}}{\tilde{n}} \Pi(n|\tilde{n})d\xi-\int_{\mathbb{R}} a^{-\mathbf{X}}\frac{n-\tilde{n}}{\tilde{n}}
\tilde{n}_{\xi\xi} \ln \frac{n}{\tilde{n}} d\xi\triangleq I_{5,1}+I_{5,2}.
\end{align*}
It follows from the definition of $(\tilde{n},\tilde{q})$ in  \eqref{n-x} that
\begin{align}\label{n'xi}
\tilde{n}_\xi=(\tilde{n}^R)_\xi
+(\tilde{n}^S)^{-\mathbf{X}}_\xi \ \text{and} \ \tilde{n}_{\xi\xi}=(\tilde{n}^R)_{\xi\xi}
+(\tilde{n}^S)^{-\mathbf{X}}_{\xi\xi},
\end{align}
which implies
\begin{align*}
I_{5,1} =\int_{\mathbb{R}} a^{-\mathbf{X}} \frac{(\tilde{n}^R)_{\xi\xi}}{\tilde{n}} \Pi(n|\tilde{n})d\xi
+\int_{\mathbb{R}} a^{-\mathbf{X}} \frac{(\tilde{n}^S)^{-\mathbf{X}}_{\xi\xi}}{\tilde{n}} \Pi(n|\tilde{n})d\xi.
\end{align*}
Integrating $I_{5,2}$ by parts and noting  \eqref{n'xi}, one can see that
\begin{align*}
I_{5,2} = &\int_{\mathbb{R}} a^{-\mathbf{X}}_\xi\frac{n-\tilde{n}}{\tilde{n}}
\tilde{n}_\xi \ln \frac{n}{\tilde{n}} d\xi+ \int_{\mathbb{R}} a^{-\mathbf{X}}\frac{n-\tilde{n}}{\tilde{n}}
\tilde{n}_\xi \partial_\xi\left(\ln \frac{n}{\tilde{n}}\right) d\xi\\
&+\int_{\mathbb{R}} a^{-\mathbf{X}}\frac{\tilde{n}_\xi}{\tilde{n}} n \left( \ln \frac{n}{\tilde{n}} \right) \partial_\xi \left( \ln \frac{n}{\tilde{n}} \right)d\xi\\
=&-I_{4}+\int_{\mathbb{R}} a^{-\mathbf{X}}\frac{(\tilde{n}^R)_\xi}{\tilde{n}} n \left( \ln \frac{n}{\tilde{n}} \right) \partial_\xi \left( \ln \frac{n}{\tilde{n}} \right)d\xi\\&+\int_{\mathbb{R}} a^{-\mathbf{X}}\frac{
(\tilde{n}^S)^{-\mathbf{X}}_\xi}{\tilde{n}} n \left( \ln \frac{n}{\tilde{n}} \right) \partial_\xi \left( \ln \frac{n}{\tilde{n}} \right)d\xi.
\end{align*}
A direct calculation by \eqref{eta''} yields
\begin{align*}
I_6 = -\int_{\mathbb{R}} a^{-\mathbf{X}} \frac{n-\tilde{n}}{\tilde{n}}F_1
d\xi-\int_{\mathbb{R}} a^{-\mathbf{X}}\frac{n-\tilde{n}}{\tilde{n}}
F_2 d\xi.
\end{align*}
Therefore, we have
\begin{equation}\label{I2-I6}
\begin{aligned}
\sum_{i=1}^{6} I_i =
&\sum_{i=2}^{9} \mathbf{B_i}(t)
- \mathbf{G_1}(t)-\mathbf{D}(t)\\
& -\int_{\mathbb{R}}
\left[ a^{-\mathbf{X}}_\xi\Pi(n| \tilde{n})
+(a^{-\mathbf{X}}_\xi
-a^{-\mathbf{X}} \frac{(\tilde{n}^S)^{-\mathbf{X}}_\xi}{\tilde{n}})(n- \tilde{n})\right](q- \tilde{q}) d\xi
-\sigma \int_{\mathbb{R}} a^{-\mathbf{X}}_\xi\frac{|q-\tilde{q}|^2}{2} d\xi.
\end{aligned}
\end{equation}
To handle the last two terms on the right side of \eqref{I2-I6}, we employ the elementary identity $-\alpha \theta^{2}-\beta \theta=-\alpha (\theta+\frac{\beta}{2\alpha})^{2}+\frac{\beta^{2}}{4\alpha}$ with $\alpha=\frac{1}{2}$, $\beta=\varphi(n)$ and $\theta=q-\tilde{q}$. It then follows from  $(\tilde{n}^S)^{-\mathbf{X}}_\xi
=-\frac{\delta_S}{\lambda}
a^{-\mathbf{X}}_\xi$ that the last two terms of \eqref{I2-I6} are equal to
\begin{equation}\label{I.1}
\begin{aligned}
&-\int_{\mathbb{R}}
\left[ a^{-\mathbf{X}}_\xi\Pi(n| \tilde{n})
+(a^{-\mathbf{X}}_\xi - a^{-\mathbf{X}} \frac{(\tilde{n}^S)^{-\mathbf{X}}_\xi}{\tilde{n}})(n- \tilde{n})\right](q- \tilde{q}) d\xi
-\sigma \int_{\mathbb{R}} a^{-\mathbf{X}}_\xi\frac{|q-\tilde{q}|^2}{2} d\xi\\
=&-\sigma \int_{\mathbb{R}} a^{-\mathbf{X}}_\xi\left[\varphi(n)(q- \tilde{q}) +\frac{|q-\tilde{q}|^2}{2}\right]d\xi\\
=&\frac{\sigma}{2}\int_{\mathbb{R}}
a^{-\mathbf{X}}_{\xi}
\left|\varphi(n)\right|^{2}d\xi
-\frac{\sigma}{2}\int_{\mathbb{R}}
a^{-\mathbf{X}}_{\xi}
\left (q-\tilde{q}+\varphi(n)\right)^{2}
d\xi\\
=&\mathbf{B_1}(t)-\mathbf{G_{2}}(t).
\end{aligned}
\end{equation}
Substituting \eqref{I.1} into \eqref{I2-I6}  gives \eqref{dt}. The proof of Lemma \ref{relative entropy} is complete.
\end{proof}

For $\mathbf{Y}$, we have from \eqref{eta} and  \eqref{eta''} that
  \begin{align*}
    \mathbf{Y}(t)= & -\int_{\mathbb{R}} a_{\xi}^{-\mathbf{X}} \eta(U | \widetilde{U}) d\xi+ \int_{\mathbb{R}} a^{-\mathbf{X}}\nabla^2 \eta(\widetilde{U})(U - \widetilde{U}) (\widetilde{U}^S)^{-\mathbf{X}}_\xi d\xi \\
    = &-\int_{\mathbb{R}}
a^{-\mathbf{X}}_{\xi}\left(\Pi(n|\tilde{n})
+\frac{\left|q-\tilde{q}\right|^{2}}{2}\right)
d\xi\\
&+\int_{\mathbb{R}}
a^{-\mathbf{X}}(\tilde{n}^S)^{-\mathbf{X}}_{\xi}
\frac{n-\tilde{n}}{\tilde{n}}d\xi
+\int_{\mathbb{R}}
a^{-\mathbf{X}}(\tilde{q}^S)^{-\mathbf{X}}_{\xi}
(q-\tilde{q})
d\xi.
  \end{align*}
In order to derive the nonlinear stability of the composite wave, we decompose the function $\mathbf{Y}(t)$ in Lemma \ref{relative entropy} as
$$\mathbf{Y}(t):=\sum_{i=1}^{7}\mathbf{Y}_{i}(t),$$
where
\allowdisplaybreaks[4]
\begin{align*}
&\mathbf{Y}_1(t):=-\frac{\delta_S}{\lambda}\int_{\mathbb{R}}
a^{-\mathbf{X}}a^{-\mathbf{X}}_{\xi}
\frac{n-\tilde{n}}{\tilde{n}}
d\xi,\\
&\mathbf{Y}_{2}(t):=-\frac{\delta_S}{\lambda}\int_{\mathbb{R}}
a^{-\mathbf{X}}a^{-\mathbf{X}}_{\xi}
\frac{\varphi(n)}{\sigma}
d\xi,\\
&\mathbf{Y}_{3}(t):=-\int_{\mathbb{R}}
a^{-\mathbf{X}}_{\xi}
\frac{\left|\varphi(n)\right|^{2}}{2}
d\xi,\\
&\mathbf{Y}_{4}(t):=-\int_{\mathbb{R}}
a^{-\mathbf{X}}_{\xi}
\Pi\left(n|\tilde{n}\right)
d\xi,\\
&\mathbf{Y}_{5}(t):=-\frac{1}{2}\int_{\mathbb{R}}
a^{-\mathbf{X}}_{\xi}
\left|q-\tilde{q}+\varphi(n)\right|^{2}
d\xi,\\
&\mathbf{Y}_{6}(t):=\int_{\mathbb{R}}
a^{-\mathbf{X}}_{\xi}\varphi(n)
\left(q-\tilde{q}+\varphi(n)\right)
d\xi,\\
&\mathbf{Y}_{7}(t):=\frac{\delta_S}{\sigma\lambda}\int_{\mathbb{R}}
a^{-\mathbf{X}}a^{-\mathbf{X}}_{\xi}
\left(q-\tilde{q}+\varphi(n)\right)
d\xi.
\end{align*}
Notice from \eqref{X}, the shift $\mathbf{X}(t)$ satisfies
\begin{equation}\label{X'}
\dot{\mathbf{X}}(t)=-\frac{M}{\delta_S}(\mathbf{Y}_1(t)
+\mathbf{Y}_{2}(t)),
\end{equation}
which yields
\begin{equation}\label{X'Y}
\dot{\mathbf{X}}(t)\mathbf{Y}(t)=
-\frac{\delta_S}{M}|\dot{\mathbf{X}}(t)|^{2}
+\dot{\mathbf{X}}(t)\sum_{i=3}^{7}\mathbf{Y}_{i}(t).
\end{equation}

We need the the following Poincar\'{e} type inequality (see \cite[Lemma 2.9]{M.-J. Kang 21}).

\begin{lemma}\label{Poincare}

For any $f:[0,1]\rightarrow\R$ satisfying $\int_0^1y(1-y)|f^\prime|^2dy<\infty$, it holds
\[\int_0^1\left|f-\int_0^1f(y)dy\right|^2dy\leq \frac{1}{2}\int_0^1y(1-y)|f^\prime|^2dy.\]

\end{lemma}

\begin{lemma}\label{Leading order estimates}
Under the hypotheses of Proposition \ref{4 A priori estimates}, there exist
positive constants $M$ and $C$, independent of
$\delta_R,\delta_S,\chi_1$ and $t$, such that for every $t\in[0,T]$,
\begin{equation}\label{leading order estimates}
\begin{aligned}
&-\frac{\delta_S}{2M}|\dot{\mathbf X}(t)|^2
+\mathbf B_1(t)+\mathbf B_2(t)+\mathbf B_3(t)
-\mathbf G_1(t)-\frac34\mathbf D(t)\\
&\le
-C\int_{\mathbb R}
\left|(\tilde n^S)^{-\mathbf X}_\xi\right|
\left|\frac{n-\tilde n}{\tilde n}\right|^2\,d\xi
+C\int_{\mathbb R}a^{-\mathbf X}_\xi
\left|\frac{n-\tilde n}{\tilde n}\right|^3\,d\xi\\
&\quad
+\mathcal R_{\mathrm{rw},1}(t)
+\mathcal R_{\mathrm{rw},2}(t),
\end{aligned}
\end{equation}
where
\begin{equation}\label{rw-remainders-L32}
\begin{aligned}
\mathcal R_{\mathrm{rw},1}(t)
&:=
\frac{C}{\delta_S}
\left(
\int_{\mathbb R}
\left|(\tilde n^S)^{-\mathbf X}_\xi\right|
\left|\tilde n^R-n_m\right|
\left|\frac{n-\tilde n}{\tilde n}\right|\,d\xi
\right)^2,\\
\mathcal R_{\mathrm{rw},2}(t)
&:=
C\int_{\mathbb R}a^{-\mathbf X}_\xi
\left(
|\tilde n^R-n_m|+|\tilde q^R-q_m|
\right)
\left[
\left|\frac{n-\tilde n}{\tilde n}\right|^2
+\left|\frac{n-\tilde n}{\tilde n}\right|^3
\right]d\xi .
\end{aligned}
\end{equation}
Here and throughout this subsection,
$(\tilde n^R,\tilde q^R)$ is evaluated at $(\xi+\sigma t,t)$.
\end{lemma}

\begin{proof}
\emph{Step 1}. We first separate the shock part from the rarefaction part. Set
\[
r_n(\xi,t):=\tilde n^R(\xi+\sigma t,t)-n_m,\qquad
r_q(\xi,t):=\tilde q^R(\xi+\sigma t,t)-q_m,
\]
and
\[
s_n(\xi,t):=(\tilde n^S)^{-\mathbf X}(\xi)-n_m,\qquad
s_q(\xi,t):=(\tilde q^S)^{-\mathbf X}(\xi)-q_m.
\]
Then
\[
\tilde n-n_m=s_n+r_n,\qquad
\tilde q-q_m=s_q+r_q.
\]
By Lemma \ref{properties of the 2-viscous shock wave},
\begin{equation}\label{n-n-}
|\sigma-\sigma_m|\le C\delta_S,\qquad
\|s_n\|_{L^\infty}+\|s_q\|_{L^\infty}\le C\delta_S.
\end{equation}
In this lemma we do not estimate $r_n$ and $r_q$ by $\delta_S$; all terms containing
$r_n$ or $r_q$ will be kept explicitly until the end.

We rewrite the main terms in terms of the variables $y$ and $w$:
\begin{equation}\label{y}
y:=\frac{n_m-(\tilde n^S)^{-\mathbf X}(\xi)}{\delta_S},
\end{equation}
and
\begin{equation}\label{4.w(y)}
w(y):=\frac{n(\xi(y))-\tilde n(\xi(y))}
{\tilde n(\xi(y))}.
\end{equation}
By changing the variable $\xi\in(-\infty,\infty)\mapsto y\in(0,1)$,
it follows from \eqref{4.a} that
\begin{equation}\label{axi}
a^{-\mathbf X}(\xi)=1+\lambda y,
\end{equation}
and
\begin{equation}\label{dy}
\frac{dy}{d\xi}
=-\frac{1}{\delta_S}(\tilde n^S)^{-\mathbf X}_\xi
=\frac{1}{\lambda}a^{-\mathbf X}_\xi>0.
\end{equation}
In the sequel, by a slight abuse of notation, $r_n$ and $r_q$ also denote
$r_n(\xi(y),t)$ and $r_q(\xi(y),t)$ after the above change of variables.

\emph{Step 2}. We next estimate
$-\frac{\delta_S}{2M}|\dot{\mathbf X}(t)|^2$.
Due to \eqref{X'}, we first control $\mathbf Y_1(t)$ and $\mathbf Y_2(t)$.
From \eqref{y},
\[
\mathbf Y_1(t)=-\delta_S\int_0^1 a^{-\mathbf X}w\,dy,
\]
and hence, using $|a^{-\mathbf X}-1|\le\lambda$,
\begin{equation}\label{4.Y1}
\left|\mathbf Y_1(t)+\delta_S\int_0^1w\,dy\right|
\le \delta_S\lambda\int_0^1|w|\,dy.
\end{equation}

In the new variable,
\[
\mathbf Y_2(t)
=-\frac{\delta_S}{\sigma}\int_0^1a^{-\mathbf X}\varphi(n)\,dy.
\]
Since
$\left\|\frac n{\tilde n}-1\right\|_{L^\infty(\mathbb R)}
\le C\chi_1$ for $\chi_1$ small, it follows from \eqref{PI<} that
\begin{equation}\label{4.phi(n)}
\left|
\varphi(n)-\frac{\tilde n}{\sigma}
\left(\frac n{\tilde n}-1\right)
\right|
\le
C\left(\chi_1+\frac{\delta_S}{\lambda}\right)
\left|\frac n{\tilde n}-1\right|.
\end{equation}
Using \eqref{axi}, \eqref{n-n-}, and
$\tilde n=n_m+s_n+r_n$, we therefore obtain
\begin{equation}\label{4.Y2}
\begin{aligned}
\left|
\mathbf Y_2(t)+\frac{\delta_Sn_m}{\sigma_m^2}
\int_0^1w\,dy
\right|
\le{}&
C\delta_S
\left(\chi_1+\frac{\delta_S}{\lambda}\right)
\int_0^1|w|\,dy\\
&+C\delta_S\int_0^1|r_n|\,|w|\,dy .
\end{aligned}
\end{equation}
Here and below we use $\lambda=\sqrt{\delta_S}$, so that
$\lambda=\delta_S/\lambda$ and all shock-only coefficient errors are contained
in the first term on the right-hand side.

By \eqref{X'}, \eqref{4.Y1}, and \eqref{4.Y2},
\begin{equation*}
\begin{aligned}
&\left|
\dot{\mathbf X}(t)
-M\left(1+\frac{n_m}{\sigma_m^2}\right)
\int_0^1w\,dy
\right|\\
&\qquad\le
C\left(\chi_1+\frac{\delta_S}{\lambda}\right)
\int_0^1|w|\,dy
+C\int_0^1|r_n|\,|w|\,dy.
\end{aligned}
\end{equation*}
For arbitrary real numbers $\alpha,\beta$,
\[
-\beta^2\le -\frac{\alpha^2}{2}+(\alpha-\beta)^2.
\]
Taking
\[
\beta=\dot{\mathbf X}(t),\qquad
\alpha=M\left(1+\frac{n_m}{\sigma_m^2}\right)\int_0^1w\,dy,
\]
and using Cauchy--Schwarz, we have
\[
\begin{aligned}
-|\dot{\mathbf X}(t)|^2
\le{}&
-\frac{M^2(\sigma_m^2+n_m)^2}{2\sigma_m^4}
\left(\int_0^1w\,dy\right)^2\\
&+C\left(\chi_1+\frac{\delta_S}{\lambda}\right)^2
\int_0^1w^2\,dy
+C\left(\int_0^1|r_n|\,|w|\,dy\right)^2 .
\end{aligned}
\]
Multiplying by $\frac{\delta_S}{2M}$ gives
\begin{equation}\label{4.X^2}
\begin{aligned}
-\frac{\delta_S}{2M}|\dot{\mathbf X}(t)|^2
\le{}&
-\frac{M\delta_S(\sigma_m^2+n_m)^2}{4\sigma_m^4}
\left(\int_0^1w\,dy\right)^2\\
&+C\delta_S
\left(\chi_1+\frac{\delta_S}{\lambda}\right)^2
\int_0^1w^2\,dy\\
&+C\delta_S
\left(\int_0^1|r_n|\,|w|\,dy\right)^2 .
\end{aligned}
\end{equation}

\emph{Step 3}. We next estimate
$\mathbf B_1(t)+\mathbf B_2(t)+\mathbf B_3(t)-\mathbf G_1(t)$.
For later use, \eqref{PI>}--\eqref{PI<} imply, under the a priori assumption,
\begin{equation}\label{Pi-expansion-L32}
\Pi(n|\tilde n)
=
\frac{\tilde n}{2}
\left(w^2-\frac13w^3\right)+\mathcal E_\Pi,
\qquad
|\mathcal E_\Pi|\le C\chi_1|w|^3 .
\end{equation}
This two-sided expansion allows us to estimate $\mathbf B_2$ without imposing
any sign condition on $\tilde q$.

From the definition of $\varphi$ and \eqref{Pi-expansion-L32},
\begin{equation}\label{B1}
\begin{aligned}
\mathbf B_1(t)\le{}&
\lambda\frac{n_m^2}{2\sigma_m}\int_0^1w^2\,dy
+\delta_S\frac{n_m}{\sigma_m}\int_0^1w^2\,dy
+\lambda\frac{n_m^2}{2\sigma_m}\int_0^1w^3\,dy\\
&+C\delta_S(\lambda+\delta_S+\chi_1)\int_0^1w^2\,dy
+C\lambda(\delta_S+\chi_1)\int_0^1|w|^3\,dy\\
&+C\lambda\int_0^1|r_n|\left(w^2+|w|^3\right)\,dy .
\end{aligned}
\end{equation}
Indeed, only $s_n$ and $\sigma-\sigma_m$ are used in the shock-only error terms;
the contribution of $r_n$ is displayed explicitly in the last line.

Using \eqref{Pi-expansion-L32} in the definition of $\mathbf B_2$ gives
\begin{equation}\label{B2}
\begin{aligned}
\mathbf B_2(t)\le{}&
-\lambda\frac{q_mn_m}{2}\int_0^1w^2\,dy
+\lambda\frac{q_mn_m}{6}\int_0^1w^3\,dy\\
&+C\lambda\delta_S\int_0^1w^2\,dy
+C\lambda(\delta_S+\chi_1)\int_0^1|w|^3\,dy\\
&+C\lambda\int_0^1
\left(|r_n|+|r_q|\right)
\left(w^2+|w|^3\right)\,dy .
\end{aligned}
\end{equation}
For $\mathbf B_3$, using \eqref{N''}, the positivity of $\tilde n$, and
\eqref{Pi-expansion-L32}, we obtain
\begin{equation}\label{B3}
\mathbf B_3(t)
\le
C\delta_S^2\int_0^1w^2\,dy
+C\delta_S^2\int_0^1|w|^3\,dy.
\end{equation}
Likewise,
\begin{equation}\label{4.G1}
\begin{aligned}
\mathbf G_1(t)\ge{}&
\frac{\lambda\sigma_mn_m}{2}\int_0^1w^2\,dy
-\frac{\lambda\sigma_mn_m}{6}\int_0^1w^3\,dy\\
&-C\lambda\delta_S\int_0^1w^2\,dy
-C\lambda(\delta_S+\chi_1)\int_0^1|w|^3\,dy\\
&-C\lambda\int_0^1
|r_n|\left(w^2+|w|^3\right)\,dy .
\end{aligned}
\end{equation}

Combining \eqref{B1}--\eqref{4.G1}, we obtain
\begin{align*}
&\mathbf B_1(t)+\mathbf B_2(t)+\mathbf B_3(t)-\mathbf G_1(t)\\
\le{}&
\frac{\lambda n_m}{2\sigma_m}
\left(n_m-q_m\sigma_m-\sigma_m^2\right)
\int_0^1w^2\,dy
+\frac{\delta_Sn_m}{\sigma_m}\int_0^1w^2\,dy\\
&+\frac{\lambda n_m}{2\sigma_m}
\left(n_m+\frac13q_m\sigma_m+\frac13\sigma_m^2\right)
\int_0^1w^3\,dy\\
&+C\delta_S(\lambda+\delta_S+\chi_1)\int_0^1w^2\,dy
+C\lambda(\delta_S+\chi_1)\int_0^1|w|^3\,dy\\
&+C\lambda\int_0^1
\left(|r_n|+|r_q|\right)
\left(w^2+|w|^3\right)\,dy .
\end{align*}
It follows directly from \eqref{-sigma} that
\begin{equation}\label{sigma-^2}
\sigma_m^2+q_m\sigma_m-n_m=0.
\end{equation}
Hence
$n_m-q_m\sigma_m-\sigma_m^2=0$ and
$n_m+\frac13q_m\sigma_m+\frac13\sigma_m^2=\frac43n_m$.
Therefore,
\begin{equation}\label{B1+B2+B3-G1}
\begin{aligned}
&\mathbf B_1(t)+\mathbf B_2(t)+\mathbf B_3(t)-\mathbf G_1(t)\\
\le{}&
\delta_S\frac{n_m}{\sigma_m}\int_0^1w^2\,dy
+\frac{2\lambda n_m^2}{3\sigma_m}\int_0^1w^3\,dy\\
&+C\delta_S(\lambda+\delta_S+\chi_1)\int_0^1w^2\,dy
+C\lambda(\delta_S+\chi_1)\int_0^1|w|^3\,dy\\
&+C\lambda\int_0^1
\left(|r_n|+|r_q|\right)
\left(w^2+|w|^3\right)\,dy .
\end{aligned}
\end{equation}

\emph{Step 4}. We estimate $\mathbf D(t)$. Since
$(\tilde n^R)_\xi<0$ and $\tilde n^R\to n_m$ as $\xi\to+\infty$,
we have $\tilde n^R\ge n_m$. Thus
$\tilde n=\tilde n^R+(\tilde n^S)^{-\mathbf X}-n_m
\ge(\tilde n^S)^{-\mathbf X}\ge n_+$.
Consequently, using the a priori bound on $w$,
\[
\inf_y\left(a^{-\mathbf X}\frac{\tilde n^2}{n}\right)
\ge n_+(1-C\chi_1)
\ge n_m-C(\delta_S+\chi_1).
\]
Moreover, by \eqref{y} and \eqref{nS'},
\[
y(1-y)
=-\frac{\sigma(\tilde n^S)^{-\mathbf X}_\xi}{\delta_S^2}
=\frac{\sigma}{\delta_S\lambda}a^{-\mathbf X}_\xi,
\]
and hence
\begin{equation}\label{dxi}
\frac{dy}{d\xi}
=\frac{a^{-\mathbf X}_\xi}{\lambda}
=\frac{\delta_S}{\sigma}y(1-y).
\end{equation}
Therefore,
\begin{equation}\label{4.D}
\begin{aligned}
\mathbf D(t)
&=\int_0^1
a^{-\mathbf X}\frac{\delta_S}{\sigma}
\frac{\tilde n^2}{n}
y(1-y)|\partial_yw|^2\,dy\\
&\ge
\delta_S\frac{n_m}{\sigma_m}
\left(1-C(\delta_S+\chi_1)\right)
\int_0^1y(1-y)|\partial_yw|^2\,dy .
\end{aligned}
\end{equation}

\emph{Step 5}. Combining \eqref{B1+B2+B3-G1} and \eqref{4.D},
and setting $\overline w:=\int_0^1w\,dy$, the 
Poincar\'e inequality in Lemma \ref{Poincare} yields
\begin{align}\label{B1+B2+B3-G1-D}
&\mathbf B_1(t)+\mathbf B_2(t)+\mathbf B_3(t)-\mathbf G_1(t)
-\frac34\mathbf D(t)\nonumber\\
\le{}&
\delta_S\left(
-\frac{n_m}{2\sigma_m}
+C(\lambda+\delta_S+\chi_1)
\right)\int_0^1w^2\,dy
+\frac{3\delta_Sn_m}{2\sigma_m}\overline w^{\,2}\nonumber\\
&+C\lambda\int_0^1|w|^3\,dy
+C\lambda\int_0^1
\left(|r_n|+|r_q|\right)
\left(w^2+|w|^3\right)\,dy .
\end{align}
Taking
\[
M=\frac{6\sigma_m^3n_m}{(\sigma_m^2+n_m)^2},
\]
the mean term in \eqref{B1+B2+B3-G1-D} is cancelled by the leading
negative term in \eqref{4.X^2}. Since $\lambda=\sqrt{\delta_S}$,
choosing $\delta_S$ and $\chi_1$ suitably small gives
\begin{equation}\label{L32-with-rw-remainder-y}
\begin{aligned}
&-\frac{\delta_S}{2M}|\dot{\mathbf X}(t)|^2
+\mathbf B_1(t)+\mathbf B_2(t)+\mathbf B_3(t)
-\mathbf G_1(t)-\frac34\mathbf D(t)\\
&\le
-C\delta_S\int_0^1w^2\,dy
+C\lambda\int_0^1|w|^3\,dy\\
&\quad
+C\delta_S
\left(\int_0^1|r_n|\,|w|\,dy\right)^2
+C\lambda\int_0^1
\left(|r_n|+|r_q|\right)
\left(w^2+|w|^3\right)\,dy .
\end{aligned}
\end{equation}

Finally, changing $y$ back to $\xi$ by \eqref{dy}, we obtain
\begin{equation}\label{L32-with-rw-remainder-xi}
\begin{aligned}
&-\frac{\delta_S}{2M}|\dot{\mathbf X}(t)|^2
+\mathbf B_1(t)+\mathbf B_2(t)+\mathbf B_3(t)
-\mathbf G_1(t)-\frac34\mathbf D(t)\\
&\le
-C\int_{\mathbb R}|(\tilde n^S)^{-\mathbf X}_\xi|
\left|\frac{n-\tilde n}{\tilde n}\right|^2\,d\xi
+C\int_{\mathbb R}a^{-\mathbf X}_\xi
\left|\frac{n-\tilde n}{\tilde n}\right|^3\,d\xi\\
&\quad
+\mathcal R_{\mathrm{rw},1}(t)+\mathcal R_{\mathrm{rw},2}(t),
\end{aligned}
\end{equation}
where $\mathcal R_{\mathrm{rw},1}$ and
$\mathcal R_{\mathrm{rw},2}$ are defined in \eqref{rw-remainders-L32}.
Thus the shock contribution has the desired coercive structure, while the
rarefaction-wave contributions are kept explicitly for the subsequent
interaction estimates.
\end{proof}					
	
In the subsequent energy estimates, the estimation of the interaction terms between the waves will be of critical importance. 
By virtue of Lemmas  \ref{properties of nRqR} and  \ref{properties of the 2-viscous shock wave}, in combination with  \eqref{4.supF}, we establish the following lemma.
\begin{lemma}[The wave interaction estimates]\label{Interaction estimates}
Let $\mathbf X$ be the shift defined by \eqref{X}. Under the same hypotheses
as in Proposition \ref{4 A priori estimates}, there exist constants
$C,c>0$, independent of $\delta_R,\delta_S$ and $t$, such that for every
$t\le T$,
\begin{equation}\label{interaction estimates}
\begin{aligned}
&\left\|(\tilde n^S)^{-\mathbf X}_\xi
(\tilde n^R-n_m)\right\|_{L^1}
+\left\|(\tilde n^R)_\xi
(\tilde n^S)^{-\mathbf X}_\xi\right\|_{L^1}
\le C\delta_R\delta_Se^{-c\delta_St},\\
&\left\|(\tilde n^S)^{-\mathbf X}_\xi
(\tilde n^R-n_m)\right\|_{L^2}
+\left\|(\tilde n^R)_\xi
(\tilde n^S)^{-\mathbf X}_\xi\right\|_{L^2}
\le C\delta_R\delta_S^{3/2}e^{-c\delta_St},\\
&\left\|(\tilde n^R)_\xi
\big((\tilde n^S)^{-\mathbf X}-n_m\big)\right\|_{L^2}
\le C\delta_R\delta_Se^{-c\delta_St}.
\end{aligned}
\end{equation}
The corresponding $q$-component estimates also hold:
\begin{equation}\label{interaction estimates q}
\begin{aligned}
&\left\|(\tilde q^S)^{-\mathbf X}_\xi
(\tilde q^R-q_m)\right\|_{L^1}
+\left\|(\tilde q^R)_\xi
(\tilde q^S)^{-\mathbf X}_\xi\right\|_{L^1}
\le C\delta_R\delta_Se^{-c\delta_St},\\
&\left\|(\tilde q^S)^{-\mathbf X}_\xi
(\tilde q^R-q_m)\right\|_{L^2}
+\left\|(\tilde q^R)_\xi
(\tilde q^S)^{-\mathbf X}_\xi\right\|_{L^2}
\le C\delta_R\delta_S^{3/2}e^{-c\delta_St},\\
&\left\|(\tilde q^R)_\xi
\big((\tilde q^S)^{-\mathbf X}-q_m\big)\right\|_{L^2}
\le C\delta_R\delta_Se^{-c\delta_St}.
\end{aligned}
\end{equation}
In particular, the mixed-component interactions used later satisfy
\begin{equation}\label{interaction estimates mixed}
\begin{aligned}
&\left\|(\tilde n^R)_\xi
\big((\tilde q^S)^{-\mathbf X}-q_m\big)\right\|_{L^2}
+\left\|(\tilde q^R)_\xi
\big((\tilde n^S)^{-\mathbf X}-n_m\big)\right\|_{L^2}\\
&\quad
+\left\|(\tilde q^S)^{-\mathbf X}_\xi
(\tilde q^R-q_m)\right\|_{L^2}
+\left\|(\tilde n^S)^{-\mathbf X}_\xi
(\tilde n^R-n_m)\right\|_{L^2}\\
&\le C\delta_R\delta_Se^{-c\delta_St}.
\end{aligned}
\end{equation}
All rarefaction profiles in these estimates are evaluated at
$(\xi+\sigma t,t)$.
\end{lemma}

\begin{proof}
For simplicity, set
\[
R(\xi,t):=\tilde n^R(\xi+\sigma t,t),\qquad
S(\xi,t):=(\tilde n^S)^{-\mathbf X}(\xi)
=\tilde n^S(\xi-\mathbf X(t)).
\]
Thus the three interaction quantities in \eqref{interaction estimates} are
$S_\xi(R-n_m)$, $R_\xi S_\xi$, and
$R_\xi(S-n_m)$.

We first record the separation of the two component waves. By the definition
\eqref{X}, the a priori assumption in Proposition \ref{4 A priori estimates}, the bound
$|\varphi(n)|\le C|n-\tilde n|$, and
$\int_{\mathbb R}a^{-\mathbf X}_\xi\,d\xi=\lambda$, we have
\[
\begin{aligned}
|\dot{\mathbf X}(t)|
&\le \frac{CM}{\lambda}
\int_{\mathbb R}a^{-\mathbf X}_\xi
\left(
\left|\frac{n-\tilde n}{\tilde n}\right|
+|\varphi(n)|
\right)d\xi\\
&\le C\chi_1 .
\end{aligned}
\]
Since $\mathbf X(0)=0$, it follows that
\begin{equation}\label{interaction-X-small}
|\mathbf X(t)|\le C\chi_1 t.
\end{equation}
Taking $\chi_1$ sufficiently small, and recalling from Lemma
\ref{properties of the 2-viscous shock wave} that
$\sigma\ge \sigma_m/2>0$, we may assume
\begin{equation}\label{interaction-X-separation}
|\mathbf X(t)|\le \frac{\sigma}{4}t,\qquad 0\le t\le T.
\end{equation}

We divide the real line into
\[
\Omega_-(t):=\left\{\xi\le-\frac{\sigma t}{2}\right\},
\qquad
\Omega_+(t):=\left\{\xi\ge-\frac{\sigma t}{2}\right\}.
\]
If $\xi\in\Omega_+(t)$, then
$\xi+\sigma t\ge \sigma t/2\ge0$. Hence Lemma
\ref{properties of nRqR}-(4) gives, for some $c>0$ independent of
$\delta_R,\delta_S$ and $t$,
\begin{equation}\label{interaction-R-tail}
\|R-n_m\|_{L^\infty(\Omega_+)}
+\|R_\xi\|_{L^\infty(\Omega_+)}
+\|R_\xi\|_{L^2(\Omega_+)}
\le C\delta_R e^{-ct}.
\end{equation}
Here the $L^2$ estimate follows directly by integrating the pointwise
exponential bound in Lemma \ref{properties of nRqR}-(4).
Moreover, by the monotonicity of the rarefaction profile and Lemma
\ref{properties of nRqR}-(2),
\begin{equation}\label{interaction-R-global}
\|R-n_m\|_{L^\infty(\mathbb R)}
+\|R_\xi\|_{L^1(\mathbb R)}
+\|R_\xi\|_{L^2(\mathbb R)}
\le C\delta_R.
\end{equation}

On the other hand, if $\xi\in\Omega_-(t)$, then by
\eqref{interaction-X-separation},
\[
\xi-\mathbf X(t)
\le -\frac{\sigma t}{2}+|\mathbf X(t)|
\le -\frac{\sigma t}{4}.
\]
Therefore Lemma \ref{properties of the 2-viscous shock wave}, in particular
\eqref{N'} and the left-tail estimate therein, yields
\begin{equation}\label{interaction-S-tail}
\begin{aligned}
\|S_\xi\|_{L^1(\Omega_-)}
&\le C\delta_S e^{-c\delta_St},\\
\|S_\xi\|_{L^2(\Omega_-)}
&\le C\delta_S^{3/2}e^{-c\delta_St},\\
\|S_\xi\|_{L^\infty(\Omega_-)}
&\le C\delta_S^2e^{-c\delta_St},\\
\|S-n_m\|_{L^\infty(\Omega_-)}
&\le C\delta_S e^{-c\delta_St}.
\end{aligned}
\end{equation}
Since the shift does not change spatial norms and the shock profile is
monotone, we also have
\begin{equation}\label{interaction-S-global}
\|S_\xi\|_{L^1(\mathbb R)}=\delta_S,\qquad
\|S_\xi\|_{L^2(\mathbb R)}\le C\delta_S^{3/2},\qquad
\|S-n_m\|_{L^\infty(\mathbb R)}\le C\delta_S.
\end{equation}

We now prove the first estimate in \eqref{interaction estimates}. By
\eqref{interaction-R-global} and \eqref{interaction-S-tail},
\[
\begin{aligned}
\|S_\xi(R-n_m)\|_{L^1(\Omega_-)}
&\le
\|R-n_m\|_{L^\infty(\mathbb R)}
\|S_\xi\|_{L^1(\Omega_-)}\\
&\le C\delta_R\delta_Se^{-c\delta_St}.
\end{aligned}
\]
On $\Omega_+(t)$, \eqref{interaction-R-tail} and
\eqref{interaction-S-global} give
\[
\begin{aligned}
\|S_\xi(R-n_m)\|_{L^1(\Omega_+)}
&\le
\|R-n_m\|_{L^\infty(\Omega_+)}
\|S_\xi\|_{L^1(\mathbb R)}\\
&\le C\delta_R\delta_Se^{-ct}
\le C\delta_R\delta_Se^{-c\delta_St},
\end{aligned}
\]
where we used $\delta_S\le1$ in the last inequality. Consequently,
\begin{equation}\label{interaction-L1-first}
\|S_\xi(R-n_m)\|_{L^1(\mathbb R)}
\le C\delta_R\delta_Se^{-c\delta_St}.
\end{equation}
Likewise, on $\Omega_-(t)$,
\[
\begin{aligned}
\|R_\xi S_\xi\|_{L^1(\Omega_-)}
&\le
\|S_\xi\|_{L^\infty(\Omega_-)}
\|R_\xi\|_{L^1(\mathbb R)}\\
&\le
C\delta_R\delta_S^2e^{-c\delta_St}
\le C\delta_R\delta_Se^{-c\delta_St},
\end{aligned}
\]
while on $\Omega_+(t)$,
\[
\begin{aligned}
\|R_\xi S_\xi\|_{L^1(\Omega_+)}
&\le
\|R_\xi\|_{L^\infty(\Omega_+)}
\|S_\xi\|_{L^1(\mathbb R)}\\
&\le C\delta_R\delta_Se^{-ct}
\le C\delta_R\delta_Se^{-c\delta_St}.
\end{aligned}
\]
Together with \eqref{interaction-L1-first}, this proves the first line of
\eqref{interaction estimates}.

For the $L^2$ estimates, using
\eqref{interaction-R-global}--\eqref{interaction-S-global}, we obtain
\[
\begin{aligned}
\|S_\xi(R-n_m)\|_{L^2(\Omega_-)}
&\le
\|R-n_m\|_{L^\infty(\mathbb R)}
\|S_\xi\|_{L^2(\Omega_-)}
\le C\delta_R\delta_S^{3/2}e^{-c\delta_St},\\
\|S_\xi(R-n_m)\|_{L^2(\Omega_+)}
&\le
\|R-n_m\|_{L^\infty(\Omega_+)}
\|S_\xi\|_{L^2(\mathbb R)}
\le C\delta_R\delta_S^{3/2}e^{-c\delta_St}.
\end{aligned}
\]
Thus
\begin{equation}\label{interaction-L2-first}
\|S_\xi(R-n_m)\|_{L^2(\mathbb R)}
\le C\delta_R\delta_S^{3/2}e^{-c\delta_St}.
\end{equation}
Similarly,
\[
\begin{aligned}
\|R_\xi S_\xi\|_{L^2(\Omega_-)}
&\le
\|S_\xi\|_{L^\infty(\Omega_-)}
\|R_\xi\|_{L^2(\mathbb R)}
\le C\delta_R\delta_S^2e^{-c\delta_St}\\
&\le C\delta_R\delta_S^{3/2}e^{-c\delta_St},\\
\|R_\xi S_\xi\|_{L^2(\Omega_+)}
&\le
\|R_\xi\|_{L^\infty(\Omega_+)}
\|S_\xi\|_{L^2(\mathbb R)}
\le C\delta_R\delta_S^{3/2}e^{-c\delta_St}.
\end{aligned}
\]
Combining this with \eqref{interaction-L2-first} proves the second line of
\eqref{interaction estimates}.

Finally, for the last interaction term, we use the exponentially small shock
tail on $\Omega_-$ and the exponentially small rarefaction tail on
$\Omega_+$. Namely,
\[
\begin{aligned}
\|R_\xi(S-n_m)\|_{L^2(\Omega_-)}
&\le
\|S-n_m\|_{L^\infty(\Omega_-)}
\|R_\xi\|_{L^2(\mathbb R)}
\le C\delta_R\delta_Se^{-c\delta_St},\\
\|R_\xi(S-n_m)\|_{L^2(\Omega_+)}
&\le
\|S-n_m\|_{L^\infty(\mathbb R)}
\|R_\xi\|_{L^2(\Omega_+)}
\le C\delta_R\delta_Se^{-ct}\\
&\le C\delta_R\delta_Se^{-c\delta_St}.
\end{aligned}
\]
This gives the third line of \eqref{interaction estimates}.

It remains to record the $q$-component consequences. From the second
traveling-wave equation in \eqref{traveling wave},
\begin{equation}\label{shock-q-n-relation-interaction}
(\tilde q^S)^{-\mathbf X}_\xi
=-\frac1{\sigma}(\tilde n^S)^{-\mathbf X}_\xi,
\qquad
(\tilde q^S)^{-\mathbf X}-q_m
=-\frac1{\sigma}
\big((\tilde n^S)^{-\mathbf X}-n_m\big).
\end{equation}
On the other hand, along the $1$-rarefaction curve,
\[
\binom{(\tilde n^R)_\xi}{(\tilde q^R)_\xi}
=\kappa(\xi,t)\,
r_1(\tilde n^R,\tilde q^R)
=\kappa(\xi,t)
\binom{-\lambda_1(\tilde n^R,\tilde q^R)}{1}.
\]
Since the rarefaction profile remains in a fixed compact subset of
$\{n>0\}$, $\lambda_1$ is bounded away from zero there. Consequently,
\begin{equation}\label{rare-q-n-comparable-interaction}
|(\tilde q^R)_\xi|
\le C|(\tilde n^R)_\xi|,
\qquad
|\tilde q^R-q_m|
\le C|\tilde n^R-n_m|.
\end{equation}
The second inequality follows by integrating
$dq/dn=-1/\lambda_1(n,q)$ along the rarefaction curve.

Combining
\eqref{shock-q-n-relation-interaction}--\eqref{rare-q-n-comparable-interaction}
with \eqref{interaction estimates} gives
\eqref{interaction estimates q}. The same relations, together with
$\delta_S^{3/2}\le\delta_S$ for $0<\delta_S\le1$, also give
\eqref{interaction estimates mixed}. This completes the proof.
\end{proof}

\begin{lemma}[Rarefaction modulation in the weighted relative entropy]
\label{rarefaction-modulation-L34}
Under the hypotheses of Proposition \ref{4 A priori estimates}, set
\[
u:=n-\tilde n,\qquad v:=q-\tilde q,\qquad
A:=a^{-\mathbf X},
\]
and
\[
N(\xi,t):=\tilde n^R(\xi+\sigma t,t),\qquad
Q(\xi,t):=\tilde q^R(\xi+\sigma t,t),
\]
\[
S(\xi,t):=(\tilde n^S)^{-\mathbf X}(\xi),\qquad
T(\xi,t):=(\tilde q^S)^{-\mathbf X}(\xi).
\]
Let
\[
\Lambda(\xi,t):=\lambda_1(N,Q),\qquad
b(\xi,t):=\frac{N(\xi,t)}{n_m}.
\]
Then
\begin{equation}\label{rare-rel-L34-int}
N_\xi=-\Lambda Q_\xi,\qquad
N_t-\sigma N_\xi=\Lambda^2Q_\xi,
\end{equation}
and, for $\delta_0$ sufficiently small,
\begin{equation}\label{lambda-bounds-L34-int}
0<c_\Lambda\le-\Lambda(\xi,t)\le C_\Lambda.
\end{equation}
In particular,
\begin{equation}\label{Nxi-infty-L34}
\|N_\xi(t)\|_{L^\infty}\le C\delta_R.
\end{equation}

Define the correction generated by the derivatives of $b$ by
\begin{equation}\label{Qb-L34}
\begin{aligned}
\mathcal Q_b(t)
={}&
\int_{\mathbb R}A b_t
\left(\Pi(n|\tilde n)+\frac{v^2}{2}\right)d\xi\\
&-\int_{\mathbb R}A b_\xi\tilde q\,\Pi(n|\tilde n)\,d\xi
-\sigma\int_{\mathbb R}A b_\xi\Pi(n|\tilde n)\,d\xi\\
&-\int_{\mathbb R}A b_\xi
\big(\Pi(n|\tilde n)+u\big)v\,d\xi
-\frac{\sigma}{2}\int_{\mathbb R}A b_\xi v^2\,d\xi\\
&-\int_{\mathbb R}A b_\xi
n\ln\frac{n}{\tilde n}\,
\partial_\xi\left(\ln\frac{n}{\tilde n}\right)d\xi .
\end{aligned}
\end{equation}
Moreover, with
\[
\mathbf B_9^{(b)}
:=
\int_{\mathbb R}
A b\,\frac{N_\xi}{\tilde n}\,uv\,d\xi,
\]
one has the exact identity
\begin{equation}\label{B9-Qb-L34}
\mathbf B_9^{(b)}+\mathcal Q_b
=
-\mathcal G_R
+\sum_{j=1}^{4}\mathcal R_{R,j},
\end{equation}
where
\begin{equation}\label{GR-def-L34-int}
\mathcal G_R(t):=
\frac1{n_m}\int_{\mathbb R}A
\left[
N|Q_\xi|\Pi(n|\tilde n)
+\frac{|N_t-\sigma N_\xi|}{2}v^2
\right]d\xi
\end{equation}
and
\begin{align}
\mathcal R_{R,1}
&:=
\frac1{n_m}\int_{\mathbb R}A
(Q-\tilde q)N_\xi\Pi(n|\tilde n)\,d\xi,
\label{RR1-L34}\\
\mathcal R_{R,2}
&:=
-\frac1{n_m}\int_{\mathbb R}A
N_\xi\Pi(n|\tilde n)v\,d\xi,
\label{RR2-L34}\\
\mathcal R_{R,3}
&:=
\frac1{n_m}\int_{\mathbb R}A
N_\xi\left(\frac{N}{\tilde n}-1\right)uv\,d\xi,
\label{RR3-L34}\\
\mathcal R_{R,4}
&:=
-\frac1{n_m}\int_{\mathbb R}A
N_\xi n\ln\frac{n}{\tilde n}\,
\partial_\xi\left(\ln\frac{n}{\tilde n}\right)d\xi.
\label{RR4-L34}
\end{align}
Furthermore,
\begin{equation}\label{GR-weighted-control-L34}
\int_{\mathbb R}A|N_\xi|(u^2+v^2)\,d\xi
\le C\mathcal G_R(t),
\end{equation}
and, for every $\varepsilon>0$,
\begin{equation}\label{RR24-absorb-L34}
|\mathcal R_{R,2}(t)|+|\mathcal R_{R,4}(t)|
\le
\varepsilon\mathbf D(t)
+C\left(\chi_1+C_\varepsilon\delta_R\right)\mathcal G_R(t).
\end{equation}
The remaining two terms satisfy
\begin{align}
|\mathcal R_{R,1}(t)|
&\le
C\chi_1^2\delta_R\delta_Se^{-c\delta_St},
\label{RR1-est-L34}\\
|\mathcal R_{R,3}(t)|
&\le
C\chi_1^2\delta_R\delta_Se^{-c\delta_St},
\label{RR3-est-L34}
\end{align}
and hence
\begin{equation}\label{RR13-time-L34}
\int_0^t
\left(
|\mathcal R_{R,1}(\tau)|
+|\mathcal R_{R,3}(\tau)|
\right)d\tau
\le C\chi_1^2\delta_R.
\end{equation}
\end{lemma}

\begin{proof}
The identities in \eqref{rare-rel-L34-int} follow from
Lemma \ref{properties of nRqR}.  Since the rarefaction profile stays in a
fixed compact subset of $\{n>0\}$, $\Lambda$ is bounded away from zero,
which gives \eqref{lambda-bounds-L34-int}; \eqref{Nxi-infty-L34} follows
from the profile estimates in the same lemma.

We next prove \eqref{B9-Qb-L34}.  Since
\[
b_\xi=\frac{N_\xi}{n_m},\qquad
b_t-\sigma b_\xi
=
\frac{N_t-\sigma N_\xi}{n_m},
\]
the terms in \eqref{Qb-L34} can be rearranged as
\[
\begin{aligned}
\mathcal Q_b
=
\frac1{n_m}\int_{\mathbb R}A\Big[
&(N_t-\sigma N_\xi)
\left(\Pi(n|\tilde n)+\frac{v^2}{2}\right)
-N_\xi\tilde q\,\Pi(n|\tilde n)\\
&-N_\xi\Pi(n|\tilde n)v
-N_\xi uv
-N_\xi n\ln\frac n{\tilde n}
\partial_\xi\left(\ln\frac n{\tilde n}\right)
\Big]d\xi.
\end{aligned}
\]
On the other hand,
\[
\mathbf B_9^{(b)}
=
\frac1{n_m}\int_{\mathbb R}
A\frac{NN_\xi}{\tilde n}uv\,d\xi.
\]
Therefore
\[
\begin{aligned}
\mathbf B_9^{(b)}+\mathcal Q_b
=
\frac1{n_m}\int_{\mathbb R}A\Big[
&(N_t-\sigma N_\xi)\Pi
-N_\xi\tilde q\,\Pi
-N_\xi\Pi v\\
&+\frac12(N_t-\sigma N_\xi)v^2
+N_\xi\left(\frac N{\tilde n}-1\right)uv\\
&-N_\xi n\ln\frac n{\tilde n}
\partial_\xi\left(\ln\frac n{\tilde n}\right)
\Big]d\xi.
\end{aligned}
\]
Using the exact rarefaction equation
\[
N_t-\sigma N_\xi=(NQ)_\xi=QN_\xi+NQ_\xi,
\]
we have
\[
(N_t-\sigma N_\xi)\Pi
-N_\xi\tilde q\,\Pi
-N_\xi\Pi v
=
NQ_\xi\Pi
+(Q-\tilde q)N_\xi\Pi
-N_\xi\Pi v.
\]
Since $Q_\xi<0$ and $N_t-\sigma N_\xi<0$,
\[
NQ_\xi\Pi=-N|Q_\xi|\Pi,\qquad
\frac12(N_t-\sigma N_\xi)v^2
=
-\frac12|N_t-\sigma N_\xi|v^2.
\]
This proves the exact decomposition
\eqref{B9-Qb-L34}--\eqref{RR4-L34}.

Under the a priori assumption,
\[
A\sim1,\qquad n\sim\tilde n\sim N\sim1,\qquad
\Pi(n|\tilde n)\sim u^2,\qquad
\left|n\ln\frac n{\tilde n}\right|\le C|u|.
\]
Together with
$|N_\xi|\sim|Q_\xi|$ and
$|N_t-\sigma N_\xi|=|\Lambda|\,|N_\xi|$,
these relations give \eqref{GR-weighted-control-L34}.

For $\mathcal R_{R,2}$, Sobolev's inequality and the a priori bound give
$\|v\|_{L^\infty}\le C\chi_1$, and hence
\[
|\mathcal R_{R,2}|
\le
C\chi_1
\int_{\mathbb R}A|N_\xi|\Pi(n|\tilde n)d\xi
\le
C\chi_1\mathcal G_R.
\]
For $\mathcal R_{R,4}$, Cauchy--Schwarz,
\eqref{GR-weighted-control-L34}, and \eqref{Nxi-infty-L34} yield
\[
\begin{aligned}
|\mathcal R_{R,4}|
&\le
C\int_{\mathbb R}A|N_\xi||u|
\left|\partial_\xi\ln\frac n{\tilde n}\right|d\xi\\
&\le
C\mathbf D^{1/2}
\left(
\int_{\mathbb R}A|N_\xi|^2u^2d\xi
\right)^{1/2}\\
&\le
C\sqrt{\delta_R}\,
\mathbf D^{1/2}\mathcal G_R^{1/2}\\
&\le
\varepsilon\mathbf D
+C_\varepsilon\delta_R\mathcal G_R.
\end{aligned}
\]
This proves \eqref{RR24-absorb-L34}.

Finally,
\[
Q-\tilde q=-(T-q_m),\qquad
\frac N{\tilde n}-1
=-\frac{S-n_m}{\tilde n}.
\]
By Lemma \ref{Interaction estimates},
\[
\|N_\xi(T-q_m)\|_{L^2}
+\|N_\xi(S-n_m)\|_{L^2}
\le
C\delta_R\delta_Se^{-c\delta_St}.
\]
Moreover,
\[
\|\Pi(n|\tilde n)\|_{L^2}
\le C\|u\|_{L^\infty}\|u\|
\le C\chi_1^2,
\qquad
\|uv\|_{L^2}
\le C\chi_1^2.
\]
The estimates \eqref{RR1-est-L34}--\eqref{RR3-est-L34} and
\eqref{RR13-time-L34} follow immediately.
\end{proof}

\subsection{$L^2$ estimate}
This section is dedicated to establish the $L^2$ estimate for $(n-\tilde{n},q-\tilde{q})$.

\begin{lemma}\label{4.L2 estimate}
Under the hypotheses of Proposition \ref{4 A priori estimates}, there exist
positive constants $\delta_{*}$ and $C$, independent of $\delta_0$, $\chi_1$
and $T$, such that, if $0<\delta_S<\delta_{*}$, then for every
$t\in[0,T]$,
\begin{equation}\label{l2 estimate}
\begin{aligned}
&\int_{\mathbb R}|n-\tilde n|^2\,d\xi
+\int_{\mathbb R}|q-\tilde q|^2\,d\xi
+\delta_S\int_0^t|\dot{\mathbf X}(\tau)|^2\,d\tau\\
&\quad
+\int_0^t\int_{\mathbb R}
\left|(\tilde n^S)_{\xi}^{-\mathbf X}\right|
|n-\tilde n|^2\,d\xi d\tau\\
&\quad
+\frac{\lambda}{\delta_S}
\int_0^t\int_{\mathbb R}
\left|(\tilde n^S)_{\xi}^{-\mathbf X}\right|
\left|q-\tilde q+\varphi(n)\right|^2\,d\xi d\tau\\
&\quad
+\int_0^t\int_{\mathbb R}
\left|(\tilde q^R)_{\xi}\right|
\left(
|n-\tilde n|^2
+|\lambda_1(\tilde n^R,\tilde q^R)|^2
|q-\tilde q|^2
\right)d\xi d\tau\\
&\quad
+\int_0^t\int_{\mathbb R}
\left|\partial_\xi(n-\tilde n)\right|^2\,d\xi d\tau\\
&\le
C\left(
\|n_0-\tilde n(\cdot,0)\|^2
+\|q_0-\tilde q(\cdot,0)\|^2
\right)
+C\delta_R^{1/3}.
\end{aligned}
\end{equation}
In the rarefaction terms,
$(\tilde n^R,\tilde q^R)$ is evaluated at $(\xi+\sigma\tau,\tau)$.
\end{lemma}

\begin{proof}
Set
\[
u:=n-\tilde n,\qquad
v:=q-\tilde q,\qquad
w:=\frac{u}{\tilde n},
\]
and introduce
\[
N(\xi,t):=\tilde n^R(\xi+\sigma t,t),\qquad
Q(\xi,t):=\tilde q^R(\xi+\sigma t,t),
\]
\[
S(\xi,t):=(\tilde n^S)^{-\mathbf X}(\xi),\qquad
T(\xi,t):=(\tilde q^S)^{-\mathbf X}(\xi).
\]
Thus
\[
\tilde n=N+S-n_m,\qquad
\tilde q=Q+T-q_m.
\]
Let
\[
\Lambda(\xi,t):=\lambda_1(N,Q).
\]
The exact rarefaction relations and the uniform bounds for $\Lambda$ are
recorded in Lemma \ref{rarefaction-modulation-L34}; in particular,
\eqref{rare-rel-L34-int}--\eqref{lambda-bounds-L34-int} hold.

We use the following rarefaction-weighted 
relative entropy
\begin{equation}\label{Em-L34-int}
b(\xi,t):=\frac{N(\xi,t)}{n_m},\qquad
\mathcal E_m(t):=
\int_{\mathbb R}a^{-\mathbf X}b
\left(
\Pi(n|\tilde n)+\frac{v^2}{2}
\right)d\xi.
\end{equation}
Since $a^{-\mathbf X}$, $b$, $n$ and $\tilde n$ are uniformly bounded above
and below under the a priori assumption,
\begin{equation}\label{Em-equivalence-L34-int}
c\big(\|u(t)\|^2+\|v(t)\|^2\big)
\le \mathcal E_m(t)
\le C\big(\|u(t)\|^2+\|v(t)\|^2\big).
\end{equation}

\emph{Step 1. Exact modulated relative entropy identity.}
Apply the calculation of Lemma \ref{relative entropy} with the multiplier
\[
h(\xi,t):=a^{-\mathbf X}(\xi)b(\xi,t).
\]
Since
\[
h_t=-\dot{\mathbf X}a^{-\mathbf X}_\xi b
+a^{-\mathbf X}b_t,\qquad
h_\xi=a^{-\mathbf X}_\xi b+a^{-\mathbf X}b_\xi,
\]
the part containing $a^{-\mathbf X}_\xi b$ has exactly the same algebraic
structure as Lemma \ref{relative entropy}. Denote by
$\mathbf Y^{(b)},\mathbf B_i^{(b)},\mathbf G_j^{(b)}$ and
$\mathbf D^{(b)}$ the quantities
$\mathbf Y,\mathbf B_i,\mathbf G_j,\mathbf D$, respectively, with an
additional factor $b$ in each integrand. Then the calculation gives the
exact identity
\begin{equation}\label{mod-RE-identity-L34}
\begin{aligned}
\frac{d}{dt}\mathcal E_m(t)
={}&
\dot{\mathbf X}\mathbf Y^{(b)}
+\sum_{i=1}^{9}\mathbf B_i^{(b)}
-\mathbf G_1^{(b)}-\mathbf G_2^{(b)}-\mathbf D^{(b)}
+\mathcal Q_b(t),
\end{aligned}
\end{equation}
where $\mathcal Q_b$ is the correction defined in
\eqref{Qb-L34}.  The formula follows directly from the six terms
$I_1,\ldots,I_6$ in the proof of Lemma \ref{relative entropy}: every
occurrence of $h_\xi$ is split into
$a^{-\mathbf X}_\xi b+a^{-\mathbf X}b_\xi$.

The point at which the rarefaction differs from the barotropic
Navier--Stokes case is precisely the term $\mathbf B_9^{(b)}+\mathcal Q_b$.
By Lemma \ref{rarefaction-modulation-L34},
\begin{equation*}
\mathbf B_9^{(b)}+\mathcal Q_b
=
-\mathcal G_R+\sum_{j=1}^{4}\mathcal R_{R,j},
\end{equation*}
with $\mathcal G_R$ and $\mathcal R_{R,j}$ given by
\eqref{GR-def-L34-int}--\eqref{RR4-L34}.

Let
\[
\rho:=b-1=\frac{N-n_m}{n_m}.
\]
Subtracting the corresponding unmodulated quantities, \eqref{mod-RE-identity-L34}
and \eqref{B9-Qb-L34} give the exact identity
\begin{equation}\label{exact-split-L34}
\begin{aligned}
\frac{d}{dt}\mathcal E_m
={}&
\left[
\dot{\mathbf X}\mathbf Y
+\sum_{i=1}^{8}\mathbf B_i
-\mathbf G_1-\mathbf G_2-\mathbf D
\right]
-\mathcal G_R\\
&+\mathcal R_{\mathrm{mod}}
+\sum_{j=1}^{4}\mathcal R_{R,j},
\end{aligned}
\end{equation}
where
\begin{equation}\label{Rmod-def-L34}
\begin{aligned}
\mathcal R_{\mathrm{mod}}
:={}&
\dot{\mathbf X}\big(\mathbf Y^{(b)}-\mathbf Y\big)
+\sum_{i=1}^{8}\big(\mathbf B_i^{(b)}-\mathbf B_i\big)\\
&-\big(\mathbf G_1^{(b)}-\mathbf G_1\big)
-\big(\mathbf G_2^{(b)}-\mathbf G_2\big)
-\big(\mathbf D^{(b)}-\mathbf D\big).
\end{aligned}
\end{equation}

\emph{Step 2. The shock-contraction part.}
Using \eqref{X'Y}, Lemma \ref{Leading order estimates}, the corrected signed
definitions of $\mathbf Y_3,\ldots,\mathbf Y_7$, and the same estimates for
$\mathbf B_4,\ldots,\mathbf B_8$ as in the shock-contraction argument, we
obtain, for some $c_0>0$,
\begin{equation}\label{shock-block-L34}
\begin{aligned}
&\dot{\mathbf X}\mathbf Y
+\sum_{i=1}^{8}\mathbf B_i
-\mathbf G_1-\mathbf G_2-\mathbf D\\
&\le
-c_0\left[
\delta_S|\dot{\mathbf X}|^2
+G^s(t)+\mathbf G_2(t)+\mathbf D(t)
\right]
+\mathcal R_{\mathrm{rw},1}(t)
+\mathcal R_{\mathrm{rw},2}(t)
+\mathcal S_{\mathrm{sh}}(t),
\end{aligned}
\end{equation}
where
\[
G^s(t):=
\int_{\mathbb R}|S_\xi|
\left|\frac{u}{\tilde n}\right|^2d\xi
\]
and
\begin{equation}\label{Ssh-L34}
\begin{aligned}
\mathcal S_{\mathrm{sh}}(t)
\le{}&
C\chi_1^{1/2}\|N_\xi\|^4
+C\chi_1^{4/3}\|N_{\xi\xi}\|^{4/3}
+C\chi_1^{2/3}\|N_{\xi\xi}\|_{L^1}^{4/3}\\
&+C\chi_1\,\mathcal I_{RS}(t).
\end{aligned}
\end{equation}
Here
\begin{equation}\label{IRS-L34-int}
\begin{aligned}
\mathcal I_{RS}(t):=
\bigg\|&
|N_\xi(T-q_m)|
+\frac1{\sigma}|S_\xi(N-n_m)|\\
&+\sigma|T_\xi(Q-q_m)|
+|N_\xi(S-n_m)|
\bigg\|.
\end{aligned}
\end{equation}
For example, the cubic shock term is estimated by
\[
\begin{aligned}
\int_{\mathbb R}a^{-\mathbf X}_\xi|w|^3d\xi
&=
\frac{\lambda}{\delta_S}
\int_{\mathbb R}|S_\xi||w|^3d\xi\\
&\le
\frac{\lambda}{\sqrt{\delta_S}}
\|w\|_{L^\infty}^2\sqrt{G^s(t)}
\le
C\chi_1\|w_\xi\|\sqrt{G^s(t)}\\
&\le
\varepsilon\mathbf D(t)+C_\varepsilon\chi_1^2G^s(t),
\end{aligned}
\]
where $\lambda=\sqrt{\delta_S}$, the Sobolev inequality, and
$\int|S_\xi|d\xi=\delta_S$ have been used. The estimates of
$\mathbf Y_3,\ldots,\mathbf Y_7$ are unchanged by the correction made in
Lemma \ref{relative entropy}, because absolute values are used only after
the exact signed decomposition of $\mathbf Y$.

The first rarefaction remainder in Lemma \ref{Leading order estimates} is
integrable by Lemma \ref{Interaction estimates}:
\begin{equation}\label{Rrw1-time-L34-int}
\begin{aligned}
\mathcal R_{\mathrm{rw},1}(t)
&\le
\frac{C}{\delta_S}\|w\|_{L^\infty}^2
\|S_\xi(N-n_m)\|_{L^1}^2\\
&\le
C\chi_1^2\delta_R^2\delta_S e^{-c\delta_St},
\end{aligned}
\end{equation}
and consequently
\begin{equation}\label{Rrw1-int-L34-int}
\int_0^t\mathcal R_{\mathrm{rw},1}(\tau)d\tau
\le C\chi_1^2\delta_R^2.
\end{equation}

We next estimate $\mathcal R_{\mathrm{rw},2}$ without losing the large
factor $\lambda/\delta_S$. By the wave separation proved in
Lemma \ref{Interaction estimates}, define
\begin{equation}\label{Gamma-L34-int}
\Gamma_S(t):=
C\delta_R\delta_S^{3/2}
\left(e^{-ct}+e^{-c\delta_St}\right).
\end{equation}
The same splitting into
$\Omega_-(t)=\{\xi\le-\sigma t/2\}$ and
$\Omega_+(t)=\{\xi\ge-\sigma t/2\}$ gives
\begin{equation}\label{Linfty-interaction-L34-int}
\left\|
a^{-\mathbf X}_\xi
\left(|N-n_m|+|Q-q_m|\right)
\right\|_{L^\infty}
\le \Gamma_S(t).
\end{equation}
Hence, using $|w|^3\le C\chi_1|w|^2$,
\begin{equation}\label{Rrw2-L34-int}
\mathcal R_{\mathrm{rw},2}(t)
\le C\Gamma_S(t)\mathcal E_m(t),
\end{equation}
and
\begin{equation}\label{Gamma-time-L34-int}
\int_0^\infty\Gamma_S(\tau)d\tau
\le C\delta_R\sqrt{\delta_S}.
\end{equation}

Finally, Lemma \ref{properties of nRqR} gives
\begin{equation}\label{nRL1}
\begin{aligned}
&\int_0^\infty\|N_\xi\|^4\,d\tau\le C\delta_R^3,\\
&\int_0^\infty\|N_{\xi\xi}\|^{4/3}\,d\tau
+\int_0^\infty\|N_{\xi\xi}\|_{L^1}^{4/3}\,d\tau
\le C\delta_R^{1/3}.
\end{aligned}
\end{equation}
Moreover, Lemma \ref{Interaction estimates}, including
\eqref{interaction estimates q}--\eqref{interaction estimates mixed},
implies
\begin{equation}\label{nSnR}
\mathcal I_{RS}(t)
\le C\delta_R\delta_Se^{-c\delta_St},
\end{equation}
and hence
\begin{equation}\label{nSnR'}
\int_0^\infty\mathcal I_{RS}(\tau)d\tau
\le C\delta_R.
\end{equation}
Consequently,
\begin{equation}\label{Ssh-time-L34-int}
\int_0^t\mathcal S_{\mathrm{sh}}(\tau)d\tau
\le C\delta_R^{1/3}.
\end{equation}

\emph{Step 3. Estimate of the modulation remainder.}
We now estimate every term in $\mathcal R_{\mathrm{mod}}$ appearing in
\eqref{exact-split-L34}.  For brevity, set
\[
A:=a^{-\mathbf X},\qquad
L:=\ln\frac n{\tilde n},\qquad
z:=v+\varphi(n),\qquad
\rho:=b-1=\frac{N-n_m}{n_m}.
\]
Under the a priori assumption, uniformly on $[0,T]$,
\begin{equation}\label{basic-mod-bounds-L34}
|\rho|\le C\delta_R,\qquad
A\sim1,\qquad
n\sim\tilde n\sim N\sim1,
\end{equation}
and
\begin{equation}\label{nonlinear-equivalence-L34}
\Pi(n|\tilde n)\sim u^2,\qquad
|nL|\le C|u|,\qquad
|\varphi(n)|\le C|u|,\qquad
|z|^2\le C(u^2+v^2).
\end{equation}
Moreover, by \eqref{Linfty-interaction-L34-int},
\begin{equation}\label{rho-localized-L34}
\|\rho A_\xi\|_{L^\infty}
\le C\Gamma_S(t).
\end{equation}
Since
\[
A_\xi=\frac{\lambda}{\delta_S}|S_\xi|,
\qquad \lambda=\sqrt{\delta_S},
\qquad
|S_{\xi\xi}|\le C\delta_S|S_\xi|,
\]
we also have
\begin{equation}\label{rho-shock-localized-L34}
\|\rho A S_\xi\|_{L^\infty}
\le C\sqrt{\delta_S}\,\Gamma_S(t),
\qquad
\|\rho A S_{\xi\xi}\|_{L^\infty}
\le C\delta_S^{3/2}\Gamma_S(t).
\end{equation}
After decreasing $\delta_0$ if necessary, we may assume
$\sup_{t\ge0}\Gamma_S(t)\le1$.

We shall repeatedly use the rarefaction controls
\eqref{GR-weighted-control-L34} and \eqref{Nxi-infty-L34} established in
Lemma \ref{rarefaction-modulation-L34}.

We first estimate the modulation remainder
\[
\mathcal R_{\mathrm{mod}}
=
\dot{\mathbf X}(\mathbf Y^{(b)}-\mathbf Y)
+\sum_{i=1}^{8}(\mathbf B_i^{(b)}-\mathbf B_i)
-(\mathbf G_1^{(b)}-\mathbf G_1)
-(\mathbf G_2^{(b)}-\mathbf G_2)
-(\mathbf D^{(b)}-\mathbf D).
\]

\emph{(i) The shift contribution.}
From the definition of $\mathbf Y$,
\begin{equation}\label{DeltaY-L34}
\begin{aligned}
\mathbf Y^{(b)}-\mathbf Y
={}&
-\int_{\mathbb R}\rho A_\xi
\left(\Pi(n|\tilde n)+\frac{v^2}{2}\right)d\xi\\
&+\int_{\mathbb R}\rho A
\left(S_\xi\frac{u}{\tilde n}+T_\xi v\right)d\xi.
\end{aligned}
\end{equation}
The estimate $|\dot{\mathbf X}|\le C\chi_1$, already used in the proof of
Lemma \ref{Interaction estimates}, together with
\eqref{rho-localized-L34} and \eqref{Em-equivalence-L34-int}, gives
\begin{equation}\label{DeltaY-first-L34}
|\dot{\mathbf X}|
\left|
\int_{\mathbb R}\rho A_\xi
\left(\Pi(n|\tilde n)+\frac{v^2}{2}\right)d\xi
\right|
\le C\Gamma_S(t)\mathcal E_m(t).
\end{equation}
For the second term in \eqref{DeltaY-L34}, using
$T_\xi=-\sigma^{-1}S_\xi$, Lemma \ref{Interaction estimates}, and the
a priori bound,
\[
\begin{aligned}
&\left|
\int_{\mathbb R}\rho A
\left(S_\xi\frac{u}{\tilde n}+T_\xi v\right)d\xi
\right|\\
&\qquad\le
C\|S_\xi(N-n_m)\|_{L^2}
(\|u\|+\|v\|)\\
&\qquad\le
C\chi_1\delta_R\delta_S^{3/2}e^{-c\delta_St}.
\end{aligned}
\]
Hence, for every $\varepsilon>0$,
\begin{equation}\label{DeltaY-second-L34}
\begin{aligned}
&|\dot{\mathbf X}|
\left|
\int_{\mathbb R}\rho A
\left(S_\xi\frac{u}{\tilde n}+T_\xi v\right)d\xi
\right|\\
&\qquad\le
\varepsilon\delta_S|\dot{\mathbf X}|^2
+C_\varepsilon\chi_1^2\delta_R^2\delta_S^2
e^{-2c\delta_St}\\
&\qquad\le
\varepsilon\delta_S|\dot{\mathbf X}|^2
+C_\varepsilon\chi_1^2\delta_R\delta_S
e^{-c\delta_St}.
\end{aligned}
\end{equation}

\emph{(ii) The localized terms
$\mathbf B_1,\mathbf B_2,\mathbf B_3,\mathbf G_1,\mathbf G_2$.}
By their definitions,
\[
\begin{aligned}
\mathbf B_1^{(b)}-\mathbf B_1
&=\frac{\sigma}{2}\int_{\mathbb R}\rho A_\xi|\varphi(n)|^2d\xi,\\
\mathbf B_2^{(b)}-\mathbf B_2
&=-\int_{\mathbb R}\rho A_\xi\tilde q\,\Pi(n|\tilde n)d\xi,\\
\mathbf G_1^{(b)}-\mathbf G_1
&=\sigma\int_{\mathbb R}\rho A_\xi\Pi(n|\tilde n)d\xi,\\
\mathbf G_2^{(b)}-\mathbf G_2
&=\frac{\sigma}{2}\int_{\mathbb R}\rho A_\xi|z|^2d\xi.
\end{aligned}
\]
Thus \eqref{rho-localized-L34}--\eqref{nonlinear-equivalence-L34} imply
\begin{equation}\label{DeltaB12G12-L34}
\begin{aligned}
&|\mathbf B_1^{(b)}-\mathbf B_1|
+|\mathbf B_2^{(b)}-\mathbf B_2|
+|\mathbf G_1^{(b)}-\mathbf G_1|
+|\mathbf G_2^{(b)}-\mathbf G_2|\\
&\qquad\le C\Gamma_S(t)\mathcal E_m(t).
\end{aligned}
\end{equation}
Similarly,
\[
\mathbf B_3^{(b)}-\mathbf B_3
=
\int_{\mathbb R}\rho A S_{\xi\xi}
\frac{\Pi(n|\tilde n)}{\tilde n}\,d\xi,
\]
and hence, by \eqref{rho-shock-localized-L34},
\begin{equation}\label{DeltaB3-L34}
|\mathbf B_3^{(b)}-\mathbf B_3|
\le
C\delta_S^{3/2}\Gamma_S(t)\mathcal E_m(t)
\le C\Gamma_S(t)\mathcal E_m(t).
\end{equation}

\emph{(iii) The term $\mathbf B_4^{(b)}-\mathbf B_4$.}
We have
\[
\mathbf B_4^{(b)}-\mathbf B_4
=
\int_{\mathbb R}\rho
\left[
A\frac{S_\xi}{\tilde n}-A_\xi
\right]nL\,L_\xi\,d\xi.
\]
By \eqref{rho-localized-L34},
\eqref{rho-shock-localized-L34}, and
\eqref{nonlinear-equivalence-L34},
\[
|\mathbf B_4^{(b)}-\mathbf B_4|
\le
C\Gamma_S(t)\|u\|\|L_\xi\|.
\]
Since $\mathbf D\sim\|L_\xi\|^2$ and
$\mathcal E_m\sim\|u\|^2+\|v\|^2$,
\begin{equation}\label{DeltaB4-L34}
\begin{aligned}
|\mathbf B_4^{(b)}-\mathbf B_4|
&\le
\varepsilon\mathbf D(t)
+C_\varepsilon\Gamma_S(t)^2\mathcal E_m(t)\\
&\le
\varepsilon\mathbf D(t)
+C_\varepsilon\Gamma_S(t)\mathcal E_m(t).
\end{aligned}
\end{equation}

\emph{(iv) The term $\mathbf B_5^{(b)}-\mathbf B_5$.}
Using \eqref{basic-mod-bounds-L34},
\eqref{nonlinear-equivalence-L34}, and Cauchy--Schwarz,
\[
\begin{aligned}
|\mathbf B_5^{(b)}-\mathbf B_5|
&\le
C\delta_R
\int_{\mathbb R}A|N_\xi||u||L_\xi|\,d\xi\\
&\le
C\delta_R
\left(\int_{\mathbb R}A|N_\xi|u^2d\xi\right)^{1/2}
\left(\int_{\mathbb R}A|N_\xi||L_\xi|^2d\xi\right)^{1/2}.
\end{aligned}
\]
By \eqref{GR-weighted-control-L34}, \eqref{Nxi-infty-L34}, and
$\mathbf D\sim\|L_\xi\|^2$,
\[
|\mathbf B_5^{(b)}-\mathbf B_5|
\le
C\delta_R^{3/2}\mathcal G_R(t)^{1/2}\mathbf D(t)^{1/2}.
\]
Therefore
\begin{equation}\label{DeltaB5-L34}
|\mathbf B_5^{(b)}-\mathbf B_5|
\le
\varepsilon\mathbf D(t)
+C_\varepsilon\delta_R^3\mathcal G_R(t)
\le
\varepsilon\mathbf D(t)
+C_\varepsilon\delta_R\mathcal G_R(t).
\end{equation}

\emph{(v) The combination
$(\mathbf B_6^{(b)}-\mathbf B_6)+
(\mathbf B_7^{(b)}-\mathbf B_7)$.}
Since
\[
F_1=S_{\xi\xi}-\tilde n_{\xi\xi}=-N_{\xi\xi},
\]
the two terms must be kept together.  Using
$\Pi(n|\tilde n)+u=nL$, we obtain the exact identity
\begin{equation}\label{DeltaB67-identity-L34}
\begin{aligned}
&(\mathbf B_6^{(b)}-\mathbf B_6)
+(\mathbf B_7^{(b)}-\mathbf B_7)\\
&\qquad=
\int_{\mathbb R}\rho A N_{\xi\xi}
\frac{\Pi(n|\tilde n)+u}{\tilde n}\,d\xi\\
&\qquad=
\int_{\mathbb R}\rho A N_{\xi\xi}
\frac n{\tilde n}\ln\frac n{\tilde n}\,d\xi.
\end{aligned}
\end{equation}
For $|w|\le C\chi_1\ll1$,
\[
\left|
\frac n{\tilde n}\ln\frac n{\tilde n}
\right|
=|(1+w)\ln(1+w)|
\le C|w|.
\]
Consequently, by Sobolev's inequality and
$\mathbf D\sim\|w_\xi\|^2$,
\[
\begin{aligned}
&|(\mathbf B_6^{(b)}-\mathbf B_6)
+(\mathbf B_7^{(b)}-\mathbf B_7)|\\
&\qquad\le
C\delta_R\|N_{\xi\xi}\|_{L^1}\|w\|_{L^\infty}\\
&\qquad\le
C\delta_R\chi_1^{1/2}
\|N_{\xi\xi}\|_{L^1}\mathbf D(t)^{1/4}.
\end{aligned}
\]
Young's inequality with exponents $4$ and $4/3$ yields
\begin{equation}\label{DeltaB67-L34}
\begin{aligned}
&|(\mathbf B_6^{(b)}-\mathbf B_6)
+(\mathbf B_7^{(b)}-\mathbf B_7)|\\
&\qquad\le
\varepsilon\mathbf D(t)
+C_\varepsilon\delta_R^{4/3}\chi_1^{2/3}
\|N_{\xi\xi}\|_{L^1}^{4/3}\\
&\qquad\le
\varepsilon\mathbf D(t)
+C_\varepsilon\delta_R\mathcal S_{\mathrm{sh}}(t),
\end{aligned}
\end{equation}
where we used $0<\delta_R\le1$ and \eqref{Ssh-L34}.

\emph{(vi) The term $\mathbf B_8^{(b)}-\mathbf B_8$.}
From the definition of $\mathbf B_8$,
\[
|\mathbf B_8^{(b)}-\mathbf B_8|
\le
C\delta_R\|u\|\|F_2\|.
\]
The decomposition of $F_2$ used in Step 2 and
\eqref{IRS-L34-int} give
$\|F_2\|\le C\mathcal I_{RS}(t)$. Hence
\begin{equation}\label{DeltaB8-L34}
|\mathbf B_8^{(b)}-\mathbf B_8|
\le
C\delta_R\chi_1\mathcal I_{RS}(t)
\le
C\delta_R\mathcal S_{\mathrm{sh}}(t).
\end{equation}

\emph{(vii) The diffusion difference.}
We have
\[
\mathbf D^{(b)}-\mathbf D
=
\int_{\mathbb R}\rho A n|L_\xi|^2d\xi,
\]
so that
\begin{equation}\label{DeltaD-L34}
|\mathbf D^{(b)}-\mathbf D|
\le C\delta_R\mathbf D(t).
\end{equation}
After fixing $\varepsilon>0$, we choose $\delta_0$ so small that
$C\delta_R\le\varepsilon$; then the right-hand side of
\eqref{DeltaD-L34} is bounded by $\varepsilon\mathbf D(t)$.

Collecting
\eqref{DeltaY-first-L34}--\eqref{DeltaD-L34}, we obtain, for every
fixed $\varepsilon>0$ and for $\delta_0$ sufficiently small,
\begin{equation}\label{Rmod-est-L34}
\begin{aligned}
|\mathcal R_{\mathrm{mod}}(t)|
\le{}&
\varepsilon\left[
\delta_S|\dot{\mathbf X}|^2
+\mathbf D(t)
\right]
+C_\varepsilon\delta_R\mathcal G_R(t)\\
&+C_\varepsilon\Gamma_S(t)\mathcal E_m(t)
+C_\varepsilon\delta_R\mathcal S_{\mathrm{sh}}(t)
+C_\varepsilon\chi_1^2\delta_R\delta_Se^{-c\delta_St}.
\end{aligned}
\end{equation}
In particular, \eqref{Rmod-est-L34} is stronger than the estimate obtained
by adding the nonnegative terms
$\varepsilon G^s+\varepsilon\mathbf G_2$ to its right-hand side.

The explicit rarefaction remainders
$\mathcal R_{R,1},\ldots,\mathcal R_{R,4}$ have already been isolated
and estimated in Lemma \ref{rarefaction-modulation-L34}; in particular,
\eqref{RR24-absorb-L34} and \eqref{RR13-time-L34} will be used below.

\emph{Step 4. Integration in time and closure.}
Substituting
\eqref{shock-block-L34} and \eqref{Rmod-est-L34}
into the exact identity \eqref{exact-split-L34}, and using
Lemma \ref{rarefaction-modulation-L34}, choosing
$\varepsilon$, $\chi_1$, $\delta_R$ and $\delta_S$ sufficiently small, and
then integrating over $[0,t]$, we obtain
\begin{equation}\label{integrated-L34}
\begin{aligned}
&\mathcal E_m(t)
+c\int_0^t
\left[
\delta_S|\dot{\mathbf X}(\tau)|^2
+G^s(\tau)
+\mathbf G_2(\tau)
+\mathbf D(\tau)
+\mathcal G_R(\tau)
\right]d\tau\\
&\le
\mathcal E_m(0)
+C\delta_R^{1/3}
+C\int_0^t\Gamma_S(\tau)\mathcal E_m(\tau)d\tau .
\end{aligned}
\end{equation}
Here we used
\eqref{Rrw1-int-L34-int},
\eqref{Ssh-time-L34-int},
\eqref{Gamma-time-L34-int}, and
\eqref{RR13-time-L34}; the lower-order contributions
$C\delta_R$, $C\delta_R^2$ and
$C\chi_1^2\delta_R$ are bounded by $C\delta_R^{1/3}$.

By Gronwall's inequality and \eqref{Gamma-time-L34-int},
\begin{equation}\label{integrated-Gronwall-L34}
\begin{aligned}
&\mathcal E_m(t)
+\int_0^t
\left[
\delta_S|\dot{\mathbf X}|^2
+G^s+\mathbf G_2+\mathbf D+\mathcal G_R
\right]d\tau\\
&\le
C\exp\left(C\delta_R\sqrt{\delta_S}\right)
\left[
\mathcal E_m(0)+\delta_R^{1/3}
\right]
\le
C\left[
\mathcal E_m(0)+\delta_R^{1/3}
\right],
\end{aligned}
\end{equation}
after decreasing $\delta_*$ if necessary.

It remains to recover the quantities in \eqref{l2 estimate}.
By \eqref{4.a'} and the definition of $\mathbf G_2$,
\begin{equation}\label{G2-char-L34-int}
\mathbf G_2(t)
=
\frac{\sigma\lambda}{2\delta_S}
\int_{\mathbb R}|S_\xi|
|v+\varphi(n)|^2d\xi.
\end{equation}
Thus the large factor $\lambda/\delta_S$ is retained only in the
characteristic combination $v+\varphi(n)$.

By \eqref{rare-rel-L34-int},
\eqref{lambda-bounds-L34-int},
the equivalence $\Pi(n|\tilde n)\sim u^2$, and the uniform bounds for
$a^{-\mathbf X}$ and $N$,
\begin{equation}\label{GR-controls-L34-int}
\mathcal G_R(t)
\ge
c\int_{\mathbb R}|Q_\xi|
\left(u^2+|\Lambda|^2v^2\right)d\xi.
\end{equation}
Also,
\begin{equation}\label{Gs-controls-L34-int}
\int_{\mathbb R}|S_\xi|u^2d\xi
\le CG^s(t).
\end{equation}

Finally,
\[
u_\xi=\tilde n\,w_\xi+\tilde n_\xi w,
\qquad
\tilde n_\xi=N_\xi+S_\xi.
\]
Since $\mathbf D\sim\|w_\xi\|^2$,
$\|S_\xi\|_{L^\infty}\le C\delta_S^2$, and
$\|N_\xi\|_{L^\infty}\le C\delta_R$, we have
\begin{align}
\|u_\xi\|^2
&\le
C\mathbf D
+C\int_{\mathbb R}|S_\xi|^2w^2d\xi
+C\int_{\mathbb R}|N_\xi|^2w^2d\xi\nonumber\\
&\le
C\mathbf D
+C\delta_S^2G^s
+C\delta_R\mathcal G_R.
\label{ux-control-L34-int}
\end{align}
Combining
\eqref{Em-equivalence-L34-int},
\eqref{integrated-Gronwall-L34},
\eqref{G2-char-L34-int},
\eqref{GR-controls-L34-int},
\eqref{Gs-controls-L34-int}, and
\eqref{ux-control-L34-int}
gives \eqref{l2 estimate}.
\end{proof}

\subsubsection{$H^1$ estimate}
We next establish estimates  for $\|\partial_{\xi}(n-\tilde{n})\|$ and $\|\partial_{\xi}(q-\tilde{q})\|$.

\begin{lemma}\label{4.First-order estimate}
Under the hypotheses of Proposition \ref{4 A priori estimates}, there exists
a constant $C>0$, independent of $\delta_0$, $\chi_1$ and $T$, such that,
for every $t\in[0,T]$,
\begin{equation}\label{4.First order estimate}
\begin{aligned}
&\int_{\mathbb R}
\left(
|\partial_\xi(n-\tilde n)|^2
+|\partial_\xi(q-\tilde q)|^2
\right)d\xi
+\int_0^t\int_{\mathbb R}
\frac{1}{\tilde n}
\left|\partial_\xi^2(n-\tilde n)\right|^2
\,d\xi d\tau\\
&\le
C\left(
\|n_0-\tilde n(\cdot,0)\|_{H^1}^2
+\|q_0-\tilde q(\cdot,0)\|_{H^1}^2
\right)\\
&\quad
+C(\chi_1+\delta_R+\delta_S)
\int_0^t
\|\partial_\xi(q-\tilde q)(\cdot,\tau)\|^2\,d\tau
+C\delta_R^{1/3}.
\end{aligned}
\end{equation}
\end{lemma}

\begin{proof}
Set
\[
u:=n-\tilde n,\qquad
v:=q-\tilde q,
\]
and use the notation introduced in Lemma \ref{4.L2 estimate},
\[
N(\xi,t):=\tilde n^R(\xi+\sigma t,t),\qquad
Q(\xi,t):=\tilde q^R(\xi+\sigma t,t),
\]
\[
S(\xi,t):=(\tilde n^S)^{-\mathbf X}(\xi),\qquad
T(\xi,t):=(\tilde q^S)^{-\mathbf X}(\xi),
\]
so that
\[
\tilde n=N+S-n_m,\qquad
\tilde q=Q+T-q_m.
\]
We also write
\[
F:=F_1+F_2,\qquad
z:=v+\varphi(n).
\]
Recall that
\begin{equation}\label{nq'}
\begin{aligned}
\partial_\xi(nq-\tilde n\tilde q)
&=\partial_\xi(\tilde n v+qu)\\
&=\tilde n\,v_\xi+\tilde n_\xi v+q\,u_\xi+u\,q_\xi.
\end{aligned}
\end{equation}

Multiplying the second equation of \eqref{4.perturbed system} by
$-v_{\xi\xi}$ and integrating over $\mathbb R$, we obtain
\begin{equation}\label{4.nabla q-q'}
\frac{d}{dt}\int_{\mathbb R}\frac{|v_\xi|^2}{2}\,d\xi
=
\int_{\mathbb R}v_\xi u_{\xi\xi}\,d\xi
-\dot{\mathbf X}(t)
\int_{\mathbb R}T_\xi v_{\xi\xi}\,d\xi.
\end{equation}

We next derive the corresponding identity for $u_\xi$.  Since
\[
N_t-\sigma N_\xi=(NQ)_\xi,\qquad
S_t=-\dot{\mathbf X}S_\xi,
\]
we have
\begin{equation}\label{ntilde-t-L35}
\tilde n_t
=
\sigma N_\xi+(NQ)_\xi-\dot{\mathbf X}S_\xi.
\end{equation}
Multiplying the first equation of \eqref{4.perturbed system} by
$-\tilde n^{-1}u_{\xi\xi}$, integrating by parts, and using
\eqref{nq'}--\eqref{ntilde-t-L35}, we obtain the exact identity
\begin{equation}\label{4.4.55}
\begin{aligned}
&\frac{d}{dt}
\int_{\mathbb R}\frac{|u_\xi|^2}{2\tilde n}\,d\xi
+\int_{\mathbb R}\frac{|u_{\xi\xi}|^2}{\tilde n}\,d\xi\\
&=
\dot{\mathbf X}(t)\sum_{i=1}^{3}\Theta_i(t)
+\sum_{i=1}^{12}J_i(t)
-\int_{\mathbb R}v_\xi u_{\xi\xi}\,d\xi,
\end{aligned}
\end{equation}
where
\begin{equation}\label{Theta123-L35}
\begin{aligned}
\Theta_1(t)
&:=
-\int_{\mathbb R}\frac{S_\xi}{\tilde n}u_{\xi\xi}\,d\xi,\\
\Theta_2(t)
&:=
\int_{\mathbb R}
\frac{S_\xi}{2\tilde n^2}|u_\xi|^2\,d\xi,\\
\Theta_3(t)
&:=
\int_{\mathbb R}
\frac{\tilde n_\xi S_\xi}{\tilde n^2}u_\xi\,d\xi,
\end{aligned}
\end{equation}
and
\begin{equation}\label{J1-J12-L35}
\begin{aligned}
J_1(t)
&:=-\int_{\mathbb R}
\frac{\tilde n_\xi}{\tilde n}\,v\,u_{\xi\xi}\,d\xi,\\
J_2(t)
&:=-\int_{\mathbb R}
\frac{q}{\tilde n}\,u_\xi u_{\xi\xi}\,d\xi,\\
J_3(t)
&:=-\int_{\mathbb R}
\frac{u}{\tilde n}\,q_\xi u_{\xi\xi}\,d\xi,\\
J_4(t)
&:=\frac{\sigma}{2}\int_{\mathbb R}
\frac{S_\xi}{\tilde n^2}|u_\xi|^2\,d\xi,\\
J_5(t)
&:=\int_{\mathbb R}
\frac{\tilde n_\xi}{\tilde n}\,v_\xi u_\xi\,d\xi,\\
J_6(t)
&:=\int_{\mathbb R}
\left(\frac{\tilde n_\xi}{\tilde n}\right)^2
u_\xi v\,d\xi,\\
J_7(t)
&:=\int_{\mathbb R}
\frac{\tilde n_\xi q}{\tilde n^2}|u_\xi|^2\,d\xi,\\
J_8(t)
&:=\int_{\mathbb R}
\frac{\tilde n_\xi}{\tilde n^2}
u\,u_\xi q_\xi\,d\xi,\\
J_9(t)
&:=\int_{\mathbb R}
\frac{\tilde n_\xi}{\tilde n^2}
u_\xi u_{\xi\xi}\,d\xi,\\
J_{10}(t)
&:=-\int_{\mathbb R}
\frac{(NQ)_\xi}{2\tilde n^2}|u_\xi|^2\,d\xi,\\
J_{11}(t)
&:=-\int_{\mathbb R}
\frac{\tilde n_\xi}{\tilde n^2}F\,u_\xi\,d\xi,\\
J_{12}(t)
&:=\int_{\mathbb R}
\frac{F}{\tilde n}u_{\xi\xi}\,d\xi.
\end{aligned}
\end{equation}
Notice in particular that the transport contribution in $J_4$ contains
$S_\xi$, not $\tilde n_\xi$.  This follows from the cancellation
\[
\frac{\sigma\tilde n_\xi-\tilde n_t}{2\tilde n^2}
=
\frac{\sigma S_\xi-(NQ)_\xi+\dot{\mathbf X}S_\xi}
{2\tilde n^2},
\]
where the last term belongs to $\dot{\mathbf X}\Theta_2$ and the
$(NQ)_\xi$ term is $J_{10}$.

Adding \eqref{4.nabla q-q'} and \eqref{4.4.55}, the mixed term cancels and
we arrive at
\begin{equation}\label{4.D1,XY+I}
\begin{aligned}
&\frac{d}{dt}
\int_{\mathbb R}
\left(
\frac{|u_\xi|^2}{2\tilde n}
+\frac{|v_\xi|^2}{2}
\right)d\xi
+\mathbf D_1(t)\\
&=
\dot{\mathbf X}(t)\sum_{i=1}^{4}\Theta_i(t)
+\sum_{i=1}^{12}J_i(t),
\end{aligned}
\end{equation}
where
\[
\Theta_4(t):=-\int_{\mathbb R}T_\xi v_{\xi\xi}\,d\xi,
\qquad
\mathbf D_1(t):=
\int_{\mathbb R}\frac{|u_{\xi\xi}|^2}{\tilde n}\,d\xi.
\]

We now estimate the right-hand side of \eqref{4.D1,XY+I}.
By Lemma \ref{properties of the 2-viscous shock wave},
\[
\|S_\xi\|_{L^\infty}\le C\delta_S^2,\qquad
\|S_\xi\|_{L^2}^2\le C\delta_S^3,
\]
and
\[
\|\tilde n_\xi\|_{L^\infty}
+\|\tilde q_\xi\|_{L^\infty}
\le C(\delta_R+\delta_S).
\]
Hence
\begin{equation}\label{Theta1-L35}
|\dot{\mathbf X}\Theta_1|
\le
C\delta_S^2|\dot{\mathbf X}|^2
+C\delta_S\mathbf D_1.
\end{equation}
Using the a priori bound $\|u_\xi\|\le\chi_1$,
\begin{equation}\label{Theta2-L35}
|\dot{\mathbf X}\Theta_2|
\le
C\delta_S^2|\dot{\mathbf X}|^2
+C\delta_S^2\chi_1^2\|u_\xi\|^2.
\end{equation}
Similarly,
\begin{equation}\label{Theta3-L35}
|\dot{\mathbf X}\Theta_3|
\le
C\delta_S^2|\dot{\mathbf X}|^2
+C(\delta_R+\delta_S)^2\delta_S\|u_\xi\|^2.
\end{equation}

For $\Theta_4$, we first record the third-derivative estimate
\begin{equation}\label{shock-third-L35}
|S_{\xi\xi\xi}|
\le C\delta_S^2|S_\xi|.
\end{equation}
Indeed, differentiating the explicit profile formula \eqref{nSqS}
(or, equivalently, the scalar shock ODE obtained from
\eqref{traveling wave}) twice, and using
$|S_{\xi\xi}|\le C\delta_S|S_\xi|$ and
$|S_\xi|\le C\delta_S^2$, gives \eqref{shock-third-L35}.
Since $T_\xi=-\sigma^{-1}S_\xi$, two integrations by parts give
\[
\Theta_4(t)
=
\frac{1}{\sigma}\int_{\mathbb R}S_{\xi\xi\xi}v\,d\xi.
\]
Set
\[
Z_S(t):=\int_{\mathbb R}|S_\xi||z|^2\,d\xi,\qquad
U_S(t):=\int_{\mathbb R}|S_\xi||u|^2\,d\xi.
\]
Since $|v|\le |z|+C|u|$,
\[
|\Theta_4(t)|
\le
C\delta_S^{5/2}
\left(Z_S(t)^{1/2}+U_S(t)^{1/2}\right).
\]
Thus, for every $\varepsilon>0$,
\begin{equation}\label{Theta4-L35}
|\dot{\mathbf X}\Theta_4|
\le
\varepsilon\delta_S|\dot{\mathbf X}|^2
+C_\varepsilon\delta_S^4
\left(Z_S+U_S\right).
\end{equation}

We next estimate the terms $J_i$.  The rarefaction dissipation
$\mathcal G_R$ introduced in \eqref{GR-def-L34-int} satisfies
\begin{equation}\label{GR-L35-control}
\int_{\mathbb R}|N_\xi|(u^2+v^2)\,d\xi
\le C\mathcal G_R(t),
\end{equation}
while
\begin{equation}\label{Nxi-L35}
\|N_\xi(t)\|_{L^\infty}\le C\delta_R.
\end{equation}
We shall also use
\begin{equation}\label{4.4.51}
\|v\|_{L^\infty}
\le C\|v\|_{H^1}
\le C\chi_1,
\end{equation}
and
\begin{equation}\label{LH1}
\|u\|_{L^\infty}
\le C\|u\|_{H^1}
\le C\chi_1.
\end{equation}

For $J_1$, split $\tilde n_\xi=N_\xi+S_\xi$.  By
\eqref{GR-L35-control}--\eqref{Nxi-L35},
\[
\begin{aligned}
\left|
\int_{\mathbb R}\frac{N_\xi}{\tilde n}v\,u_{\xi\xi}\,d\xi
\right|
&\le
C\left(\int|N_\xi|v^2\right)^{1/2}
\left(\int|N_\xi||u_{\xi\xi}|^2\right)^{1/2}\\
&\le
C\sqrt{\delta_R}\,
\mathcal G_R^{1/2}\mathbf D_1^{1/2}.
\end{aligned}
\]
For the shock part, using $|v|\le |z|+C|u|$,
\[
\left|
\int_{\mathbb R}\frac{S_\xi}{\tilde n}v\,u_{\xi\xi}\,d\xi
\right|
\le
C\delta_S\mathbf D_1^{1/2}
\left(Z_S^{1/2}+U_S^{1/2}\right).
\]
Consequently,
\begin{equation}\label{J1-L35}
|J_1|
\le
\varepsilon\mathbf D_1
+C_\varepsilon\delta_R\mathcal G_R
+C_\varepsilon\delta_S^2(Z_S+U_S).
\end{equation}

Since $q$ is uniformly bounded under the a priori assumption,
\begin{equation}\label{J2-L35}
|J_2|
\le
\varepsilon\mathbf D_1
+C_\varepsilon\|u_\xi\|^2.
\end{equation}

For $J_3$, we use
\begin{equation}\label{4.qxi}
q_\xi=v_\xi+Q_\xi+T_\xi,\qquad
|Q_\xi|\le C|N_\xi|,\qquad
|T_\xi|=\frac1\sigma|S_\xi|.
\end{equation}
The $v_\xi$ part is bounded by
$C\chi_1(\|v_\xi\|^2+\mathbf D_1)$.
The $Q_\xi$ and $T_\xi$ parts are estimated as for $J_1$.
Thus
\begin{equation}\label{J3-L35}
|J_3|
\le
C\chi_1\|v_\xi\|^2
+(C\chi_1+\varepsilon)\mathbf D_1
+C_\varepsilon\delta_R\mathcal G_R
+C_\varepsilon\delta_S^2U_S.
\end{equation}

Since $J_4$ contains only $S_\xi$,
\begin{equation}\label{J4-L35}
|J_4|
\le C\delta_S^2\|u_\xi\|^2.
\end{equation}
Furthermore,
\begin{equation}\label{J5-L35}
|J_5|
\le
C(\delta_R+\delta_S)
\left(\|u_\xi\|^2+\|v_\xi\|^2\right).
\end{equation}

For $J_6$, using
$(N_\xi+S_\xi)^2\le2N_\xi^2+2S_\xi^2$,
\eqref{GR-L35-control}, and $|v|\le|z|+C|u|$, we find
\begin{equation}\label{J6-L35}
\begin{aligned}
|J_6|
\le{}&
C\delta_R\mathcal G_R
+C\delta_R^2\|u_\xi\|^2\\
&+C\delta_S^3
\left(
\|u_\xi\|^2+Z_S+U_S
\right).
\end{aligned}
\end{equation}
Since $q$ is bounded,
\begin{equation}\label{J7-L35}
|J_7|
\le
C(\delta_R+\delta_S)\|u_\xi\|^2.
\end{equation}

To estimate $J_8$, write
\[
q_\xi=v_\xi+\tilde q_\xi,
\qquad
|\tilde n_\xi|+|\tilde q_\xi|
\le C(|N_\xi|+|S_\xi|).
\]
The part containing $v_\xi$ is bounded by
\[
C\chi_1(\delta_R+\delta_S)
\left(\|u_\xi\|^2+\|v_\xi\|^2\right).
\]
For the background part, using
$(|N_\xi|+|S_\xi|)^2\le2|N_\xi|^2+2|S_\xi|^2$,
we obtain
\begin{equation}\label{J8-L35}
\begin{aligned}
|J_8|
\le{}&
C\chi_1(\delta_R+\delta_S)
\left(\|u_\xi\|^2+\|v_\xi\|^2\right)
+C\delta_R\mathcal G_R\\
&+C\delta_R^2\|u_\xi\|^2
+C\delta_S^3\left(U_S+\|u_\xi\|^2\right).
\end{aligned}
\end{equation}

For $J_9$, using
$\|\tilde n_\xi\|_{L^\infty}\le C(\delta_R+\delta_S)$,
\begin{equation}\label{J9-L35}
|J_9|
\le
C(\delta_R+\delta_S)
\left(\mathbf D_1+\|u_\xi\|^2\right).
\end{equation}
Since
\[
|(NQ)_\xi|\le C|N_\xi|,
\]
we also have
\begin{equation}\label{J10-L35}
|J_{10}|
\le C\delta_R\|u_\xi\|^2.
\end{equation}

Finally, $F_1=-N_{\xi\xi}$ and, by the decomposition of $F_2$ used in
Step 2 of Lemma \ref{4.L2 estimate},
\begin{equation}\label{F-L35}
\|F\|^2
\le
C\|N_{\xi\xi}\|^2
+C\mathcal I_{RS}(t)^2.
\end{equation}
Therefore,
\begin{align}
|J_{11}|
&\le
C(\delta_R+\delta_S)
\left(
\|u_\xi\|^2+\|N_{\xi\xi}\|^2+\mathcal I_{RS}^2
\right),
\label{J11-L35}\\
|J_{12}|
&\le
\varepsilon\mathbf D_1
+C_\varepsilon
\left(
\|N_{\xi\xi}\|^2+\mathcal I_{RS}^2
\right).
\label{J12-L35}
\end{align}

Combining
\eqref{Theta1-L35}--\eqref{J12-L35}, choosing $\varepsilon>0$ fixed
and then $\delta_0,\chi_1$ sufficiently small, we obtain
\begin{equation}\label{4.D1}
\begin{aligned}
&\frac{d}{dt}
\int_{\mathbb R}
\left(
\frac{|u_\xi|^2}{\tilde n}+|v_\xi|^2
\right)d\xi
+c\,\mathbf D_1(t)\\
&\le
C\delta_S|\dot{\mathbf X}(t)|^2
+C\frac{\delta_S}{\lambda}Z_S(t)
+C U_S(t)
+C\delta_R\mathcal G_R(t)\\
&\quad
+C\|u_\xi(t)\|^2
+C(\chi_1+\delta_R+\delta_S)\|v_\xi(t)\|^2\\
&\quad
+C\|N_{\xi\xi}(t)\|^2
+C\mathcal I_{RS}(t)^2.
\end{aligned}
\end{equation}
Here we used $\lambda=\sqrt{\delta_S}$ and $0<\delta_S\le1$ to
replace all higher powers of $\delta_S$ multiplying $Z_S$ and $U_S$
by the displayed terms.

We now integrate \eqref{4.D1} over $[0,t]$.  By Lemma
\ref{properties of nRqR},
\[
\|N_{\xi\xi}(t)\|
\le C\min\{\delta_R,t^{-1}\},
\]
and hence
\begin{equation}\label{Nxx-L2time-L35}
\int_0^\infty\|N_{\xi\xi}(\tau)\|^2\,d\tau
\le C\delta_R.
\end{equation}
Moreover, by \eqref{nSnR},
\begin{equation}\label{IRS-L2time-L35}
\int_0^\infty\mathcal I_{RS}(\tau)^2\,d\tau
\le C\delta_R^2\delta_S
\le C\delta_R.
\end{equation}

The remaining terms in the first two lines on the right-hand side of
\eqref{4.D1} are controlled by Lemma \ref{4.L2 estimate}: indeed,
\[
\delta_S\int_0^t|\dot{\mathbf X}|^2,\qquad
\int_0^tU_S,\qquad
\int_0^t\mathcal G_R,\qquad
\int_0^t\|u_\xi\|^2
\]
are bounded by the right-hand side of \eqref{l2 estimate}, while
\[
\frac{\delta_S}{\lambda}Z_S
=
\delta_S
\left(
\frac{\lambda}{\delta_S}Z_S
\right)
\]
is bounded by the characteristic shock dissipation in
\eqref{G2-char-L34-int}.  Therefore,
\[
\begin{aligned}
&\|u_\xi(t)\|^2+\|v_\xi(t)\|^2
+\int_0^t\mathbf D_1(\tau)\,d\tau\\
&\le
C\left(
\|n_0-\tilde n(\cdot,0)\|_{H^1}^2
+\|q_0-\tilde q(\cdot,0)\|_{H^1}^2
\right)\\
&\quad
+C(\chi_1+\delta_R+\delta_S)
\int_0^t\|v_\xi(\tau)\|^2\,d\tau
+C\delta_R^{1/3},
\end{aligned}
\]
where we also used $\delta_R\le\delta_R^{1/3}$ for
$0<\delta_R\le1$.  This is exactly \eqref{4.First order estimate}.
\end{proof}

We are now ready to establish the $H^1$ estimate.
\begin{lemma}\label{4.H1-estimate}
Under the hypotheses of Proposition \ref{4 A priori estimates}, there exists
a constant $C>0$, independent of $\delta_0$, $\chi_1$ and $T$, such that,
for every $t\in[0,T]$,
\begin{equation}\label{4.H1 estimate}
\begin{aligned}
&\|\partial_\xi(n-\tilde n)(\cdot,t)\|^2
+\|\partial_\xi(q-\tilde q)(\cdot,t)\|^2\\
&\quad
+\int_0^t
\|\partial_\xi^2(n-\tilde n)(\cdot,\tau)\|^2\,d\tau
+\int_0^t
\|\partial_\xi(q-\tilde q)(\cdot,\tau)\|^2\,d\tau\\
&\le
C\left(
\|n_0-\tilde n(\cdot,0)\|_{H^1}^2
+\|q_0-\tilde q(\cdot,0)\|_{H^1}^2
\right)
+C\delta_R^{1/3}.
\end{aligned}
\end{equation}
\end{lemma}

\begin{proof}
We keep the notation
\[
u:=n-\tilde n,\qquad
v:=q-\tilde q,\qquad
N:=\tilde n^R(\xi+\sigma t,t),\qquad
Q:=\tilde q^R(\xi+\sigma t,t),
\]
\[
S:=(\tilde n^S)^{-\mathbf X},\qquad
T:=(\tilde q^S)^{-\mathbf X},
\]
introduced in the previous lemmas.  Set
\begin{equation}\label{Eq-L36}
\mathcal E_q(t):=
\int_{\mathbb R}\tilde n\,|v_\xi|^2\,d\xi,
\end{equation}
and, as in Lemma \ref{4.First-order estimate},
\[
Z_S(t):=\int_{\mathbb R}|S_\xi|
|v+\varphi(n)|^2\,d\xi,\qquad
U_S(t):=\int_{\mathbb R}|S_\xi||u|^2\,d\xi.
\]
Since $\tilde n$ is uniformly bounded above and below,
\begin{equation}\label{Eq-equivalence-L36}
c\|v_\xi(t)\|^2
\le\mathcal E_q(t)
\le C\|v_\xi(t)\|^2.
\end{equation}

By Lemma \ref{4.First-order estimate}, it remains only to estimate
$\int_0^t\|v_\xi(\tau)\|^2d\tau$.  From the first equation of
\eqref{4.perturbed system} and \eqref{nq'}, we have
\begin{equation}\label{q-qxi}
\begin{aligned}
\mathcal E_q(t)
-\frac{d}{dt}\int_{\mathbb R}u\,v_\xi\,d\xi
={}&
-\int_{\mathbb R}u\,v_{\xi t}\,d\xi
-\sigma\int_{\mathbb R}u_\xi v_\xi\,d\xi\\
&-\int_{\mathbb R}q\,u_\xi v_\xi\,d\xi
-\int_{\mathbb R}u\,q_\xi v_\xi\,d\xi\\
&-\dot{\mathbf X}(t)\int_{\mathbb R}S_\xi v_\xi\,d\xi
-\int_{\mathbb R}u_{\xi\xi}v_\xi\,d\xi\\
&-\int_{\mathbb R}\tilde n_\xi v\,v_\xi\,d\xi
+\int_{\mathbb R}F_1v_\xi\,d\xi
+\int_{\mathbb R}F_2v_\xi\,d\xi.
\end{aligned}
\end{equation}
Differentiating the second equation of \eqref{4.perturbed system} with
respect to $\xi$ yields
\begin{equation}\label{4.4.61}
v_{\xi t}-\sigma v_{\xi\xi}
-\dot{\mathbf X}(t)T_{\xi\xi}
-u_{\xi\xi}=0.
\end{equation}
Using \eqref{4.4.61} in the first two terms on the right-hand side of
\eqref{q-qxi} and integrating by parts, we get the exact cancellation
\begin{equation}\label{K12-L36}
\begin{aligned}
&-\int_{\mathbb R}u\,v_{\xi t}\,d\xi
-\sigma\int_{\mathbb R}u_\xi v_\xi\,d\xi\\
&\qquad=
\dot{\mathbf X}(t)
\int_{\mathbb R}u_\xi T_\xi\,d\xi
+\int_{\mathbb R}|u_\xi|^2\,d\xi.
\end{aligned}
\end{equation}
Since $T_\xi=-\sigma^{-1}S_\xi$ and
$\|S_\xi\|_{L^2}^2\le C\delta_S^3$,
\begin{equation}\label{K12-est-L36}
\left|
\dot{\mathbf X}\int_{\mathbb R}u_\xi T_\xi\,d\xi
\right|
\le
C\delta_S^2|\dot{\mathbf X}|^2
+C\delta_S\|u_\xi\|^2.
\end{equation}

We now estimate the remaining terms in \eqref{q-qxi}.  First, since $q$
is uniformly bounded under the a priori assumption, for every
$\varepsilon>0$,
\begin{equation}\label{K3-L36}
\left|
\int_{\mathbb R}q\,u_\xi v_\xi\,d\xi
\right|
\le
\varepsilon\mathcal E_q(t)
+C_\varepsilon\|u_\xi(t)\|^2.
\end{equation}

For the term containing $q_\xi$, write
\[
q_\xi=v_\xi+Q_\xi+T_\xi.
\]
The first part is bounded by
$C\chi_1\mathcal E_q(t)$.  For the rarefaction part, using
$|Q_\xi|\le C|N_\xi|$, \eqref{GR-L35-control}, and
$\|Q_\xi\|_{L^\infty}\le C\delta_R$, we obtain
\[
\begin{aligned}
\int_{\mathbb R}|Q_\xi||u||v_\xi|\,d\xi
&\le
C\left(\int_{\mathbb R}|N_\xi|u^2d\xi\right)^{1/2}
\left(\int_{\mathbb R}|N_\xi||v_\xi|^2d\xi\right)^{1/2}\\
&\le
C\sqrt{\delta_R}\,
\mathcal G_R(t)^{1/2}\mathcal E_q(t)^{1/2}.
\end{aligned}
\]
For the shock part,
\[
\begin{aligned}
\int_{\mathbb R}|T_\xi||u||v_\xi|\,d\xi
&\le
C U_S(t)^{1/2}
\left(\int_{\mathbb R}|S_\xi||v_\xi|^2d\xi\right)^{1/2}\\
&\le
C\delta_S\,U_S(t)^{1/2}\mathcal E_q(t)^{1/2}.
\end{aligned}
\]
Consequently,
\begin{equation}\label{K4-L36}
\left|
\int_{\mathbb R}u\,q_\xi v_\xi\,d\xi
\right|
\le
(C\chi_1+\varepsilon)\mathcal E_q(t)
+C_\varepsilon\delta_R\mathcal G_R(t)
+C_\varepsilon\delta_S^2U_S(t).
\end{equation}

For the shift term, by
$\|S_\xi\|_{L^2}^2\le C\delta_S^3$,
\begin{equation}\label{K5-L36}
\left|
\dot{\mathbf X}
\int_{\mathbb R}S_\xi v_\xi\,d\xi
\right|
\le
\varepsilon\mathcal E_q(t)
+C_\varepsilon\delta_S^3|\dot{\mathbf X}(t)|^2.
\end{equation}
Also,
\begin{equation}\label{K6-L36}
\left|
\int_{\mathbb R}u_{\xi\xi}v_\xi\,d\xi
\right|
\le
\varepsilon\mathcal E_q(t)
+C_\varepsilon\mathbf D_1(t).
\end{equation}

We next estimate the coefficient-gradient term.  Since
$\tilde n_\xi=N_\xi+S_\xi$ and
$|v|\le |v+\varphi(n)|+C|u|$, the rarefaction part satisfies
\[
\begin{aligned}
\int_{\mathbb R}|N_\xi||v||v_\xi|\,d\xi
&\le
\left(\int_{\mathbb R}|N_\xi|v^2d\xi\right)^{1/2}
\left(\int_{\mathbb R}|N_\xi||v_\xi|^2d\xi\right)^{1/2}\\
&\le
C\sqrt{\delta_R}\,
\mathcal G_R(t)^{1/2}\mathcal E_q(t)^{1/2},
\end{aligned}
\]
whereas the shock part satisfies
\[
\begin{aligned}
\int_{\mathbb R}|S_\xi||v||v_\xi|\,d\xi
&\le
C\left(Z_S(t)+U_S(t)\right)^{1/2}
\left(\int_{\mathbb R}|S_\xi||v_\xi|^2d\xi\right)^{1/2}\\
&\le
C\delta_S
\left(Z_S(t)+U_S(t)\right)^{1/2}
\mathcal E_q(t)^{1/2}.
\end{aligned}
\]
Therefore,
\begin{equation}\label{K7-L36}
\left|
\int_{\mathbb R}\tilde n_\xi v\,v_\xi\,d\xi
\right|
\le
\varepsilon\mathcal E_q(t)
+C_\varepsilon\delta_R\mathcal G_R(t)
+C_\varepsilon\delta_S^2\left(Z_S(t)+U_S(t)\right).
\end{equation}

Finally, $F_1=-N_{\xi\xi}$ and the interaction source $F_2$ satisfies
the estimate used in Lemma \ref{4.First-order estimate}.  Hence
\begin{align}
\left|\int_{\mathbb R}F_1v_\xi\,d\xi\right|
&\le
\varepsilon\mathcal E_q(t)
+C_\varepsilon\|N_{\xi\xi}(t)\|^2,
\label{K8-L36}\\
\left|\int_{\mathbb R}F_2v_\xi\,d\xi\right|
&\le
\varepsilon\mathcal E_q(t)
+C_\varepsilon\mathcal I_{RS}(t)^2.
\label{K9-L36}
\end{align}

Substituting
\eqref{K12-L36}--\eqref{K9-L36} into \eqref{q-qxi}, first choosing
$\varepsilon>0$ sufficiently small and then taking $\chi_1$ sufficiently
small, gives
\begin{equation}\label{vq-diff-L36}
\begin{aligned}
c\mathcal E_q(t)
\le{}&
\frac{d}{dt}\int_{\mathbb R}u\,v_\xi\,d\xi
+C\|u_\xi(t)\|^2
+C\mathbf D_1(t)\\
&+C\delta_S^2|\dot{\mathbf X}(t)|^2
+C\delta_R\mathcal G_R(t)\\
&+C\delta_S^2\left(Z_S(t)+U_S(t)\right)
+C\|N_{\xi\xi}(t)\|^2
+C\mathcal I_{RS}(t)^2.
\end{aligned}
\end{equation}

We now integrate \eqref{vq-diff-L36} over $[0,t]$.  For the boundary
term, Young's inequality gives, for any fixed $\eta>0$,
\begin{equation}\label{boundary-L36}
\begin{aligned}
\left|\int_{\mathbb R}u(t)v_\xi(t)d\xi\right|
+\left|\int_{\mathbb R}u(0)v_\xi(0)d\xi\right|
\le{}&
\eta\|v_\xi(t)\|^2
+C_\eta\|u(t)\|^2\\
&+C\left(
\|u(0)\|_{H^1}^2+\|v(0)\|_{H^1}^2
\right).
\end{aligned}
\end{equation}
By Lemma \ref{4.L2 estimate},
\begin{equation}\label{L2-use-L36}
\begin{aligned}
&\|u(t)\|^2
+\int_0^t\|u_\xi(\tau)\|^2d\tau
+\delta_S\int_0^t|\dot{\mathbf X}(\tau)|^2d\tau\\
&\quad
+\int_0^t U_S(\tau)d\tau
+\int_0^t\mathcal G_R(\tau)d\tau
+\frac{\lambda}{\delta_S}
\int_0^t Z_S(\tau)d\tau\\
&\le
C\left(
\|u(0)\|^2+\|v(0)\|^2
\right)+C\delta_R^{1/3}.
\end{aligned}
\end{equation}
Since $0<\delta_S\le1$ and $\lambda=\sqrt{\delta_S}$,
the terms
\[
\delta_S^2\int_0^t|\dot{\mathbf X}|^2,\qquad
\delta_S^2\int_0^t(U_S+Z_S),\qquad
\delta_R\int_0^t\mathcal G_R
\]
are all controlled by the right-hand side of \eqref{L2-use-L36}.
Moreover, \eqref{Nxx-L2time-L35} and \eqref{IRS-L2time-L35} give
\begin{equation}\label{sources-L36}
\int_0^t
\left(
\|N_{\xi\xi}(\tau)\|^2+\mathcal I_{RS}(\tau)^2
\right)d\tau
\le C\delta_R.
\end{equation}

Combining
\eqref{vq-diff-L36}--\eqref{sources-L36} yields
\begin{equation}\label{vq-before-L35-L36}
\begin{aligned}
\int_0^t\|v_\xi(\tau)\|^2d\tau
\le{}&
\eta\|v_\xi(t)\|^2
+C_\eta\|u(t)\|^2
+C\int_0^t\mathbf D_1(\tau)d\tau\\
&+C\left(
\|u(0)\|_{H^1}^2+\|v(0)\|_{H^1}^2
\right)
+C\delta_R^{1/3}.
\end{aligned}
\end{equation}
Notice that the third term on the right-hand side is
$\int_0^t\mathbf D_1(\tau)d\tau$, not the pointwise quantity
$\mathbf D_1(t)$.

We now use Lemma \ref{4.First-order estimate} in
\eqref{vq-before-L35-L36}.  Together with Lemma
\ref{4.L2 estimate}, it gives
\[
\begin{aligned}
\int_0^t\|v_\xi(\tau)\|^2d\tau
\le{}&
C\left(
\|u(0)\|_{H^1}^2+\|v(0)\|_{H^1}^2
\right)
+C\delta_R^{1/3}\\
&+
C(\chi_1+\delta_R+\delta_S)
\int_0^t\|v_\xi(\tau)\|^2d\tau.
\end{aligned}
\]
Taking $\delta_0$ and $\chi_1$ sufficiently small, we absorb the last
term and conclude that
\begin{equation}\label{4.divq-q}
\int_0^t\|\partial_\xi(q-\tilde q)(\cdot,\tau)\|^2d\tau
\le
C\left(
\|n_0-\tilde n(\cdot,0)\|_{H^1}^2
+\|q_0-\tilde q(\cdot,0)\|_{H^1}^2
\right)
+C\delta_R^{1/3}.
\end{equation}

Substituting \eqref{4.divq-q} back into
\eqref{4.First order estimate}, and using the uniform bounds for
$\tilde n$, we obtain
\begin{equation}\label{4.4.62}
\begin{aligned}
&\|u_\xi(t)\|^2+\|v_\xi(t)\|^2
+\int_0^t\|u_{\xi\xi}(\tau)\|^2d\tau\\
&\le
C\left(
\|u(0)\|_{H^1}^2+\|v(0)\|_{H^1}^2
\right)
+C\delta_R^{1/3}.
\end{aligned}
\end{equation}
Combining \eqref{4.divq-q} and \eqref{4.4.62} gives
\eqref{4.H1 estimate}.
\end{proof}

\begin{proof}[Proof of Proposition \ref{4 A priori estimates}]
The a priori estimate  \eqref{4.priori estimate} follows from Lemmas \ref{4.L2 estimate} and \ref{4.H1-estimate}. Moreover, in view of  \eqref{X}, one can see that
\begin{equation*}
|\dot{\mathbf{X}}(t)|
\leq \frac{C}{\delta_S}
\|n-\tilde{n}\|_{L^{\infty}}
\int_{\mathbb{R}}
\left|(\tilde{n}^S)^{-\mathbf{X}}_{\xi}\right|
d\xi
\leq C\|n-\tilde{n}^{-\mathbf{X}}\|_{L^{\infty}},
\end{equation*}
which implies \eqref{4.dotX}.
\end{proof}

\subsubsection{Proof of Theorem \ref{composite wave theorem}}
We now show how Proposition \ref{4 A priori estimates} implies
Theorem \ref{composite wave theorem}.
\begin{proof}
Let $\delta_0,\chi_1$ and $C_0$ be the constants in Proposition
\ref{4 A priori estimates}.  After increasing the constant if necessary,
the a priori estimate \eqref{4.priori estimate} implies
\begin{equation}\label{CE-Theorem}
\sup_{0\le \tau\le T}
\left\|(n-\tilde n^{-\mathbf X},q-\tilde q^{-\mathbf X})(\cdot,\tau)
\right\|_{H^1}
\le
C_E\left(
\|(n_0-\tilde n,q_0-\tilde q)\|_{H^1}
+\delta_R^{1/6}
\right)
\end{equation}
on every interval on which the bootstrap assumption
\eqref{4.assumption} holds, where $C_E\ge1$ is independent of
$\delta_R,\delta_S$ and $T$.

\emph{Step 1. Global existence and the solution space.}
Choose smooth monotone functions
$(\underline n,\underline q)$ as in Proposition \ref{Local existence}.
They may be chosen so that, for some constant $C_*>0$,
\begin{equation}\label{n-n+-}
\begin{aligned}
&\sum_{\pm}
\left(
\|\underline n-n_\pm\|_{L^2(\mathbb R_\pm)}
+\|\underline q-q_\pm\|_{L^2(\mathbb R_\pm)}
\right)
+\|\partial_x\underline n\|
+\|\partial_x\underline q\|\\
&\qquad\le
C_*(\delta_R+\delta_S).
\end{aligned}
\end{equation}
At $t=0$, Lemma \ref{properties of nRqR} and the shock estimates
\eqref{N''} give
\begin{equation}\label{initial-profile-halves}
\begin{aligned}
&\|\tilde n^R(0)-n_-\|_{L^2(\mathbb R_-)}^2
+\|\tilde n^R(0)-n_m\|_{L^2(\mathbb R_+)}^2
\le C\delta_R^2,\\
&\|\tilde n^S-n_m\|_{L^2(\mathbb R_-)}^2
+\|\tilde n^S-n_+\|_{L^2(\mathbb R_+)}^2
\le C\delta_S,
\end{aligned}
\end{equation}
and the analogous estimates hold for the $q$-components.  Therefore,
using the triangle inequality on each half-line rather than an equality
between squared norms,
\begin{equation}\label{initial-profile-L2}
\begin{aligned}
\|\underline n-\tilde n(\cdot,0)\|^2
&\le
C\Big(
\|\underline n-n_-\|_{L^2(\mathbb R_-)}^2
+\|\underline n-n_+\|_{L^2(\mathbb R_+)}^2\\
&\qquad
+\|\tilde n^R(0)-n_-\|_{L^2(\mathbb R_-)}^2
+\|\tilde n^R(0)-n_m\|_{L^2(\mathbb R_+)}^2\\
&\qquad
+\|\tilde n^S-n_m\|_{L^2(\mathbb R_-)}^2
+\|\tilde n^S-n_+\|_{L^2(\mathbb R_+)}^2
\Big)\\
&\le C(\delta_R^2+\delta_S).
\end{aligned}
\end{equation}
Moreover,
\[
\|\partial_x\tilde n^R(0)\|
+\|\partial_x\tilde q^R(0)\|
\le C\delta_R,
\qquad
\|\tilde n^S_\xi\|
+\|\tilde q^S_\xi\|
\le C\delta_S,
\]
where the last, slightly weaker, bound is sufficient for
$0<\delta_S\le1$.  Consequently, for some $C_1>0$,
\begin{equation}\label{n-n0H1}
\|\underline n-\tilde n(\cdot,0)\|_{H^1}
+\|\underline q-\tilde q(\cdot,0)\|_{H^1}
\le
C_1(\delta_R+\sqrt{\delta_S}).
\end{equation}

We now choose the small constants uniformly.  First decrease $\delta_0$
so that
\begin{equation}\label{small-delta-theorem}
C_E\delta_0^{1/6}\le\frac{\chi_1}{8},
\qquad
2C_*\delta_0+C_1(\delta_0+\sqrt{\delta_0})
\le\frac{\chi_1}{16C_E}.
\end{equation}
Then fix
\begin{equation}\label{chi0-theorem}
\chi_0:=\frac{\chi_1}{16C_E}.
\end{equation}
In particular, $\chi_0$ depends only on the constants in the theorem and
not on the particular values of $\delta_R,\delta_S\in(0,\delta_0)$.

Let $(n_0,q_0)$ satisfy \eqref{initial data}.  From
\eqref{n-n+-},
\begin{equation}\label{n0-n}
\|n_0-\underline n\|_{H^1}
+\|q_0-\underline q\|_{H^1}
\le
\chi_0+C_*(\delta_R+\delta_S).
\end{equation}
After decreasing $\delta_0$ once more if necessary, Sobolev embedding
and \eqref{n0-n} yield the uniform positive bounds
\begin{equation}\label{n0-positive}
\frac12\min\{n_-,n_+\}
\le n_0(x)\le
2\max\{n_-,n_+\},
\qquad x\in\mathbb R.
\end{equation}
Thus Proposition \ref{Local existence} gives a unique local solution.

Furthermore, by \eqref{n-n0H1}, \eqref{n0-n}, and
\eqref{small-delta-theorem}--\eqref{chi0-theorem},
\begin{equation}\label{initial-composite-small}
\begin{aligned}
&\|n_0-\tilde n(\cdot,0)\|_{H^1}
+\|q_0-\tilde q(\cdot,0)\|_{H^1}\\
&\qquad\le
\chi_0+C_*(\delta_R+\delta_S)
+C_1(\delta_R+\sqrt{\delta_S})
\le\frac{\chi_1}{8C_E}.
\end{aligned}
\end{equation}
Since the local solution and the smooth composite profile
$(\tilde n^{-\mathbf X},\tilde q^{-\mathbf X})$ are continuous in
$H^1$ at $t=0$, there exists $T_1>0$ such that
\begin{equation}\label{initial-bootstrap-theorem}
\sup_{0\le t\le T_1}
\|(n-\tilde n^{-\mathbf X},q-\tilde q^{-\mathbf X})(\cdot,t)\|_{H^1}
<\chi_1.
\end{equation}

We now apply the standard continuation argument.  Let $T_M$ be the
supremum of times up to which the solution can be continued while
satisfying the bootstrap bound in \eqref{4.assumption}.  If $T_M<\infty$,
then for every $T<T_M$, Proposition \ref{4 A priori estimates},
\eqref{CE-Theorem}, \eqref{initial-composite-small}, and
\eqref{small-delta-theorem} give
\[
\sup_{0\le t\le T}
\|(n-\tilde n^{-\mathbf X},q-\tilde q^{-\mathbf X})(\cdot,t)\|_{H^1}
\le
C_E\left(\frac{\chi_1}{8C_E}+\delta_R^{1/6}\right)
\le\frac{\chi_1}{4}.
\]
By continuity, the same strict bound holds up to $T_M$.  Since the
composite profile is uniformly bounded away from zero, Sobolev embedding
then keeps $n$ uniformly positive.  Proposition \ref{Local existence}
therefore continues the solution beyond $T_M$, still with the norm
strictly below $\chi_1$, which contradicts the definition of $T_M$.
Hence the solution exists globally and the bootstrap assumption holds
for all time.

Letting $T\to\infty$ in Proposition \ref{4 A priori estimates}, we obtain
\begin{equation}\label{global energy}
\begin{aligned}
&\sup_{t\ge0}
\left(
\|n-\tilde n^{-\mathbf X}\|_{H^1}^2
+\|q-\tilde q^{-\mathbf X}\|_{H^1}^2
\right)
+\delta_S\int_0^\infty|\dot{\mathbf X}(t)|^2dt\\
&\quad+
\int_0^\infty
\left(
\|\partial_\xi(n-\tilde n^{-\mathbf X})(t)\|_{H^1}^2
+\|\partial_\xi(q-\tilde q^{-\mathbf X})(t)\|^2
\right)dt\\
&\le
C\|(n_0-\tilde n,q_0-\tilde q)\|_{H^1}^2
+C\delta_R^{1/3}.
\end{aligned}
\end{equation}
Moreover,
\begin{equation}\label{4.dotXt}
|\dot{\mathbf X}(t)|
\le
C\|(n-\tilde n^{-\mathbf X})(\cdot,t)\|_{L^\infty},
\qquad t>0.
\end{equation}

The estimates above are formulated with the smooth rarefaction profile.
To obtain the solution space stated in Theorem
\ref{composite wave theorem}, fix $0<t_0<T<\infty$.  By Lemma
\ref{lem2.1}-(3),
\begin{equation}\label{exact-smooth-H1}
(\tilde n^R,\tilde q^R)(\cdot,t)
-(n^r,q^r)(\cdot,t)
\in C([t_0,T];H^1(\mathbb R)).
\end{equation}
Here the restriction $t_0>0$ is essential: the derivative of the exact
self-similar rarefaction is of size $O(t^{-1})$ inside a fan of length
$O(t)$ and is not uniformly $H^1$-regular as $t\downarrow0$.
Combining \eqref{exact-smooth-H1} with the global smooth-profile
perturbation gives
\[
\begin{aligned}
&n(t,x)-\left(
n^r(x/t)+\tilde n^S(x-\sigma t-\mathbf X(t))-n_m
\right)
\in C((0,\infty);H^1(\mathbb R)),\\
&q(t,x)-\left(
q^r(x/t)+\tilde q^S(x-\sigma t-\mathbf X(t))-q_m
\right)
\in C((0,\infty);H^1(\mathbb R)).
\end{aligned}
\]
Finally, \eqref{global energy} gives
$(n-\tilde n^{-\mathbf X})_{\xi\xi}\in
L^2(0,\infty;L^2)$, while Lemma
\ref{properties of nRqR} yields
$(\tilde n^R)_{\xi\xi}\in L^2(0,\infty;L^2)$.  Hence
\[
n_{xx}(t,x)
-\tilde n^S_{xx}(x-\sigma t-\mathbf X(t))
\in L^2(0,\infty;L^2(\mathbb R)),
\]
which completes the proof of \eqref{1D solution space}.

\emph{Step 2. Large-time behavior.}
Set
\[
\phi:=n-\tilde n^{-\mathbf X},
\qquad
\psi:=q-\tilde q^{-\mathbf X},
\qquad
g(t):=\|\phi_\xi(t)\|^2+\|\psi_\xi(t)\|^2.
\]
By \eqref{global energy},
\begin{equation}\label{g-L1-theorem}
\int_0^\infty g(t)\,dt<\infty.
\end{equation}
We next prove
\begin{equation}\label{gprime-L1-theorem}
\int_0^\infty|g'(t)|\,dt<\infty.
\end{equation}

For a.e. $t>0$,
\[
\frac12g'(t)
=
\int_{\mathbb R}\phi_\xi\phi_{\xi t}\,d\xi
+\int_{\mathbb R}\psi_\xi\psi_{\xi t}\,d\xi.
\]
We first consider the $\phi$-term.  From the first equation of
\eqref{4.perturbed system},
\[
\phi_t
=
\sigma\phi_\xi
+\dot{\mathbf X}S_\xi
+(nq-\tilde n\tilde q)_\xi
+\phi_{\xi\xi}
-(F_1+F_2).
\]
Therefore, integrating by parts in $\xi$ and observing that
$\int\phi_{\xi\xi}\phi_\xi=0$,
\begin{equation}\label{phi-gprime-identity}
\begin{aligned}
\left|\int_{\mathbb R}\phi_\xi\phi_{\xi t}\,d\xi\right|
\le{}&
|\dot{\mathbf X}|\|S_\xi\|\,\|\phi_{\xi\xi}\|
+\|\phi_{\xi\xi}\|^2\\
&+\|\phi_{\xi\xi}\|
\|(nq-\tilde n\tilde q)_\xi\|
+\|\phi_{\xi\xi}\|\,\|F_1+F_2\|.
\end{aligned}
\end{equation}
Using
\[
(nq-\tilde n\tilde q)_\xi
=
\tilde n\,\psi_\xi+\tilde n_\xi\psi
+q\,\phi_\xi+\phi\,q_\xi,
\qquad
q_\xi=\psi_\xi+\tilde q_\xi,
\]
the a priori $H^1$ bound and
$\|\tilde n_\xi\|_\infty+\|\tilde q_\xi\|_\infty
\le C(\delta_R+\delta_S)$ give
\begin{equation}\label{flux-gradient-theorem}
\begin{aligned}
\|(nq-\tilde n\tilde q)_\xi\|^2
\le{}&
C\big(\|\phi_\xi\|^2+\|\psi_\xi\|^2\big)\\
&+C(\delta_R+\delta_S)
\int_{\mathbb R}
\big(|N_\xi|+|S_\xi|\big)
\big(|\phi|^2+|\psi|^2\big)\,d\xi.
\end{aligned}
\end{equation}
The weighted term is integrable in time.  Indeed, by the rarefaction
dissipation and the shock dissipation in Lemma \ref{4.L2 estimate},
\begin{equation}\label{weighted-profile-theorem}
\begin{aligned}
\int_{\mathbb R}|N_\xi|(|\phi|^2+|\psi|^2)\,d\xi
&\le C\mathcal G_R(t),\\
\int_{\mathbb R}|S_\xi|(|\phi|^2+|\psi|^2)\,d\xi
&\le C\big(U_S(t)+Z_S(t)\big),
\end{aligned}
\end{equation}
where
\[
U_S(t):=\int_{\mathbb R}|S_\xi||\phi|^2d\xi,
\qquad
Z_S(t):=\int_{\mathbb R}|S_\xi|
|\psi+\varphi(n)|^2d\xi.
\]
Furthermore,
\[
\|F_1(t)\|=\|N_{\xi\xi}(t)\|,
\qquad
\|F_2(t)\|\le C\mathcal I_{RS}(t).
\]
By \eqref{Nxx-L2time-L35} and the pointwise interaction estimate
\eqref{nSnR},
\begin{equation}\label{F1F2-theorem}
\int_0^\infty\|F_1(t)\|^2dt\le C\delta_R,
\qquad
\int_0^\infty\|F_2(t)\|^2dt
\le C\delta_R^2\delta_S.
\end{equation}
Also $\|S_\xi\|^2\le C\delta_S^3$.  Hence
\eqref{global energy}, Lemma \ref{4.L2 estimate},
\eqref{phi-gprime-identity}--\eqref{F1F2-theorem}, and Young's
inequality imply
\begin{equation}\label{4.Phix2t}
\int_0^\infty
\left|
\int_{\mathbb R}\phi_\xi\phi_{\xi t}\,d\xi
\right|dt
<\infty.
\end{equation}

We next treat the $\psi$-term.  Differentiating the second equation of
\eqref{4.perturbed system} in $\xi$ gives
\begin{equation}\label{psi-xit-correct-theorem}
\psi_{\xi t}
=
\sigma\psi_{\xi\xi}
+\dot{\mathbf X}T_{\xi\xi}
+\phi_{\xi\xi}.
\end{equation}
Therefore
\begin{equation}\label{psi-gprime-identity}
\begin{aligned}
\int_{\mathbb R}\psi_\xi\psi_{\xi t}\,d\xi
&=
\dot{\mathbf X}
\int_{\mathbb R}\psi_\xi T_{\xi\xi}\,d\xi
+\int_{\mathbb R}\psi_\xi\phi_{\xi\xi}\,d\xi,
\end{aligned}
\end{equation}
because
$\int\psi_\xi\psi_{\xi\xi}\,d\xi=0$.  By the shock profile estimates,
\[
\|T_{\xi\xi}\|^2\le C\delta_S^5.
\]
Thus Cauchy--Schwarz in time and \eqref{global energy} yield
\begin{equation}\label{psixi}
\begin{aligned}
&\int_0^\infty
\left|
\int_{\mathbb R}\psi_\xi\psi_{\xi t}\,d\xi
\right|dt\\
&\quad\le
C\delta_S^{5/2}
\left(\int_0^\infty|\dot{\mathbf X}|^2dt\right)^{1/2}
\left(\int_0^\infty\|\psi_\xi\|^2dt\right)^{1/2}\\
&\qquad
+\left(\int_0^\infty\|\psi_\xi\|^2dt\right)^{1/2}
\left(\int_0^\infty\|\phi_{\xi\xi}\|^2dt\right)^{1/2}
<\infty.
\end{aligned}
\end{equation}
Indeed, the first term is finite because
$\delta_S\int_0^\infty|\dot{\mathbf X}|^2dt<\infty$.

Combining \eqref{4.Phix2t} and \eqref{psixi} gives
\eqref{gprime-L1-theorem}.  Hence
\[
g\in W^{1,1}(0,\infty)\cap L^1(0,\infty).
\]
Thus $g(t)$ has a finite limit as $t\to\infty$, and
\eqref{g-L1-theorem} forces this limit to be zero:
\begin{equation}\label{4.limphix}
\lim_{t\to\infty}
\|(\phi_\xi,\psi_\xi)(\cdot,t)\|^2=0.
\end{equation}
The one-dimensional Sobolev inequality and the uniform $L^2$ bound from
\eqref{global energy} now imply
\begin{equation}\label{4.Sobolev's inequality}
\begin{aligned}
\sup_{\xi\in\mathbb R}|(\phi,\psi)(\xi,t)|^2
&\le
2\left(
\|\phi(t)\|\,\|\phi_\xi(t)\|
+\|\psi(t)\|\,\|\psi_\xi(t)\|
\right)
\longrightarrow0.
\end{aligned}
\end{equation}
By Lemma \ref{lem2.1}-(2), the smooth rarefaction converges uniformly to
the exact self-similar rarefaction.  Therefore
\eqref{4.Sobolev's inequality} implies the asymptotic statement
\eqref{Limsup}.

Finally, \eqref{4.dotXt} and \eqref{4.Sobolev's inequality} yield
\[
\dot{\mathbf X}(t)\longrightarrow0.
\]
Since $\mathbf X$ is absolutely continuous,
\[
\frac{\mathbf X(t)}{t}
=
\frac{\mathbf X(0)}{t}
+\frac1t\int_0^t\dot{\mathbf X}(s)\,ds.
\]
The Ces\`aro mean of a function converging to zero also converges to zero;
hence
\[
\lim_{t\to\infty}\frac{\mathbf X(t)}{t}=0.
\]
This proves \eqref{LimX'} and completes the proof of Theorem
\ref{composite wave theorem}.
\end{proof}

\section*{Acknowledgements}
This work is supported by the National Natural Science Foundation of China (No. 12371216).

\end{document}